%% file: main.tex
\documentclass[12pt]{amsart}

\usepackage{amsmath}
\usepackage{amsfonts}
\usepackage{amssymb}
\usepackage{amscd}
\usepackage{graphicx}
\usepackage[abbrev,alphabetic]{amsrefs}
\RequirePackage[dvipsnames,usenames]{color}
\usepackage{soul,xcolor}
\setstcolor{red}
\usepackage{stmaryrd}
\usepackage{mathtools}
\usepackage{booktabs}
\usepackage{multirow}
\usepackage{mathdots}
\newtagform{tiny}{\tiny(}{)}

\usepackage{mathtools}
\usepackage{hyperref}
\usepackage[margin=1.25in]{geometry}

\usepackage{amsthm}
\usepackage{comment}
\usepackage[all,cmtip]{xy}
\usepackage{tikz-cd}
\usetikzlibrary{cd}

\usepackage[all]{xy}

\makeatletter
\def\@tocline#1#2#3#4#5#6#7{\relax
  \ifnum #1>\c@tocdepth % then omit
  \else
    \par \addpenalty\@secpenalty\addvspace{#2}%
    \begingroup \hyphenpenalty\@M
    \@ifempty{#4}{%
      \@tempdima\csname r@tocindent\number#1\endcsname\relax
    }{%
      \@tempdima#4\relax
    }%
    \parindent\z@ \leftskip#3\relax \advance\leftskip\@tempdima\relax
    \rightskip\@pnumwidth plus4em \parfillskip-\@pnumwidth
    #5\leavevmode\hskip-\@tempdima
      \ifcase #1
       \or\or \hskip 1em \or \hskip 2em \else \hskip 3em \fi%
      #6\nobreak\relax
    \hfill\hbox to\@pnumwidth{\@tocpagenum{#7}}\par% <---- \dotfill -> \hfill
    \nobreak
    \endgroup
  \fi}
\makeatother

\newsavebox{\pullback}
\sbox\pullback{%
\begin{tikzpicture}%
\draw (0,0) -- (1ex,0ex);%
\draw (1ex,0ex) -- (1ex,1ex);%
\end{tikzpicture}}

\newsavebox{\pullbackdl}
\sbox\pullbackdl{%
\begin{tikzpicture}%
\draw (-1ex,0ex) -- (0ex,0ex);%
\draw (0ex,-1ex) -- (0ex,0ex);%
\end{tikzpicture}}

\newcommand{\rup}[1]{\lceil #1 \rceil}
\newcommand{\rdown}[1]{\lfloor #1 \rfloor}

\newcommand{\Z}{\mathbb{Z}}
\newcommand{\Q}{\mathbb{Q}}

\newcommand{\cHom}{\mathop{\mathcal{H}\! \mathit{om}}}

\newcommand{\bA}{\mathbb{A}}

\newcommand{\bQ}{\mathbb{Q}}

\newcommand{\bZ}{\mathbb{Z}}
\newcommand{\bolda}{\boldsymbol{a}}
\newcommand{\boldb}{\boldsymbol{b}}
\newcommand{\boldx}{\boldsymbol{x}}

\newcommand{\cA}{\mathcal{A}}

\newcommand{\cF}{\mathcal{F}}
\newcommand{\cG}{\mathcal{G}}
\newcommand{\cH}{\mathcal{H}}

\newcommand{\cO}{\mathcal{O}}
\newcommand{\cT}{\mathcal{T}}

\newcommand{\sO}{\mathcal{O}}

\newcommand{\m}{\mathfrak{m}}

\newcommand{\perf}{\mathrm{perf}}

\DeclareMathOperator{\Supp}{Supp}
\DeclareMathOperator{\Spec}{Spec}

\DeclareMathOperator{\Hom}{Hom}
\DeclareMathOperator{\Ext}{Ext}

\DeclareMathOperator{\Exc}{Exc}

\DeclareMathOperator{\eh}{eh}

\theoremstyle{plain}
\newtheorem{theorem}{Theorem}[section]

\newtheorem{proposition}[theorem]{Proposition}
\newtheorem{lemma}[theorem]{Lemma}
\newtheorem{corollary}[theorem]{Corollary}
\newtheorem{conjecture}[theorem]{Conjecture}

\newtheorem{claim}[theorem]{Claim}
\newtheorem*{claim*}{Claim}

\newtheorem{theoremA}{Theorem}

\theoremstyle{definition}
\newtheorem{definition}[theorem]{Definition}
\newtheorem{setting}[theorem]{Setting}
\newtheorem{construction}[theorem]{Construction}

\newtheorem*{setup*}{Setup}

\theoremstyle{remark}
\newtheorem{remark}[theorem]{Remark}



\numberwithin{equation}{theorem}
\title[Higher F-injective singularities]
{Higher F-injective singularities}
\author{Tatsuro Kawakami}
\address{Department of Mathematics, Graduate School of Science, Kyoto University, Kyoto 606-8502, Japan} 
\email{tatsurokawakami0@gmail.com}
\author{Jakub Witaszek} 
\address{Department of Mathematics \\
Fine Hall, Washington Road\\
Princeton University\\  
Princeton NJ 08544-1000, USA}
\email{jwitaszek@princeton.edu}

\begin{document}

\begin{abstract}
We introduce the concept of higher $F$-injectivity, a generalisation of $F$-injectivity. We prove that an isolated singularity over a field of characteristic zero is $k$-Du Bois if it is $k$-$F$-injective after reductions modulo infinitely many primes $p$. Under the ordinarity conjecture, we also establish the converse. As an application, we study Frobenius liftable hypersurfaces.
\end{abstract}

\subjclass[2010]{}   
\maketitle

\setcounter{tocdepth}{2}
\tableofcontents
\section{Introduction}

In recent years, considerable attention has been drawn to the study of complex singularities through the lens of Hodge theory. In this context, one of the key objects of interest is the Deligne--Du Bois complex $\underline{\Omega}^\bullet_X$ acting as the replacement of the de Rham complex $\Omega^{\bullet}_X$ when $X$ is not smooth.  Not only is this filtered complex fundamental in the study of mixed Hodge structures, but it also allows for generalisations of various statements in complex birational geometry to the singular setting\footnote{Such as Akizuki--Nakano vanishing for projective varieties: $\mathbb{H}^j(X, \underline{\Omega}^i_X \otimes \cA) = 0$ for $i+j>{\rm dim}(X)$, where $\underline{\Omega}^i_X$ are the graded pieces of 
$\underline{\Omega}^\bullet_X$ and $\cA$ is an ample line bundle. We emphasise that $\underline{\Omega}^i_X \in D^b_{\rm coh}(X)$, in contrast to $\Omega^i_X$, is not a sheaf, but a complex itself.}.

In view of the above, it is natural to wonder when $\underline{\Omega}^i_X$ and $\Omega^i_X$ agree. In this spirit, one says that a complex hypersurface singularity:
\begin{quote} 
$X$ is \emph{$k$-Du Bois} if and only if  $\underline{\Omega}^i_X \cong \Omega^i_X$ for all integers $0 \leq i \leq k$.
\end{quote}
The $0$-Du Bois singularities, satisfying $\underline{\cO}_X \cong \cO_X$, are known as \emph{Du Bois} singularities, and they play a crucial role in birational geometry and moduli theory. To shed more light on the above notion, let us point out that recently \cite{MOPW23} and \cite{JKSY22} proved that $k$-Du Bois and $k$-log-canonical singularities coincide for hypersurfaces, where the latter is another significant notion of singularities defined via Hodge modules in \cite{MP19} (see also \cite{Mustata-Popa22}). Higher Du Bois singularities also feature fundamentally in the recent work of Friedman--Laza \cites{Friedman-Laza1,Friedman-Laza2}, as well as that of Shen--Venkatesh--Vo \cite{SVV}. Finally, note that $k$-Du Bois singularities can also be defined for arbitrary, non-hypersurface, singularities; we refer to Subsection \ref{ss:prekFinj} for a more thorough discussion.

It has been speculated for some time what the correct analogue of Hodge theory for singularities $X = \Spec R$ should be in positive characteristic. In the case of the usual Du Bois singularities, the relevant notion is that of \emph{$F$-injective} rings -- those for which the Frobenius action $F : H^i_\m(R) \to H^i_\m(F_*R)$ on local cohomology is injective for every maximal ideal $\m$ and $0 \leq i \leq \dim R$. By combining the work of Schwede \cite{Schwede} and Bhatt--Schwede--Takagi \cite{Bhatt-Schwede-Takagi}, we know that a singularity in characteristic zero is Du Bois if and only if it is $F$-injective for reductions modulo infinitely many prime numbers $p > 0$, with the caveat that the implication from right to left is contingent upon the weak ordinarity conjecture.\\

The goal of our article is to introduce a positive characteristic analogue, which we call \emph{$k$-$F$-injectivity}, of the notion of $k$-Du Bois singularities and to generalise the work of Schwede and Bhatt--Schwede--Takagi to this new class of singularities. We believe that this is a significant step towards gaining a better understanding of Hodge theory of singularities and possibly, in the long run, that of Hodge ideals, Hodge modules, and Riemann-Hilbert correspondences in positive characteristic. From now on, we assume for simplicity that $\Omega^i_X$ is reflexive for every integer $0 \leq i \leq k$. This is a natural assumption in the context of hypersurfaces or complete intersections, and we provide most general definitions in Subsection \ref{ss:prekFinj}. 
\begin{definition} \label{def:intro-kFinj}
Let $X=\Spec\,R$ be a normal affine scheme of finite type over a perfect field of positive characteristic.
Fix an integer $k>0$ and assume that $\Omega^i_R$ is reflexive for every integer $0 \leq i \leq k$. We say that $X$ is \emph{$k$-$F$-injective} if and only if
\[
C^{-1} \colon H^j_\m(\Omega^{[i]}_R) \to H^j_{\m}\left(\frac{F_*\Omega^{[i]}_R}{B\Omega^{[i]}_R}\right)
\]
is injective\footnote{A priori, it is not clear if one should use $\frac{F_*\Omega^{[i]}_R}{B\Omega^{[i]}_R}$ or the reflexivisation $\left(\frac{F_*\Omega^{i}_R}{B\Omega^{i}_R}\right)^{**}$; however, we believe that the former is more natural from the viewpoint of iterating the Cartier operator as in Subsection \ref{ss:proof-intro}.} for all integers $0 \leq i \leq k$, $0 \leq j \leq \dim\,X - i$, and all maximal ideals $\m$. See Remark \ref{remark:C-1} below for the definition of $C^{-1}$, $\Omega^{[i]}_R$, and $B\Omega^{[i]}_R$.
\end{definition}
Since $B\Omega^0_R=0$, this agrees with $F$-injectivity when $k=0$. The above definition is motivated by the first author's work on the logarithmic extension theorem for differential forms \cite{Kaw4}. 
\begin{remark} \label{remark:C-1} 
First, let us assume that $R$ is a regular ring. Then by the result of Deligne and Illusie, we know that $\cH^i(F_*\Omega^\bullet_R) \cong \Omega^i_R$. This isomorphism yields the following exact sequence:
\[
0 \to B\Omega^i_R \to Z\Omega^i_R \xrightarrow{{C}} \Omega^i_R \to 0,
\]
where $B\Omega^i_R$ and $Z\Omega^i_R$ are the boundaries and cycles of the de Rham complex $F_*\Omega^\bullet_R$, respectively. 

In the non-regular case, we get the following sequence:
\[
0 \to B\Omega^{[i]}_R \to Z\Omega^{[i]}_R \xrightarrow{C} \Omega^{[i]}_R,
\]
where $B\Omega^{[i]}_R$, $Z\Omega^{[i]}_R$, and $\Omega^{[i]}_R$ are the reflexivisations of $B\Omega^i_R$, $Z\Omega^i_R$, and $\Omega^i_R$, respectively. When $\Omega^i_R$ are reflexive, the map $C \colon Z\Omega^{[i]}_R \to \Omega^{[i]}_R$ turns out to be surjective (see Proposition \ref{prop:when-omega-reflexive-basic-properties}) and we define $C^{-1}$ as the composition:
\[
C^{-1} \colon \Omega^{[i]}_R \xleftarrow{\cong} \frac{Z\Omega^{[i]}_R}{B\Omega^{[i]}_R} \hookrightarrow \frac{F_*\Omega^{[i]}_R}{B\Omega^{[i]}_R}.
\]
Of course, under our assumptions, $\Omega^{[i]}_R = \Omega^i_R$, but it is not true in general that $B\Omega^{[i]}_R = B\Omega^i_R$, and so we use the superscript $[i]$ consistently in the definition.
\end{remark}

Our main theorem is the following. Note that the most general variant of this result does not require $\Omega^i_X$ to be reflexive (see Theorems \ref{thm:k-F-inj to k-Du Bois} and \ref{thm:pre-k-F-inj type to pre-$k$-Du Bois}). We will discuss the case of non-isolated singularities in an upcoming article.
\begin{theoremA}[Theorems \ref{thm:k-F-inj to k-Du Bois}, \ref{thm:pre-k-F-inj type to pre-$k$-Du Bois} and Corollaries \ref{cor:k-F-inj to k-Du Bois}, \ref{cor:k-Du Bois to k-F-inj}] \label{thm:intro-main}
Let $X$ be a normal affine variety over a field of characteristic zero {with isolated singularities.}
Fix an integer $k>0$ and assume that $\Omega^i_X$ is reflexive for every integer $0 \leq i \leq k$. 

Then $X$ is $k$-Du Bois if and only if $X_{p=0}$ is $k$-$F$-injective for infinitely many primes $p>0$. Here, the implication from left to right is contingent upon the ordinarity conjecture.
\end{theoremA}
\noindent In the above, $X_{p=0}$ denotes the reduction modulo $p$ for some spread-out of $X$ over a finitely generated $\bZ$-subalgebra of $k$. The ordinarity conjecture states the following.

\begin{conjecture} \label{conj:ordinarity} Let $X$ be a smooth projective scheme defined over a field of characteristic zero. Then, after spreading-out, for infinitely many primes $p$, the reduction $X_{p=0}$ is ordinary in the sense of Bloch-Kato, that is, $H^i(X_{p=0}, B\Omega^j_{X_{p=0}}) = 0$ for all $i,j \geq 0$.
\end{conjecture}
Note that we are not assuming that $X$ is connected. If $X$ is connected, then it is \emph{$k$-ordinary} if and only if the first $k$ slopes of the Newton polygon agree with those of the Hodge polygon. In particular, $0$-ordinarity, also known as \emph{weak ordinarity}, requires the Frobenius action $F \colon H^{d}(\cO_X) \to H^{d}(\cO_X)$ to be bijective, where $d$ is the dimension of $X$. Note that $F$-injectivity can be viewed as a commutative-algebraic generalisation of weak ordinarity: for instance, a cone over a strict Calabi-Yau variety $X$ is $F$-injective if and only if $X$ is weakly ordinary. Extending this analogy, $k$-$F$-injectivity may be regarded as a commutative-algebraic generalisation of $k$-ordinarity.\\

Finally, we discuss applications of the above theory to the study of Frobenius liftable singularities. In what follows, we say that a variety $X$ over a perfect field $k$ of positive characteristic is \emph{Frobenius liftable modulo $p^2$} if and only if there exists a flat lift $\widetilde{X}$ of $X$ over $W_2(k)$ and an endomorphism $\widetilde{F} \colon \widetilde{X} \to \widetilde{X}$ which reduces to $F$ modulo $p$. In general, we expect that admitting a Frobenius lift imposes strong conditions on the singularities and global geometry of $X$. For example, it is speculated that a Frobenius liftable singularity is automatically a quotient of a toric singularity (cf.\ Bourger \cite{borger2009lambdaringsfieldelement} and \cite{AWZ}). Using the theory developed in this paper, we prove the following.
\begin{theoremA}[{Theorem \ref{thm:Frobenius-application}}] \label{thm:Frobenius-application-intro}
Let $X \subseteq \bA^n$ be a rational isolated hypersurface singularity defined over a field of characteristic zero.  Suppose that, after spreading out, $X_{p=0}$ is Frobenius liftable for infinitely many primes $p>0$. Then the following statements hold.
\begin{enumerate}
\item If {$\dim\,X\geq 4$}, then $X$ is smooth.
\item If {$\dim\,X=3$}, then $X$ is analytically isomorphic to {$k\llbracket x,y,z,w \rrbracket/ (xy-zw)$}.
\end{enumerate}
\end{theoremA}
In particular, (2) shows that three-dimensional Gorenstein terminal singularities which are Frobenius liftable for reductions modulo infinitely many primes $p>0$ are either smooth or are cones over a quadric (see Corollary \ref{cor:Frobenius-applications}). In an upcoming article, we extend this result to non-isolated singularities showing that if ${\rm codim}_X {\rm Sing}(X) \geq 4$ and $X_{p=0}$ is Frobenius liftable for infinitely many $p>0$, then $X$ is smooth.

Note that Frobenius liftability of complete intersections have been extensively studied in \cite{Zda18} (see also \cite{AWZ}, \cite{AWZ2}, \cite{JS23}, and \cite{Kawakami-Takamatsu}). The key components of the proof of the above theorem is the study of pre-$k$-$F$-injectivity, discussed below, and \cite[Theorem 1.5]{MOPW23} which appeals to deep properties of the $V$-filtration on Hodge modules. For this reason, we do not know if the above result holds for a Frobenius liftable hypersurface singularity in a fixed characteristic $p>0$.

\subsection{A few words on the proof of Theorem \ref{thm:intro-main}} \label{ss:proof-intro}
In what follows, we sketch the idea of the proof of the implication in Theorem \ref{thm:intro-main} that if a reduction $X$ modulo some large enough $p>0$ of a complex singularity $\mathcal{X}$ is $k$-$F$-injective, then $\mathcal{X}$ is $k$-Du Bois. To simplify the explanation below, we shall assume the existence of strong resolution of singularities (see Condition $R(d)$ in Definition \ref{def:strong resolutions}), so that the Deligne--Du Bois complex $\underline{\Omega}^\bullet_X$ can be defined in positive characteristic as in characteristic zero. 

Suppose that $X$ is $k$-$F$-injective. To prove the theorem we would like to iterate the inverse Cartier operator $C^{-1} \colon \Omega^i_X \to \frac{F_*\Omega^i_X}{B\Omega^i_X}$. To this end, we construct 
\[
C^{-1}_n \colon \Omega^i_X \to \frac{F^n_*\Omega^i_X}{B_n\Omega^{[i]}_X} \qquad \text{ and the colimit} \qquad  \Omega^{i,{\rm perf}}_{X} \coloneqq \varinjlim_n \frac{F^n_*\Omega^i_X}{B_n\Omega^{[i]}_X},
\] 
where $B_n\Omega^{[i]}_X$ are reflexivisations of, so called, higher boundaries (see \ref{sss:higher-Cartier}). 
As in the case of usual $F$-injectivity, we deduce from $k$-$F$-injectivity of $X$ that
\[
H^j_\m(\Omega^i_X) \to H^j_{\m}\left(\Omega^{i,{\rm perf}}_{X}\right)
\]
is injective for all $0 \leq i \leq k$ and $0 \leq j \leq \dim(X)-i$. Moreover, using Serre's vanishing, we show that we have a factorisation
\[
\Omega^i_X \to \underline{\Omega}^i_X \to \Omega^{i,{\rm perf}}_{X}.
\]
In particular, we get that
\[
H^j_\m(\Omega^i_X) \to \mathbb{H}^j_{\m}\left(\underline{\Omega}^i_{X}\right)
\]
is injective for all $0 \leq i \leq k$ and $0 \leq j \leq \dim(X)-i$. We call varieties satisfying the above statement \emph{pseudo-$k$-Du Bois}. It turns out that this notion of singularities agrees with $k$-Du Bois in characteristic zero.
\begin{theoremA}[{see Theorem \ref{thm:pseudo-pre-k-Du Bois=pre-k-Du Bois in char 0} for a stronger version}]
    Let $\mathcal{X}$ be a normal variety over a field of characteristic zero {with isolated singularities}.
    Assume that $\Omega^i_\mathcal{X}$ is reflexive for every integer $0 \leq i \leq k$.
    Then $\mathcal{X}$ is $k$-Du Bois if and only if $\mathcal{X}$ is pseudo-$k$-Du Bois.
\end{theoremA}

To sum up the strategy of proof of Theorem \ref{thm:intro-main}, from the reduction $X$ being $k$-$F$-injective, we deduce that $\mathcal{X}$ is pseudo-$k$-Du Bois and hence $k$-Du Bois by the above discussion, concluding the proof. Interestingly, the above proof suggests that $\Omega^{i,{\rm perf}}_{X}$ plays the role of $\underline{\Omega}^i_X$ in positive characteristic. This should be compared with the fact that $R_{\rm perf}$ plays the role of $\underline{\cO}_R$ in characteristic $p>0$ as in the work of Bhatt--Schwede--Takagi \cite{Bhatt-Schwede-Takagi}.

\subsection{Pre-$k$-$F$-injective singularities} \label{ss:prekFinj}

Since $\Omega^i_X$ is rarely reflexive for singular varieties beyond the complete intersection case, it is not clear how to define the $k$-Du Bois property in general (see \cite[Definition 1.2]{SVV} for one candidate for such a definition). However, the following weaker definition from op.cit.\ is of substantial interest for arbitrary varieties; we say that a characteristic zero variety:
\begin{quote}
$X$ is \emph{pre-$k$-Du Bois} if $\underline{\Omega}^i_X \cong \cH^0(\underline{\Omega}^i_X)[0]$ for all integers $0 \leq i \leq k$.
\end{quote}
Here $\cH^0(\underline{\Omega}^i_X)$ is nothing but the sheafification $\Omega^i_{X, {\rm \eh}}$ of $\Omega^i_X$ in the eh-topology as studied in \cite{Huber-Kebekus-Kelly}. We can then define.
\begin{definition} \label{def:intro-pre-kFinj}
We say that $X = \Spec R$ is \emph{pre-$k$-$F$-injective} if 
\begin{enumerate}
\item $C \colon Z\Omega^i_{R, \eh} \to \Omega^i_{R, \eh}$ is surjective for all integers $0 \leq i \leq k$, and
\item $C^{-1} \colon H^j_\m(\Omega^i_{R, \eh}) \to H^j_{\m}\left(\frac{F_*\Omega^i_{R, \eh}}{B\Omega^i_{R, \eh}}\right)$ is injective for all integers $0 \leq i \leq k$, $0 \leq j \leq \dim(R) - i$, and all maximal ideals $\m$.
\end{enumerate}
\end{definition}
\noindent Here, one needs $(1)$ to even define the map $C^{-1}$ from $(2)$. Note that (1) is automatic when $\Omega^i_X$ are reflexive (see Proposition \ref{prop:when-omega-reflexive-basic-properties}).

However, instead of working directly with Definition \ref{def:intro-pre-kFinj}, in this paper we define the notions of \emph{pre-$k$-$F$-injectivity along a fixed resolution $\pi \colon Y \to \Spec R$} (see Definition \ref{def:prekFinj-iso}).
Using \cite{Huber-Kebekus-Kelly}, one can check that the definition is independent of the choice of $\pi$ in the case of isolated singularities and in fact agrees with Definition \ref{def:intro-pre-kFinj}. 

\subsection{Non-isolated singularities}
Most of the results in this paper extend to the case of non-isolated singularities. However, due to additional subtleties and difficulties, we have chosen to address non-isolated singularities in a separate article.

\subsection{Acknowlegements}
The authors thank Bhargav Bhatt, Brad Dirks, Eamon Quinlan-Gallego, Mircea Musta{\c{t}}{\u{a}}, Mihnea Popa, Supravat Sarkar, Wanchun Shen, Karl Schwede,  Teppei Takamatsu, Sridhar Venkatesh, Duc Vo, Shou Yoshikawa, and Bogdan Zavyalov for valuable conversations related to the content of the paper.
Kawakami was supported by JSPS KAKENHI Grant number JP22KJ1771 and JP24K16897.
Witaszek was supported by NSF research grants DMS-2101897 and DMS-2401360.

\section{Preliminaries}
\subsection{Notation and terminology}
Throughout this paper, we use the following notation:
\begin{itemize}
\item A \textit{variety} is an integral separated scheme of finite type over a field. 
\item Given a proper birational morphism $\pi\colon Y\to X$, we denote by $\Exc(\pi)$ the reduced exceptional divisor. 
\item Given an integral normal Noetherian scheme $X$, a proper birational morphism $\pi \colon Y \to X$ is called \emph{a log resolution of $X$} if $Y$ is regular and $\mathrm{Exc}(\pi)$ is a simple normal crossing divisor.
\item For a Noetherian scheme $X$, we denote by $D^b_{\mathrm{coh}}(X)$ the derived category consisting of bounded complexes of coherent $\cO_X$-modules.
\item For  $\mathcal{F}^{\bullet}\in D^b_{\mathrm{coh}}(X)$, we denote the canonical truncation of $\mathcal{F}^{\bullet}$ by $\tau^{>j}\mathcal{F}^{\bullet}$ (cf.\ \cite[\href{https://stacks.math.columbia.edu/tag/0118}{Tag 0118}]{stacks-project})
\item For a proper birational morphism $\pi\colon Y\to X$ of Noetherian schemes,
we denote $\tau^{>j}R\pi_{*}\mathcal{F}$ by $R^{>j}\pi_{*}\mathcal{F}$.
\item 
Given an integral normal Noetherian scheme $X$ and a $\bQ$-divisor $D$, 
we define the subsheaf $\sO_X(D)$ of sheaf $K(X)$ of rational functions on $X$ 
by the following formula
\[
\Gamma(U, \sO_X(D)) = 
\{ \varphi \in K(X) \mid 
\left({\rm div}(\varphi)+D\right)|_U \geq 0\}
\]
{for every} open subset $U$ of $X$. 
In particular, 
$\cO_X(\rdown{{D}}) = \cO_X({D})$.
\end{itemize}
\begin{definition}
    Let $X=\Spec\,R$ be an affine scheme of finite type over a perfect field of positive characteristic.
    We say that $X$ is \emph{$F$-injective} if the Frobenius map
    \[
    F\colon H^i_{\m}(R) \to H^i_{\m}(R)
    \]
    is injective for all $j\geq 0$ and for all maximal ideals $\m\subset R$.
\end{definition}

\begin{definition}\label{def:strong resolutions}
    Let $k$ be a perfect field.
    For a positive integer $d>0$, we denote by $R(d)$ the condition on the existence of strong resolutions for varieties up to dimension $d$, that is, that the following statements hold:
    \begin{itemize}
        \item for every integral separated scheme $X$ of finite type over $k$ and dimension at most $d$, there is a proper, birational
        morphism $f\colon Y\to X$ over $k$ such that $Y$ is smooth.
        \item for every smooth scheme $X$ over $k$ of dimension at most $d$ and every proper birational morphism $f\colon Y\to X$, there is a sequence $X_{n}\to X_{n-1}\to \cdots \to X_{0}=X$ of blowups along smooth centers such that the composition $X_{n}\to X$ factors through $f$.
    \end{itemize}
\end{definition}
Note that we do not want to assume the condition $R(d)$ throughout the paper as it is may not be true that a reduction of a characteristic zero variety modulo $p \gg 0$ satisfies $R(d)$. However, such a reduction at least admits a resolution of singularities and that will be enough for our applications.

\subsection{Derived categories}

Let $\mathcal{T}$ be a triangulated category with a $t$-structure and let $\cF, \cG, \cH \in \mathcal{T}^{\geq 0}$ be such that they fit in an exact triangle
\[
\cF \to \cG \to \cH \xrightarrow{+1}.
\]
Then we have the following diagram:
\begin{equation} \label{eq:fgh-derived}
\begin{tikzcd}
\mathcal{H}^0(\mathcal{F}) \arrow[r] \arrow[d] & \mathcal{H}^0(\mathcal{G}) \arrow[r] \arrow[d] & \mathcal{H}^0(\mathcal{G})/\mathcal{H}^0(\mathcal{F}) \arrow[r, "+1"] \arrow[d] & {} \\
\mathcal{F} \arrow[r] \arrow[d] & \mathcal{G} \arrow[r] \arrow[d] & \mathcal{H} \arrow[r, "+1"] \arrow[d] & {} \\
\tau^{>0}\mathcal{F} \arrow[r] \arrow[d, "+1"] & \tau^{>0}\mathcal{G} \arrow[r] \arrow[d, "+1"] & \text{Cone}(\tau^{>0}\mathcal{F} \to \tau^{>0}\mathcal{G}) \arrow[r, "+1"] \arrow[d, "+1"] & {} \\
{} & {} & {} & {}
\end{tikzcd}
\end{equation}

\begin{lemma} \label{lem:derived-cat-factor}
Let $\mathcal{T}$ be a triangulated category with a $t$-structure and let
\[\mathcal{K} \to \cG \to \cH \xrightarrow{+1}\] be an exact triangle in $\cT$.
Let $f \colon \cF \to \cG$ be a map in $\cT$ such that the composition
\[
\cF \xrightarrow{f} \cG \to \cH
\]
is a zero map. Then $f$ factorises as follows:
\[
f \colon \cF \to \mathcal{K} \to \cG.
\]
In particular, when we take $\cG \to \cH$ to be the natural map $\cG \to \tau^{>0}\cG$, we get the following factorisation of $f$:
\[
f \colon \cF \to H^0(\cG) \xrightarrow{\rm nat} \cG.
\]
\end{lemma}
\begin{proof}
By applying  $\Hom_{\cT}(\cF, -)$ 
to the exact triangle (cf.\ the proof of \cite[Theorem 0.3]{Bhatt_2012}):
\[
\mathcal{K} \to \cG \to \cH\xrightarrow{+1},
\]
we get an exact sequence of abelian groups:
\[
\Hom_{\cT}(\cF, \mathcal{K}) \to 
\Hom_{\cT}(\cF, \cG) \to 
\Hom_{\cT}(\cF, \cH). 
\]
This immediately implies the statement of the lemma.
\end{proof}

\subsection{Cartier operators}
We briefly review the theory of Cartier operators. Let $X$ be a separated scheme of finite type over a perfect field $k$ of characteristic $p>0$. Then $F_*\Omega^\bullet_X$ is a complex of coherent sheaves with $\cO_X$-linear maps. We define
\begin{align*}
B\Omega^i_X &\coloneqq {\rm im}(d \colon F_*\Omega^{i-1}_X \to F_*\Omega^i_X), \text{ and }\\
Z\Omega^i_X &\coloneqq \ker(d \colon F_*\Omega^{i}_X \to F_*\Omega^{i+1}_X).
\end{align*}
The following map
\begin{equation} \label{eq:inverse-Cartier-basic}
C^{-1} \colon \Omega^i_X \to \frac{Z\Omega^i_X}{B\Omega^i_X},
\end{equation}
called the \emph{inverse Cartier operator} is constructed as follows (see \cite[Section 1.3]{fbook}). As definition is local, we assume that $X = \Spec A$ for a ring $A$ of finite type over $k$, and consider the function $\gamma \colon A \to \Omega^1_A$ given by $\gamma(a) \coloneqq a^{p-1}da$ for $a \in A$. Since 
\begin{enumerate}
\item $\gamma(ab) = a^p\gamma(b) + \gamma(a)b^p$,
\item $d\gamma(a)=0$, and
\item $\gamma(a+b) - \gamma(a) - \gamma(b) \in B\Omega^1_A$,
\end{enumerate}
for $a,b\in A$ (see \cite[Lemma 1.3.3]{fbook}), we get that $\gamma$ induces a map
\begin{align*}
\gamma \colon \Omega^1_A &\to \frac{Z\Omega^{1}_A}{B\Omega^{1}_A}\\
adb &\mapsto a^pb^{p-1}db
\end{align*}
by the universal property of K\"ahler differentials. Given that $\gamma(da) \wedge \gamma(da) = 0$, taking exterior derivatives yields $C^{-1}$ as in \eqref{eq:inverse-Cartier-basic}. When $X$ is regular, $C^{-1} \colon \Omega^i_X \to \frac{Z\Omega^i_X}{B\Omega^i_X}$ is an isomorphism (see \cite[Theorem 1.3.4]{fbook}), and so we have a short exact sequence:
\[
0 \to B\Omega^i_X \to Z\Omega^i_X \xrightarrow{C} \Omega^i_X \to 0,
\]
where the rightmost nontrivial arrow is denoted by $C$ and called the \emph{Cartier operator}. Below in Proposition \ref{prop:when-omega-reflexive-basic-properties}, we will show that the same is true when the differential forms are reflexive even when $X$ is not regular.

From now on, assume that $X$ is normal and integral. Define $\Omega^{[i]}_X$ as the reflexivisations of $\Omega^{i}_X$.
Define $Z\Omega^{[i]}_X$ and $B\Omega^{[i]}_X$ as follows:
\[Z\Omega^{[i]}_X\coloneqq \ker(F_{*}\Omega^{[i]}_X\xrightarrow{d} \Omega^{[i+1]}_X)\quad\text{and}\quad
 B\Omega^{[i]}_X\coloneqq \ker(Z\Omega^{[i]}_X\xrightarrow{C} \Omega^{[i]}_X),\]
where the maps $d$ and $C$ are defined by pushing forward from the regular locus. Since the kernel of a map of reflexive sheaves is reflexive, $Z\Omega^{[i]}_X$ and $B\Omega^{[i]}_X$ are reflexive.
Note that the equality $B\Omega^{[i]}_X= \mathrm{im}(F_{*}\Omega^{[i-1]}_X\xrightarrow{d} \Omega^{[i]}_X)$ need not hold in general.

\begin{proposition} \label{prop:when-omega-reflexive-basic-properties}
Fix an integer $k \geq 0$ and let $X$ be a normal variety over a perfect field $k$ of characteristic $p>0$. Moreover assume that $\Omega^i_X$ is reflexive for all $0 \leq i \leq k$. Then the following statements hold true for $i \leq k$:
\begin{enumerate}
\item $\Omega^i_X = \Omega^{[i]}_X$;
\item  $Z\Omega^i_X = Z\Omega^{[i]}_X$ {when $i\leq k-1$ and $Z\Omega_X^{k}$ is torsion-free};
\item $B\Omega^{i}_X \subseteq B\Omega^{[i]}_X$ is torsion-free {when $i\leq {k}$};
\item the reflevisation $C \colon Z\Omega^{[i]}_X \to \Omega^{[i]}_X$ of the Cartier map on the regular locus is surjective, and so there exists an isomorphism $$C^{-1} \colon\Omega^{[i]}_X \xrightarrow{\cong} Z\Omega^{[i]}_X / B\Omega^{[i]}_X.$$
\item the natural map $Z\Omega^i_X / B\Omega^i_X \to Z\Omega^{[i]}_X / B\Omega^{[i]}_X \cong \Omega^{[i]}_X$ is split surjective with kernel being torsion.
\end{enumerate}
\end{proposition}
\noindent If $\Omega^{k+1}_X$ is torsion-free, then the proof below also shows that  $B\Omega_X^k$ is torsion-free and $Z\Omega^k_X$ is reflexive.
\begin{proof}
(1) and (2) hold by assumption.
Since $B\Omega^i_X={\rm coker}(Z\Omega^{i-1}_X \hookrightarrow F_*\Omega^{i-1}_X)$, it is torsion-free, and the natural map $B\Omega^i_X\hookrightarrow B\Omega^{[i]}_X$ to its reflexivisation is injective when $i\leq k$. Thus (3) holds. As for (4) and (5), consider the following diagram
\[
\begin{tikzcd}
0 \arrow[r] & B\Omega^{[i]}_X \arrow[r]  & Z\Omega^{[i]}_X \arrow[r]  & \Omega^{[i]}_X \\
0 \arrow[r] & B\Omega^i_X \arrow[u, ""] \arrow[r] & Z\Omega^i_X \arrow[r]\arrow[u] & Z\Omega^i_X / B\Omega^i_X \arrow[r] \arrow[u, "\tau"] & 0
\end{tikzcd}
\]
in which the upper row comes from pushing forward the Cartier operator from the smooth locus. All the vertical maps are identities on the smooth locus. Now look at the map  $C^{-1} \colon \Omega^i_X \to Z\Omega^i_X / B\Omega^i_X$. 
The composition
\[
\Omega^i_X \xrightarrow{C^{-1}} Z\Omega^i_X / B\Omega^i_X \xrightarrow{\tau} \Omega^{[i]}_X, 
\]
is an isomorphism, as it is so on the smooth locus and $\Omega^i_X$ is reflexive. Therefore, 
\[
Z\Omega^i_X / B\Omega^i_X \cong \Omega^i_X \oplus \text{torsion}.
\]
This immediately implies (4) and (5).
\end{proof}

\begin{definition}
    Let $X$ be a normal integral separated scheme of finite type over a perfect field $k$ of characteristic $p>0$.
    We define a coherent $\sO_X$-module $G\Omega^{[i]}_X := F_*\Omega^{[i]}_X / B\Omega^{[i]}_X$.
    Since $B\Omega^{[i]}_X$ and $F_*\Omega^{[i]}_X$ are reflexive, $G\Omega^{[i]}_X$ is torsion-free.
\end{definition}

\subsubsection{Higher Cartier operator}  \label{sss:higher-Cartier}
In what follows we review iterated Cartier operators. We refer to \cite[Section 2.4]{KTTWYY1} for a more detailed recollection. As before, $X$ is a smooth scheme of finite type over a perfect field $k$ of characteristic $p>0$.

Note that the Cartier operator may be though of as a partial map\footnote{a \emph{partial map} of coherent sheaves $f \colon \cF \rightharpoonup \cG$ consists of a data of a subsheaf $\cH \subseteq \cF$ and an $\cO_X$-linear map $f_{\cH} \colon \cH \rightharpoonup \cG$.} $C \colon F_*\Omega_X^i \rightharpoonup \Omega_X^i$. In this case $B\Omega^i_X$ is the kernel of $C$ and $Z\Omega^i_X$ is the domain of $C$. By composing this partial map $n$ times we get:
\[
C_{n} \colon F^n_*\Omega_X^i \rightharpoonup \Omega_X^i.
\]
We define $B_n\Omega_X^i$ and $Z_n\Omega_X^i$ to be the kernel and domain of $C_{n}$, respectively, so that we get a short exact sequence
\[
0 \rightarrow B_n\Omega_X^i \rightarrow Z_n\Omega_X^i \xrightarrow{C_{n}} \Omega_X^i \rightarrow 0.
\]

\subsubsection{Hara's Cartier operator} \label{sss:Hara}
Let $X$ be a separated scheme of finite type over a perfect field $k$ of characteristic $p>0$. Assume that $X$ is smooth and let $E$ be a simple normal crossing divisor on $X$. Similarly to the previous subsection, we can define the logarithmic de Rham complex $F_*\Omega^\bullet(\log E)$ and the logarithmic Cartier operator \cite{Kat70} so that we have a short exact sequence
\begin{equation}\label{eq:log Cartier operator}
    0 \to B\Omega^i_X(\log E) \to Z\Omega^i_X(\log E) \xrightarrow{C} \Omega^i_X(\log E) \to 0.
\end{equation}

This construction can be extended to $\bQ$-divisors as in \cite[Section 3]{Hara98}.
Let $\Delta$ be a $\bQ$-divisor on $X$ such that $\Supp(\Delta) \subseteq E$. Then $F_*(\Omega^\bullet_X(\log E)(\Delta))$ is a complex\footnote{following our convention, this means that we consider $F_*(\Omega^\bullet_X(\log E)(\rdown{\Delta}))$}, which for simplicity of notation we denote by
\[
F_*\Omega^\bullet_X(\log E)(\Delta).
\]
Caution is advised here, as this notation can be easily misinterpreted. 
Set
\begin{align*}
B\Omega^i_X(\log E)(\Delta) &\coloneqq {\rm im}\left(d \colon F_*\Omega^{i-1}_X(\log E)(\Delta) \to F_*\Omega^i_X(\log E)(\Delta)\right), \text{ and }\\
Z\Omega^i_X(\log E)(\Delta) &\coloneqq \ker\left(d \colon F_*\Omega^{i}_X(\log E)(\Delta) \to F_*\Omega^{i+1}_X(\log E)(\Delta)\right).
\end{align*}
Then we have isomorphisms
\[
\Omega_X^i(\log E)(\Delta)\cong \mathcal{H}^i(F_*\Omega^{\bullet}_X(\log E)(p\lfloor \Delta \rfloor)) \cong \mathcal{H}^i(F_*\Omega^{\bullet}_X(\log E)(p\Delta)),
\]
where the first ones follows from the usual log Cartier operator \eqref{eq:log Cartier operator} and the second one follows from \cite[Lemma 3.3]{Hara98} and is induced by the natural inclusion of complexes.
This yields an exact sequence
\begin{equation} \label{eq:Hara-ses}
0 \to B\Omega^i_X(\log E)(p\Delta) \to Z\Omega^i_X(\log E)(p\Delta) \xrightarrow{C_{\Delta}} \Omega^i_X(\log E)(\Delta) \to 0.
\end{equation}

\begin{remark}
In this remark we discuss the functoriality of the above definitions under the change of $\Delta$. For $\Delta'\leq \Delta$ with $\Supp(\Delta')\subseteq E$, we obtain the following commutative diagram:
\[
\begin{tikzcd}
    \Omega_X^i(\log E)(\Delta') \arrow[r,"\cong"]\arrow[d] & \mathcal{H}^i(F_*\Omega^{\bullet}_X(\log E)(p\lfloor\Delta'\rfloor)) \arrow[r,"\cong"]\arrow[d] & \mathcal{H}^i(F_*\Omega^{\bullet}_X(\log E)(p\Delta'))\arrow[d]\\
    \Omega_X^i(\log E)(\Delta) \arrow[r,"\cong"] &\mathcal{H}^i(F_*\Omega^{\bullet}_X(\log E)(p\lfloor \Delta \rfloor)) \arrow[r,"\cong"] & \mathcal{H}^i(F_*\Omega^{\bullet}_X(\log E)(p\Delta)),
\end{tikzcd}
\]
which yields the commutative diagram:
\[
\begin{tikzcd}
    0\arrow[r] & B\Omega_X^i(\log E)(p\Delta') \arrow[r]\arrow[hook]{d} & Z\Omega_X^i(\log E)(p\Delta') \arrow[r]\arrow[hook]{d} & \Omega_X^i(\log E)(\Delta')\arrow[r] \arrow[hook]{d} & 0\\
    0\arrow[r] & B\Omega_X^i(\log E)(p\Delta) \arrow[r] & Z\Omega_X^i(\log E)(p\Delta) \arrow[r] & \Omega_X^i(\log E)(\Delta)\arrow[r] & 0.
\end{tikzcd}
\]
\end{remark}
\vspace{1em}
Next, we set
\[
G\Omega^i_X(\log E)(\Delta) \coloneqq \frac{F_*\Omega^{i}_X(\log E)(\Delta)}{B\Omega^{i}_X(\log E)(\Delta)}.
\]
Then we have the inverse Cartier operator
\begin{align}
\label{eq:dual-Cartier-Hara} C^{-1}_{\Delta}\colon \Omega^i_X(\log E)(\Delta)&\underset{\cong}{\xleftarrow{C_{\Delta}}} \frac{Z\Omega^{i}_X(\log E)(p\Delta)}{B\Omega^{i}_X(\log E)(p\Delta)} \\[0.2em]
&\xhookrightarrow{\hphantom{C_{\Delta}}} \frac{F_*\Omega^{i}_X(\log E)(p\Delta)}{B\Omega^{i}_X(\log E)(p\Delta)}=G\Omega^i_X(\log E)(p\Delta). \nonumber
 \end{align}

\subsubsection{Hara's higher Cartier operators} \label{sss:higher-Cartier-Hara}
We continue with the notation as in the above subsection.
As in \ref{sss:higher-Cartier}, by composing $C_\Delta$ with itself as a partial map, one constructs the iterated Cartier operator 
\[
C_{n,\Delta} \colon F^n_*\Omega^i_X(\log E)(p^n\Delta) \rightharpoonup \Omega^i_X(\log E)(\Delta).
\]
With $Z_n\Omega^i_X(\log E)(p^n\Delta)$ and $B_n\Omega^i_X(\log E)(p^n\Delta)$ defined to be the cycles and boundaries in the complex 
\begin{equation} \label{eq:logDeltadeRhamHigher}
F^n_*\Omega^\bullet_X(\log E)(p^n\Delta),
\end{equation}
respectively, we get a short exact sequence:
\begin{equation}\label{eq:Hara-ses(iterated)}
    0 \rightarrow B_n\Omega_X^i(\log E)(p^n \Delta) \rightarrow Z_n\Omega_X^i(\log E)(p^n \Delta) \xrightarrow{C_{n,\Delta}} \Omega_X^i(\log E)(\Delta) \rightarrow 0.
\end{equation} 

\begin{remark} \label{eq:Hara-ses(iterated) commuative} \label{remark:Hara-ses(iterated) commuative}
For $\Delta'\leq \Delta$ with $\Supp(\Delta')\subseteq E$, we obtain the following commutative diagram:
\begin{equation*}
    \begin{tikzcd}
    0\arrow[r] & B_n\Omega_X^i(\log E)(p^n\Delta') \arrow[r]\arrow[hook]{d} & Z_n\Omega_X^i(\log E)(p^n\Delta') \arrow[r]\arrow[hook]{d} & \Omega_X^i(\log E)(\Delta')\arrow[r]\arrow[hook]{d} & 0\\
    0\arrow[r] & B_n\Omega_X^i(\log E)(p^n\Delta) \arrow[r] & Z_n\Omega_X^i(\log E)(p^n\Delta) \arrow[r] & \Omega_X^i(\log E)(\Delta)\arrow[r] & 0.
\end{tikzcd}
\end{equation*}
\end{remark}
Set
\[
G_n\Omega^i_X(\log E)(\Delta) \coloneqq \frac{F^n_*\Omega^{i}_X(\log E)(\Delta)}{B_n\Omega^{i}_X(\log E)(\Delta)}.
\]
We have the inverse Cartier operator
\begin{align}
\label{eq:dual-Cartier-Hara-higher} 
C^{-1}_{n,\Delta}\colon \Omega^i_X(\log E)(\Delta)&\underset{\cong}{\xleftarrow{C_{n,\Delta}}} \frac{Z_n\Omega^{i}_X(\log E)(p^n\Delta)}{B_n\Omega^{i}_X(\log E)(p^n\Delta)} \\[0.2em]
&\xhookrightarrow{\hphantom{C_{n,\Delta}}} \frac{F^n_*\Omega^{i}_X(\log E)(p^n\Delta)}{B_n\Omega^{i}_X(\log E)(p^n\Delta)}=G_n\Omega^i_X(\log E)(p^n\Delta). \nonumber
\end{align}

\subsubsection{Duality for boundaries and cycles} \label{sss:duality}
We continue with the same assumptions as above: $X$ is a smooth variety over a perfect field $k$ of positive characteristic $p>0$ and $E$ is a simple normal crossing divisor. In this article, we are particularly interested in the complex $F_*\Omega^\bullet_X(\log E)(-E)$, which is constructed by taking, for example, $\Delta = -\frac{1}{p}E$ in Construction \ref{sss:Hara}. We warn the reader again that this complex is \emph{not} the same as $\left(F_*\Omega^\bullet_X(\log E)\right) \otimes \cO_X(-E)$.

An important reason why $F_*\Omega^\bullet_X(\log E)(-E)$ shows up in this article is that it is the dual of $F_*\Omega^\bullet_X(\log E)$.

\begin{lemma} \label{lem:preliminaries-duality} With notation as above, 
\begin{equation} \label{eq:dual-of-log-de-rham}
R\Hom_{\cO_X}(F_*\Omega^\bullet_X(\log E), \omega_X)[-d] \cong F_*\Omega^\bullet_X(\log E)(-E).
\end{equation}
Moreover,
\begin{enumerate}
\item $\cHom_{\cO_X}(F_*\Omega^i_X(\log E), \omega_X) \cong F_*\Omega^{n-i}_X(\log E)(-E)$,\\[-0.8em]
\item $\cHom_{\cO_X}(B\Omega^{i}_X(\log E), \omega_X) \cong B\Omega_X^{n-i+1}(\log E)(-E)$, and\\[-0.8em] 
\item $\cHom_{\cO_X}(Z\Omega^i_X(\log E), \omega_X) \cong G\Omega^{n-i}_X(\log E)(-E).$
\end{enumerate}
\end{lemma}
\begin{proof}
First, note that 
\[
{\cHom}_{\cO_X}(F_*\Omega^i_X(\log E), \omega_X) \cong F_*{\cHom}_{\cO_X}(\Omega^i_X(\log E), \omega_X) \cong F_*\Omega^{n-i}_X(\log E)(-E),
\]
where the last isomorphism follows from the fact that we have a perfect pairing: $\Omega^i_X(\log E) \times \Omega^{n-i}_X(\log E) \to \omega_X(E)$ given by the wedge product. In particular, 
\[
{R\Hom}_{\cO_X}(F_*\Omega^\bullet_X(\log E), \omega_X)[-d] \cong F_*\Omega^\bullet_X(\log E)(-E).
\]
This concludes the proof of (\ref{eq:dual-of-log-de-rham}) and (1). 

By applying ${\cHom}_{\cO_X}(-, \omega_X)$ to the sequence of maps:
\[
Z\Omega^i_X(\log E) \hookrightarrow F_*\Omega^i_X(\log E) \twoheadrightarrow B^{i+1}_X(\log E) \hookrightarrow F_*\Omega^{i+1}_X(\log E)
\]
we get:
\begin{multline*}
{\cHom}_{\cO_X}(Z\Omega^i_X(\log E), \omega_X) \twoheadleftarrow F_*\Omega^{n-i}_X(\log E)(-E) \\ \hookleftarrow  {\cHom}_{\cO_X}(B^{i+1}_X(\log E), \omega_X)  \twoheadleftarrow F_*\Omega^{n-i-1}_X(\log E)(-E),
\end{multline*}

and so we immediately obtain that:
\begin{align*}
{\cHom}_{\cO_X}(B^{i+1}_X(\log E), \omega_X) &\cong B\Omega_X^{n-i}(\log E)(-E), \text{ and } \\
{\cHom}_{\cO_X}(Z\Omega^i_X(\log E), \omega_X) &\cong \frac{F_*\Omega^{n-i}_X(\log E)(-E)}{B\Omega_X^{n-i}(\log E)(-E)} = G\Omega^{n-i}_X(\log E)(-E),
\end{align*}
concluding the proof of (2) and (3).
\end{proof}

\subsubsection{Residue exact sequences}
We continue with the same assumptions as above.  In what follows, we generalise the residue exact sequences to sheaves $B\Omega_X^i(\log E)$, $Z\Omega_X^i(\log E)$, and $G\Omega_X^i(\log E) = \frac{F_*\Omega^i_X(\log E)}{B\Omega^i_X(\log E)}$.

\begin{lemma}\label{lem:residue exact sequences} Let $X$ be a smooth variety over a perfect field $k$ of positive characteristic $p>0$ and let $E$ be a simple normal crossing divisor. Pick an irreducible component $D \subseteq E$. Then there exist exact sequences:
\begin{enumerate}
    \item $0\to \Omega^{i}_X(\log E-D)\to \Omega^{i}_X(\log E) \to \Omega_D^{i-1}(\log \, (E-D)|_D)\to 0$,\\[-1em]
    \item $0\to B\Omega^{i}_X(\log E-D)\to B\Omega^{i}_X(\log E) \to B\Omega_D^{i-1}(\log \, (E-D)|_D)\to 0$,\\[-1em]
    \item $0\to Z\Omega^{i}_X(\log E-D)\to Z\Omega^{i}_X(\log E) \to Z\Omega_D^{i-1}(\log \, (E-D)|_D)\to 0$, and\\[-1em]
    \item $0\to G\Omega^{i}_X(\log E-D)\to G\Omega^{i}_X(\log E) \to G\Omega_D^{i-1}(\log \, (E-D)|_D)\to 0$.
\end{enumerate}
\end{lemma}
\begin{proof}
(1) is the usual residue exact sequence \cite[2.3 Properties (b)]{EV92}.
The existence of  (2) and (3) is proven by descending induction on $i$.
Specifically, by the short exact sequence
\[
0\to B\Omega^d_X(\log E) \to F_{*}\Omega^d_X(\log E) \xrightarrow{C} \Omega^d_X(\log E) \to 0,
\]
we can deduce (2) for $i=d$, where $d\coloneqq \dim X$.
Then 
by the short exact sequence
\[
0\to Z\Omega^{d-1}_X(\log E) \to F_{*}\Omega^{d-1}_X(\log E) \xrightarrow{d} B\Omega^d_X(\log E) \to 0,
\]
we obtain (3) for $i=d-1$. By repeating this argument, we get (2) and (3) in general.
Finally, (4) follows from (2) and the short exact sequence
\[
0\to B\Omega^{i}_X(\log E) \to F_{*}\Omega^{i}_X(\log E) \to G\Omega^{i}_X(\log E) \to 0.
\]
\end{proof}

\subsubsection{Vanishing theorem}
The following vanishing result will be key in the proof of the main theorem. Note that we will only use this lemma for $l=0$, but the statement for $l>0$ will be needed in the upcoming work on non-isolated singularities.
\begin{lemma} \label{lemma:Kawakami-vanishing}
Let $X$ be a normal variety over a perfect field $k$ of characteristic $p>0$. Let $\pi \colon Y \to X$ be a log resolution of $X$ with exceptional divisor $E$. Let $A$ be an ample anti-effective exceptional $\bQ$-divisor on $Y$ such that $\rdown{A}=-E$. Fix an integer $l\geq 0$ and assume that 
\begin{equation} \label{eq:assumption-Kawakami-vanishing}
    R^{>l}\pi_*\Omega^i_Y(\log E)(p^rA) = 0
\end{equation}
for all $0 \leq i \leq k-1$ and $r \geq 0$. Then for $n \gg 0$ and all $0 \leq i \leq k$ we have that
\begin{enumerate}
    \item $R^{>0}\pi_*F^n_*\Omega^i_Y(\log E)(p^nA) = 0$, \\[-0.9em]
    \item $R^{>l}\pi_*B_n\Omega^i_Y(\log E)(p^nA) = 0$, and \\[-0.9em]
    \item $R^{>l}\pi_*G_n\Omega^i_Y(\log E)(p^nA) = 0$.
\end{enumerate}
\end{lemma}
\noindent {In fact, in the proof below we show that (2) holds for all $n \geq 1$.}

\begin{proof}
Vanishing (1) is clear by Serre vanishing, and (3) follows immediately from (1) and (2) in view of the exact triangle 
\[
R\pi_*B_n\Omega^i_Y(\log E)(p^nA) \to R\pi_*F^n_*\Omega^i_Y(\log E)(p^nA) \to R\pi_*G_n\Omega^i_Y(\log E)(p^nA) \xrightarrow{+1}.
\]
Thus we can focus on showing (2):
\[
R^{>l}_*B_n\Omega^i_Y(\log E)(p^nA) = 0.
\]
To this end, first consider the short exact sequence (see \cite[(5.6.3)]{KTTWYY1}):
\[
0 \to F^{n-1}_*B\Omega^i_Y(\log E)(p^nA) \to B_{n}\Omega^i_Y(\log E)(p^nA) \to B_{n-1}\Omega^i_Y(\log E)(p^{n-1}A) \to 0,
\]
which reduces the problem to showing that
\begin{equation} \label{eq:n1:Kawakami-vanishing}
R^{>l}\pi_*B\Omega^i_Y(\log E)(p^rA) = 0
\end{equation}
for all $r \geq 1$ and $0 \leq i \leq k$. 
To prove this, we argue by ascending induction on $i$. The statement holds for $i=0$ by definition, and so we may assume that it holds for $i-1$:
\begin{equation} \label{eq:n1:Kawakami-vanishing-induction}
R^{>l}\pi_*B\Omega^{i-1}_Y(\log E)(p^rA) = 0
\end{equation}
and aim for showing that it holds for $0 < i \leq k$. Now the short exact sequence
\[
0 \to Z\Omega^{i-1}_Y(\log E)(p^rA) \to F_*\Omega^{i-1}_Y(\log E)(p^rA) \to B\Omega^i_Y(\log E)(p^rA) \to 0
\]
and (\ref{eq:assumption-Kawakami-vanishing}) reduces (\ref{eq:n1:Kawakami-vanishing}) to showing that
\[
R^{>l}\pi_*Z\Omega^{i-1}_Y(\log E)(p^rA) = 0,
\]
which follows by (\ref{eq:n1:Kawakami-vanishing-induction}) and the short exact sequence
\[
0 \to B\Omega^{i-1}_Y(\log E)(p^rA) \to Z\Omega^{i-1}_Y(\log E)(p^rA) \xrightarrow{C_{p^{r-1}A}} \Omega^{i-1}_Y(\log E)(p^{r-1}A) \to 0. \qedhere
\]
\end{proof}

\subsection{Deligne--Du Bois complex via eh-topology} 
For the usual definition of the Deligne--Du Bois complex using simplicial resolutions we refer the reader to \cite{DuBois}, \cite{GNPP}, and \cite[Section 7.3]{Peter-Steenbrink(Book)}.
Here, we define the Deligne--Du Bois complex via eh-differentials following \cite{Geisser}, \cite{Huber-Jorder}, \cite{Huber-Kebekus-Kelly}, and \cite{Huber-Kelly}. 
In this section, $k$ is a perfect field of arbitrary characteristic.
We denote by $\mathrm{Sch}_k$ the category of separated schemes of finite type over $k$.

\begin{definition}\label{def:abstract blow-up}
    Consider a Cartesian diagram
    \[
    \begin{tikzcd}
        E \arrow[r]\arrow[d] & Y\arrow[d,"\pi"]\\
 Z\arrow[r,"\iota"] & X
    \end{tikzcd}
    \]
in $\mathrm{Sch}_k$.
 We say that this diagram is an \textit{abstract blow-up square} if $\pi$ is proper, $\iota$ is a closed immersion, and $\pi|_{Y\setminus E}\colon Y\setminus E \to X\setminus Z $ is an isomorphism. 
 We denote the diagram of the abstract blow-up by $(Y\to X, E\to Z)$.
\end{definition}

\begin{definition}
    The \textit{eh-topology} on $\mathrm{Sch}_k$ is the Grothendieck topology generated by:
    \begin{enumerate}
        \item  \'etale coverings and
        \item the induced map $Z \coprod Y \to X$ for every abstract blow-up $(Y\to X, E\to Z)$.
    \end{enumerate}
    We denote by $\mathrm{Sch_{eh}}$ the site defined by the eh-topology.
\end{definition}

\begin{definition}
    Let $X\in \mathrm{Sch_{k}}$.
    Let $\rho_{X}\colon \mathrm{Sch_{eh}}\to X_{\mathrm{Zar}}$ be the natural morphism of sites.
    We define the Du Bois complexes by 
    \[
    \underline{\Omega}^i_{X,\eh} \coloneqq R\rho_{X*}\Omega^i_{\mathrm{eh}}.
    \]
    Further, we denote $\rho_{X*}\Omega^i_{\mathrm{eh}}(=\mathcal{H}^0(\underline{\Omega}^i_{X,\eh}))$ by $\Omega^i_{X,{\rm eh}}$.

    When $k$ is of characteristic zero, the complex $R\rho_{X*}\Omega^i_{\mathrm{eh}}$ coincides with the usual Deligne--Du Bois complex by \cite[Theorem 7.12]{Huber-Jorder}.
    Moreover, in this case, it is equal to $R\rho_{X*}\Omega^i_{\mathrm{h}}$, where $\Omega^i_{\mathrm{h}}$ denotes the sheafification in the $h$-topology \cite[Theorem 3.6]{Huber-Jorder}.
    In view of this, when $k$ is of characteristic zero, we denote $\Omega^i_{X,{\rm eh}}$ by $\Omega^i_{X,h}$ and $\underline{\Omega}^i_{X,\eh}$ by just $\underline{\Omega}^i_{X}$. 
\end{definition}

\begin{proposition}\label{prop:blow-up}
    Let $(Y\to X, E\to Z)$ be an abstract blow-up square. 
    Then we have the following exact triangle:
    \[\underline{\Omega}^i_{X,\eh} \to \underline{\Omega}^i_{Z,\eh}\oplus R\pi_{*}\underline{\Omega}^i_{Y,\eh} \to R\pi_{*}\underline{\Omega}^i_{E,\eh} \xrightarrow{+1}.
    \]
\end{proposition}
\begin{proof}
    The assertion follows from the same argument as in \cite[Proposition 3.2]{Geisser} (cf.\ \cite[proof of Proposition 6.15]{Huber-Jorder}).
\end{proof}

\begin{theorem}[\textup{\cite[Theorem 4.7]{Geisser} and \cite[Theorem 1.1 (1.1.1)]{Huber-Kebekus-Kelly}}]\label{thm:eh for smooth}
    Let $X$ be a $d$-dimensional smooth scheme over $k$.
    Then $\Omega^i_{X,{\rm eh}}=\Omega^{i}_X$.
    Moreover, if the $R(d)$-condition holds, then $\underline{\Omega}_{X,\eh}^{i}=\Omega^{i}_X$.
\end{theorem}

\begin{corollary}\label{cor:eh for SNC}
    Let $E$ be a simple normal crossing $d$-dimensional reduced scheme over $k$.
    Then $\Omega^i_{E,\eh}\cong \Omega^i_E/{\rm tors}$.
    
    Moreover, if the $R(d)$-condition holds, then $\underline{\Omega}^{i}_{E,\eh}=\Omega^{i}_E/{\rm tors}$.
\end{corollary}
\begin{proof}
    Let $m$ be the number of the irreducible components of $E$.
    By induction, we may assume that the assertions hold for $d'$-dimensional simple normal crossing reduced schemes over $k$ with $m'$-irreducible components with $d'<d$ or $m'<m$.
    By a direct computation (Theorem \ref{thm:torsion exact sequence}), we have a short exact sequence
    \[
    0\to \Omega^i_{E}/{\rm tors} \to \Omega^i_{E_1^c}/{\rm tors} \oplus \Omega^i_{E_1} \to {\Omega}^i_{E_1^c\cap E_1}/{\rm tors} \to 0.
    \]
    %\commentbox{Jakub: that's not exactly what the reference says.}
    By applying $\mathcal{H}^0$ to the exact triangle in Proposition \ref{prop:blow-up} and by induction, we thus get the diagram:
    \[
    \begin{tikzcd}
    0 \arrow[r]&\Omega^i_{E}/{\rm tors} \arrow[r] \arrow[d] & \Omega^i_{E_1^c}/{\rm tors} \oplus \Omega^i_{E_1} \arrow[r] \ar[d,,equal] & {\Omega}^i_{E_1^c\cap E_1}/{\rm tors} \arrow[r]\ar[d,,equal]&0 \\
0 \arrow[r]& \Omega^i_{E,\eh} \arrow[r]  & \Omega^i_{E_1^c,\eh}\oplus \Omega^i_{E_1,\eh} \arrow[r]  & \Omega^i_{E_1^c\cap E_1,\eh} & {} 
\end{tikzcd}
    \]
    and the first assertion holds.
    If the $R(d)$-condition is satisfied, then we have 
    \[
    \begin{tikzcd}
    \Omega^i_{E}/{\rm tors} \arrow[r] \arrow[d] & \Omega^i_{E_1^c}/{\rm tors} \oplus \Omega^i_{E_1} \arrow[r] \ar[d,,equal] & {\Omega}^i_{E_1^c\cap E_1}/{\rm tors} \arrow[r, "+1"]\ar[d,,equal]& {} \\
\underline{\Omega}^i_{E,\eh} \arrow[r]  & \underline{\Omega}^i_{E_1^c,\eh}\oplus \Omega^i_{E_1,\eh} \arrow[r]  & \underline{\Omega}^i_{E_1^c\cap E_1,\eh} \arrow[r, "+1"]& {} 
\end{tikzcd}
    \]
    by Theorem \ref{thm:eh for smooth}, and the latter assertion is valid. \qedhere
\end{proof}

%\commentbox{Jakub: we do not need the statement below in this generality, so we should remove it?}
\begin{corollary}\label{cor:blow-up triangle}
Let $X$ be a $d$-dimensional reduced separated scheme of finite type over $k$ and let $Z\hookrightarrow X$ be a closed immersion over $k$.
Suppose that the $R(d)$-condition holds.
Let $\pi\colon Y\to X$ be a proper birational morphism such that $Y$ is smooth and the reduced preimage $E$ of $Z$ is simple normal crossing.

Then we have the following diagram consisting of two exact triangles:
\[
\begin{tikzcd}
R\pi_*\Omega^i_Y(\log E)(-E) \arrow[r] \arrow[d, equal] & R\pi_*\Omega^i_Y \arrow[r]  & R\pi_*\frac{\Omega^i_E}{\rm tors} \arrow[r, "+1"]  & {} \\
\underline{\Omega}^i_{X,Z,\eh} \arrow[r] &  \underline{\Omega}^i_{X,\eh} \arrow[r] \arrow[u] & \underline{\Omega}^i_{Z,\eh} \arrow[u] \arrow[r, "+1"] & {}, 
\end{tikzcd}
\]
where $\underline{\Omega}^i_{X,Z,\eh} \coloneqq \mathrm{Cocone} (\underline{\Omega}^i_{X,\eh} \to \underline{\Omega}^i_{Z,\eh})$.
\end{corollary}
\begin{proof}
    Since we are assuming the $R(d)$-condition, we have $\underline{\Omega}^i_{Y,\eh}=\Omega^i_Y$ and $\underline{\Omega}^i_{E,\eh}=\Omega^i_E/{\rm tors}$ by Theorem \ref{thm:eh for smooth} and Corollary \ref{cor:eh for SNC}.
    By Proposition \ref{prop:blow-up}, we thus get the exact triangle
    \[
    \underline{\Omega}^i_{X,\eh}\to \underline{\Omega}^i_{Z,\eh}\oplus R\pi_{*}\Omega^i_{Y} \to R\pi_{*}\Omega^i_{E}/{\rm tors} \xrightarrow{+1}.
    \]
    Recall that $\Omega^i_Y(\log\,E)(-E)={\rm ker}(\Omega^i_Y\to \Omega^i_E/{\rm tors})$.
    Now, the assertion follows by the octahedral axiom. \qedhere
\end{proof}

\section{$k$-$F$-injective and $k$-Du Bois singularities}
In this section we introduce the notion of $k$-$F$-injective and pseudo-$k$-Du Bois singularities in characteristic $p>0$ in the case of isolated singularities.

Throughout this section, we work in the following setting. 

\begin{setting} \label{setting:iso}
Let $X = \Spec R$ be a $d$-dimensional normal affine variety over a perfect field of arbitrary characteristic. Let $x$ be a closed point corresponding to a maximal ideal $\m$. Assume that $X \,\backslash\, \{x\}$ is regular. Further, assume there exists $\pi \colon Y \to X$ a projective birational morphism which is an isomorphism away from $\{x\}$ and such that $Y$ is regular and $E = {\rm Exc}(\pi)$ is simple normal crossing.
\end{setting}

\subsection{$\underline{\cO}_{X,\pi}$ and pseudo-Du Bois singularities}
Since the definitions of the Deligne Du Bois sheaves $\underline{\Omega}^k_{X,\pi}$ for $k=0$ and $k>0$ will be significantly different, we focus on the former case first.

Note that the short exact sequence 
\[
0 \to \cO_Y(-E) \to \cO_Y \to \cO_E \to 0
\]
induces a natural restriction map $R\pi_*\cO_Y \to R\pi_*\cO_E$. 
\begin{definition}\label{def:0thDu Bois cpx(isolated)}
With notation of Setting \ref{setting:iso}, 
\[\underline{\cO}_{X,\pi} \coloneqq {\rm Cocone}(R\pi_*\cO_Y \to R\pi_*\cO_E \to R^{>0}\pi_*\cO_E)\,\,\text{with}\,\,\sO_{X,\pi}\coloneqq \mathcal{H}^0(\underline{\cO}_{X,\pi}).\]
\end{definition}

In particular, by the following diagram
\[
\begin{tikzcd}
\hphantom{a} & \hphantom{a} &  \hphantom{a} & \hphantom{a} \\
0 \arrow[r] \arrow[u, "+1"] & R^{>0}\pi_*\cO_E \arrow[r,equal] \arrow[u, "+1"] & R^{>0}\pi_*\cO_E \arrow[r, "+1"] \arrow[u, "+1"] & {} \\
R\pi_*\cO_Y(-E) \arrow[r] \arrow[u] & R\pi_*\cO_Y \arrow[r] \arrow[u] & R\pi_*\cO_E \arrow[r, "+1"] \arrow[u] & {} \\
R\pi_*\cO_Y(-E) \arrow[r] \arrow[u, "="] & \underline{\cO}_{X,\pi} \arrow[r] \arrow[u] & \pi_*\cO_E \arrow[r, "+1"] \arrow[u] & {}
\end{tikzcd}
\]
we get an exact triangle
\begin{equation} \label{eq:0iso-exact-triangle}
R\pi_*\cO_Y(-E) \to \underline{\cO}_{X,\pi} \to k \xrightarrow{+1}
\end{equation}
where $\pi_*\cO_E=\cO_{x} = k$. Moreover, the natural pullback map decomposes as $\cO_X \to \underline{\cO}_{X,\pi} \to R\pi_*\cO_Y$, which is an isomorphism on $\cH^0$, hence
\begin{equation} \label{eq:iso-du-bois-structure-sheaf-0}
\sO_{X,\pi}=\cH^0(\underline{\cO}_{X,\pi}) = \cO_X.
\end{equation}
This also shows that the map $\cH^0(\underline{\cO}_{X,\pi}) \to k$ induced by the above exact triangle is surjective, hence the connecting homomorphism $k \to R^1\pi_*\cO_Y(-E)$ is zero and
\begin{equation}\label{eq:du bois isolated}
    \tau^{>0}\underline{\sO}_{X,\pi} \cong  R^{>0}\pi_{*}\sO_Y(-E).
\end{equation}

\begin{definition}\label{def:Du Bois(isolated)}
With the notation of Setting \ref{setting:iso}, we say that $X$ is \emph{Du Bois along $\pi$} if the natural map induced by \eqref{eq:iso-du-bois-structure-sheaf-0}:
\[
\cO_X[0] \to \underline{\cO}_{X,\pi}
\]
is a quasi-isomorphism.
\end{definition}
\begin{definition}\label{def:pseudo-du-bois(isolated)}
With the notation of Setting \ref{setting:iso}, we say that $X$ is \emph{pseudo-Du Bois along $\pi$} if the natural map induced by (\ref{eq:iso-du-bois-structure-sheaf-0}): 
\[
H^j_\m(\cO_X) \to \mathbb{H}^j_\m(\underline{\cO}_{X,\pi})
\]
is injective for every $0 \leq j \leq d$.
\end{definition}

\begin{lemma}\label{lem:independence of resolutions,i=0,iso}
    With the notation of Setting \ref{setting:iso}, the following is true.
    \begin{enumerate}
        \item We have that $\sO_{X,\pi}=\sO_{X}$.  In particular, $\sO_{X,\pi}$ does not depend on $\pi$.
        \item Suppose that the the $R(d)$-condition holds. Then $\underline{\sO}_{X,\pi}=\underline{\Omega}_{X,\eh}^{0}$.  In particular, the definitions of $\underline{\sO}_{X,\pi}$ in Definition \ref{def:0thDu Bois cpx(isolated)}, Du Bois along $\pi$ in Definition \ref{def:Du Bois(isolated)}, and pseudo-Du Bois along $\pi$ in Definition \ref{def:pseudo-du-bois(isolated)} do not depend on $\pi$.
        \item Suppose that $k$ is of characteristic zero. Then pseudo-Du Bois is equivalent to Du Bois.
    \end{enumerate}
\end{lemma}
\begin{proof}
(1) follows from \eqref{eq:iso-du-bois-structure-sheaf-0}.
    Now, suppose that the $R(d)$-condition holds.
    By Corollary \ref{cor:blow-up triangle}, we have the following diagram
    \[
\begin{tikzcd}
R\pi_*\cO_Y(-E) \arrow[r] \arrow[d, equal] & R\pi_*\cO_Y \arrow[r]  & R\pi_*\cO_E \arrow[r, "+1"]  & {} \\
\underline{\cO}_{X,x,\eh} \arrow[r] & \underline{\cO}_{X,\eh} \arrow[r] \arrow[u] & \cO_{x}=\pi_{*}\sO_E \arrow[u] \arrow[r, "+1"] & {}, 
\end{tikzcd}
\]
which shows that $\underline{\cO}_{X,\eh}={\rm Cocone}(R\pi_*\cO_Y \to R\pi_*\cO_E \to \tau^{>0}R\pi_*\cO_E)$.
Thus, we obtain $\underline{\cO}_{X,\pi}=\underline{\cO}_{X,\eh}$ and (2) holds.
Finally, (3) follows from Theorem \ref{thm:pseudo-pre-k-Du Bois=pre-k-Du Bois in char 0}.
\end{proof}

The reader should draw an analogy between the above definitions and those of rational ($R\pi_*\cO_Y = \cO_X$) and pseudo-rational (CM, $\pi_*\omega_Y=\omega_X$) singularities. The two notions coincide in characteristic zero, but that is not the case in positive characteristic due to the lack of the Grauert--Riemenschneider vanishing.

Before proceeding, observe that $\underline{\cO}_{X,\pi}$ is endowed with a natural Frobenius map when $X$ is of characteristic $p>0$. Indeed, we have the following diagram
\[
\begin{tikzcd}
\underline{\cO}_{X,\pi} \ar[d] \arrow[r] & R\pi_*\cO_Y \ar{d} \arrow[r, "F"] & F_*R\pi_*\cO_Y \arrow[d] & \\
\pi_*\cO_E \arrow[bend right = 15]{rrr}[swap]{0} \arrow[r] & R\pi_*\cO_E \arrow[r, "F"] & F_*R\pi_*\cO_E \ar{r} & F_*R^{>0}\pi_*\cO_E,
\end{tikzcd}
\]
which shows that the map $\underline{\cO}_{X,\pi} \to R\pi_*\cO_Y \xrightarrow{F} F_*R\pi_*\cO_Y$ factors through
\[
{\rm Cocone}(F_*R\pi_*\cO_Y \to F_*R^{>0}\pi_*\cO_E) \cong F_*\underline{\cO}_{X,\pi}
\]
by Lemma \ref{lem:derived-cat-factor}.
This yields the Frobenius map
\begin{equation}\label{eq:Frobenius of  O_{pi}}
    F \colon \underline{\cO}_{X,\pi} \to F_*\underline{\cO}_{X,\pi}
\end{equation}
which fits into the following commutative diagram:
\[
\begin{tikzcd}
\underline{\cO}_{X,\pi} \arrow[r,"F"] \arrow[d] & F_{*}\underline{\cO}_{X,\pi}\arrow[d]  \\
R\pi_{*}\sO_Y \arrow[r,"F"] & R\pi_{*}\sO_Y.
\end{tikzcd}
\]
Motivated by \cite{Bhatt-Schwede-Takagi}, we prove the following.
\begin{proposition}\label{prop:BST}
With notation of Setting \ref{setting:iso}, suppose that
$k$ is of characteristic $p>0$ and
$X$ is $F$-injective. Then $X$ is pseudo-Du Bois along $\pi$.
\end{proposition}

\begin{proof}
Pick an ample exceptional Cartier divisor $A$.
By negativity lemma, we have $\Supp A = E$ and $A\leq 0$. We may assume that $\rup{A}=0$. 
By Serre vanishing we may pick $m \gg 0$ such that
\begin{equation}
R^{>0}\pi_*\cO_Y(mA)=0.\label{eq:Serre-Frob0}
\end{equation}
Now pick $e\gg 0$ so that $-E \geq mA \geq -p^eE$. In particular,
\begin{equation}
F^e \colon R\pi_*\cO_Y(-E) \to F^e_*R\pi_*\cO_Y(-E) \ \text{ factors through\ \ $F^e_*R\pi_*\cO_Y(mA)$. } \label{eq:factor-Frob0}
\end{equation}
Now consider the following diagram
\[
\begin{tikzcd}
\mathllap{\cO_X =\ } \cO_{X,\pi} \ar{r} & \underline{\cO}_{X,\pi} \ar[d] \arrow[r,"F^e"] & F^e_*\underline{\cO}_{X,\pi} \ar{d}  \\
& \tau^{>0}\underline{\cO}_{X,\pi} \arrow[r,"F^e"] & \tau^{>0}F^e_*\underline{\cO}_{X,\pi}.
\end{tikzcd}
\]
Since $\tau^{>0}\underline{\cO}_{X,\pi}\cong R^{>0}\pi_*\cO_Y(-E)$ (see \eqref{eq:du bois isolated}), the lower horizontal arrow in the above diagram may be identified with:
\[
F^e \colon R^{>0}\pi_*\cO_Y(-E) \to R^{>0}\pi_*F^e_*\cO_Y(-E)
\]
which is zero by \eqref{eq:factor-Frob0} and \eqref{eq:Serre-Frob0}. Then $F^e \colon \underline{\cO}_{X,\pi} \to F^e_*\underline{\cO}_{X,\pi}$
factors into
\[
\underline{\cO}_{X,\pi} \to F^e_{*}\sO_X \xrightarrow{\text{nat.}} F^e_*\underline{\cO}_{X,\pi}
\]
by Lemma \ref{lem:derived-cat-factor}. By taking $\mathcal{H}^0$, we obtain a factorisation
\[
F^e \colon \cO_X\to \underline{\cO}_{X,\pi} \to F^e_*\cO_X.
\]
This immediately implies the statement of the proposition.
\end{proof}

Let us try to understand what being pseudo-Du Bois means. By local duality  \cite[\href{https://stacks.math.columbia.edu/tag/0AAK}{Tag 0AAK}]{stacks-project},
Definition \ref{def:pseudo-du-bois(isolated)} is equivalent to the surjectivity of
\begin{equation} \label{eq:equivalent-def-pseudo-Du-Bois}
{\rm Ext}^j(\underline{\cO}_{X,\pi}, \omega^\bullet_X) \to {\rm Ext}^j(\cO_X, \omega^\bullet_X) = \cH^j(\omega^\bullet_X).
\end{equation}
\begin{proposition}[\textup{cf.~\cite[Theorem 3.1]{Kovacs-Schwede-Smith}}]
With the notation of Setting \ref{setting:iso}, suppose that $X$ is Cohen-Macaulay. Then $X$ is pseudo-Du Bois if and only if $\pi_*\omega_Y(E) = \omega_X$.
\end{proposition}
\noindent In particular, an isolated Gorenstein singularity is pseudo-Du Bois if and only if it is log canonical.
\begin{proof}
We may assume that the dimension $d \geq 2$. Applying $R{\rm Hom}(-,\omega^\bullet_X)$ to the exact triangle (\ref{eq:0iso-exact-triangle}):
\[
R\pi_*\cO_Y(-E) \to \underline{\cO}_{X,\pi} \to k \xrightarrow{+1}
\]
yields the triangle:
\[
R{\rm Hom}(k,\omega^\bullet_X) \to R{\rm Hom}(\underline{\cO}_{X,\pi},\omega^\bullet_X) \to R{\rm Hom}(R\pi_*\cO_Y(-E),\omega^\bullet_X) \xrightarrow{+1}
\]
which simplifies to:
\[
k \to R{\rm Hom}(\underline{\cO}_{X,\pi},\omega^\bullet_X) \to R\pi_*\omega^\bullet_Y(E) \xrightarrow{+1}.
\]
Since $d \geq 2$, we have that 
\[
{\rm Ext}^{-d}(\underline{\cO}_{X,\pi},\omega^\bullet_X) \cong R^{-d}\pi_*\omega^\bullet_Y(E) =  \pi_*\omega_Y(E). 
\]
Therefore, as $X$ is Cohen-Macaulay, \eqref{eq:equivalent-def-pseudo-Du-Bois} implies that being pseudo-Du Bois is equivalent to:
\[
\pi_*\omega_Y(E) = \omega_X. \qedhere
\]  
\end{proof}

\subsection{Deligne--Du Bois complex along a resolution}

\begin{definition}\label{def:Du Bois Complex(isolated)}
With the notation of Setting \ref{setting:iso},
we define the \textit{Deligne--Du Bois complexes along $\pi$} by
\[
\underline{\Omega}^i_{X,\pi} \coloneqq R\pi_*\Omega^i_Y(\log E)(-E)\quad \text{ with }\quad \Omega^i_{X,\pi}\coloneqq  \pi_*\Omega^i_Y(\log E)(-E)
\]
for $i>0$. For $i=0$, we set $\underline{\Omega}^0_{X,\pi} \coloneqq \underline{\cO}_{X,\pi}$ and  $\Omega^0_{X,\pi} \coloneqq \cO_X$ (see Definition \ref{def:0thDu Bois cpx(isolated)}).
\end{definition}
{\noindent Note that $\Omega^i_{X,\pi}$ is torsion-free for all $i\geq 0$ essentially by definition.}

Next, we define a variant of pre-$k$-Du Bois singularities and pseudo-pre-$k$-Du Bois singularities in arbitrary characteristic:

\begin{definition}\label{def:pre-k-Du Bois(isolated)}
With the notation of Setting \ref{setting:iso}, pick an integer $k \geq 0$.
We say that $X$ is \emph{pre-$k$-Du Bois along $\pi$} if the natural map 
\[
\Omega^i_{X,\pi}[0] \to \underline{\Omega}^i_{X,\pi}
\]
is a quasi-isomorphism for every integer $0 \leq i \leq k$.
\end{definition}

In practice, the notion of pre-$k$-Du-Bois singularities along a resolution is a bit too restrictive in positive characteristic. Instead it is more natural to work with the following variant.

\begin{definition}\label{def:pseudo-pre-k-Du Bois(isolated)}
With the notation of Setting \ref{setting:iso}, pick an integer $k \geq 0$. 
We say that $X$ is \emph{pseudo-pre-$k$-Du Bois along $\pi$} if the map
\[
H^j_\m(\Omega^i_{X,\pi}) \to \mathbb{H}^j_\m(\underline{\Omega}^i_{X,\pi})
\]
is injective for every $0 \leq i \leq k$ and every $j \leq d-i$.
\end{definition}

\begin{lemma}\label{lem:independence of resolutions,i>0,iso}
 With the notation of Setting \ref{setting:iso}, the following statements hold.
\begin{enumerate}
    \item $\Omega^i_{X,\pi}$ in Definition \ref{def:Du Bois Complex(isolated)} does not depend on the choice of $\pi$ and, in fact, \[\Omega^i_{X,\pi}=\Omega^i_{X,\eh}.\]
    \item Suppose that the $R(d)$-condition holds. Then $\underline{\Omega}^i_{X,\pi}=\underline{\Omega}^i_{X,\eh}$.
    In particular, $\underline{\Omega}^i_{X,\pi}$ in Definition \ref{def:Du Bois Complex(isolated)}, pre-$k$-Du Bois along $\pi$ in Definition \ref{def:pre-k-Du Bois(isolated)}, and pseudo-pre-$k$-Du Bois along $\pi$ in Definition \ref{def:pseudo-pre-k-Du Bois(isolated)} do not depend on $\pi$.
    \item Suppose that $k$ is of characteristic zero. Then pseudo-pre-$k$-Du Bois is equivalent to pre-$k$-Du Bois.
\end{enumerate}
\end{lemma}
\begin{proof}
By Lemma \ref{lem:independence of resolutions,i=0,iso}, we may assume that $i>0$.
    We have the following diagram \[
\begin{tikzcd}
0\arrow[r] &\Omega^i_{X,\pi} \arrow[r]\arrow[d]  & \pi_{*}\Omega^i_{Y} \arrow[r]\arrow[d, equal]  & \pi_{*}\Omega^{i}_E/{\rm tors} \arrow[d, equal,"\text{Cor }\ref{cor:eh for SNC}"]  \\
0\arrow[r] & \Omega^i_{X,\eh} \arrow[r] & \pi_{*}\Omega^i_{Y,\eh} \arrow[r] & \pi_{*}\Omega^i_{E,\eh}
\end{tikzcd}
\]
where {we used the equality $\Omega^i_Y(\log E)(-E)=\ker(\Omega^i_{Y}\to \Omega^{i}_E/{\rm tors})$ for the upper exact sequence}, while the lower exact sequence exists by Proposition \ref{prop:blow-up} and the vanishing $\Omega^i_{x,\eh}=0$.
Thus $\Omega^i_{X,\pi}=\Omega^i_{X,\eh}$ and (1) holds.

We next prove (2).
Suppose that the $R(d)$-condition holds.
    By Corollary \ref{cor:blow-up triangle}, we have the following diagram
    \[
\begin{tikzcd}
R\pi_*\Omega^i_Y(\log E)(-E) \arrow[r] \arrow[d, equal] & R\pi_*\Omega^i_Y \arrow[r]  & R\pi_*\Omega^i_E/{\rm tors} \arrow[r, "+1"]  & {} \\
\underline{\Omega}^i_{X,x, {\eh}} \arrow[r,equal] & \underline{\Omega}^i_{X,{\eh}} \arrow[r] \arrow[u] & \underline{\Omega}^i_{x,{\eh}}=0 \arrow[u] \arrow[r, "+1"] & {}.
\end{tikzcd}
\]
Thus $\underline{\Omega}^i_{X,\pi}\coloneqq R\pi_*\Omega^i_Y(\log E)(-E)=\underline{\Omega}^i_{X,x,\eh}=\underline{\Omega}^i_{X,\eh}$.
Moreover, the map $\Omega^i_{X,\pi}=\mathcal{H}^0(\underline{\Omega}^i_{X,\pi})\to \underline{\Omega}^i_{X,\pi}$ is nothing but the natural map $\Omega^i_{X,\eh}=\mathcal{H}^0(\underline{\Omega}^i_{X,\eh})\to \underline{\Omega}^i_{X,\eh}$. Thus (2) holds.

Finally, (3) follows from Theorem \ref{thm:pseudo-pre-k-Du Bois=pre-k-Du Bois in char 0}.
\end{proof}

We state and prove the following proposition, which we shall need later on.
\begin{proposition} \label{prop:omega-pullback-reflexive}
With the notation of Setting \ref{setting:iso}, fix an integer $i>0$. Then the {reflexivisation} $\Omega^i_X \to \Omega^{[i]}_X$ factorises as follows 
\[\Omega^i_X \to {\Omega^i_{X,\pi}} \xhookrightarrow{(-)^{**}} \Omega^{[i]}_X.\]
In particular, if $\Omega^i_X$ is reflexive, then 
\[
\pi_*\Omega^i_Y(\log E)(-E) = \Omega^i_X.
\]
\end{proposition}
\begin{proof}
%{\commentbox{I made very small modifications below}}
We have a factorisation
$\Omega^i_X \xrightarrow{\pi^{*}} \pi_{*}\Omega^i_Y \xhookrightarrow{(-)^{**}} \Omega^{[i]}_X$
of the reflexivisation. 
We will show below that the pullback map $\pi^{*}\colon \Omega^i_X\to \pi_{*}\Omega^i_Y$ factors through $\Omega^i_X\to \Omega^i_{X,\pi}\coloneqq \pi_{*}\Omega^i_Y(\log E)(-E)$.
To this end, by considering the following diagram:
\[
\begin{tikzcd}
E \arrow[r] \arrow[d] & Y \arrow[d,"\pi"] \\
\{x\} \arrow[r] & X,
\end{tikzcd}
\]
we see that the composition $\Omega^i_X \xrightarrow{\pi^*} \pi_{*}\Omega^i_Y \to \pi_{*}\Omega^i_E\to \pi_{*}(\Omega^i_E / {\rm tors}) $ is zero for $i>0$ since $\{x\}$ has dimension zero. By the exact sequence
\[
0 \to \Omega^i_{X,\pi} \to \pi_{*}\Omega^i_Y \to \pi_{*}(\Omega^i_E / {\rm tors}), 
\]
we conclude that $\Omega^i_X\to \pi_{*}\Omega^i_Y$ factors through $\Omega^i_X\to \Omega^i_{X,\pi}$.
Finally, the injectivity of $\Omega^i_{X,\pi} \xhookrightarrow{(-)^{**}} \Omega^{[i]}_X$ follows from the torsion-freeness of $\Omega^i_{X,\pi}=\pi_{*}\Omega^i_Y(\log E)(-E)$. 
\end{proof}

\subsection{Cartier operators and $k$-$F$-injectivity}

We will define $k$-$F$-injectivity using Cartier operators. For now on, we will work entirely in characteristic $p>0$.
\begin{setting} \label{setting:iso-charp}
We make the same assumptions as in Setting \ref{setting:iso} but assume further that $k$ is {of characteristic $p>0$}.
\end{setting}

Consider the complex $F_*\Omega^{\bullet}_Y(\log E)(-E)$ from Section \ref{sss:duality}. We warn the reader again that 
\[
F_*\Omega^i_Y(\log E)(-E) \not \cong (F_*\Omega^i_Y(\log E)) \otimes_{\cO_Y} \cO_Y(-E).
\]
By definition and (\ref{eq:Hara-ses}), 
we have the short exact sequences
\begin{align} 
\label{eq:dlogEE}&0 \to Z\Omega^i_Y(\log E)(-E) \to F_*\Omega^i_Y(\log E)(-E) \xrightarrow{d} B\Omega^{i+1}_Y(\log E)(-E) \to 0 \\
\label{eq:Hara-sesEE}
&0 \to B\Omega^i_Y(\log E)(-E) \to Z\Omega^i_Y(\log E)(-E) \xrightarrow{C_{-(1/p)E}} \Omega^i_Y(\log E)(-E) \to 0.
\end{align}
We also have the following map, see (\ref{eq:dual-Cartier-Hara}):
\begin{equation} \label{eq:dual-Cartier-Hare-EE}
C^{-1}_{-(1/p)E}\colon \Omega^i_Y(\log E)(-E) \to G\Omega^i_Y(\log E)(-E).
\end{equation}
\begin{definition} \label{def:DDB-iso}
With the notation of Setting \ref{setting:iso-charp}, we define the boundaries and cycles in the Deligne Du-Bois complex along the resolution $\pi$ for $i>0$ as follows:
\begin{alignat*}{5}
B\underline{\Omega}^i_{X,\pi} &\coloneqq R\pi_*B\Omega^i_Y(\log E)(-E) \quad &&\text{ with }\quad  &B{\Omega}^i_{X,\pi}&\coloneqq \pi_*B\Omega^i_Y(\log E)(-E),  \\
Z\underline{\Omega}^i_{X,\pi} &\coloneqq R\pi_*Z\Omega^i_Y(\log E)(-E) \quad &&\text{ with }\quad  &Z{\Omega}^i_{X,\pi}&\coloneqq \pi_*Z\Omega^i_Y(\log E)(-E),\\
G\underline{\Omega}^i_{X,\pi} &\coloneqq R\pi_*G\Omega^i_Y(\log E)(-E) \quad &&\text{ with }\quad  &G{\Omega}^i_{X,\pi}&\coloneqq F_{*}{\Omega}^i_{X,\pi}/{B{\Omega}^i_{X,\pi}}. \qquad\qquad\quad\ \,
\end{alignat*}
For $i=0$, we set
$B\underline{\Omega}^0_{X,\pi} \coloneqq 0$, $B\Omega^0_{X,\pi} \coloneqq 0$, $Z\underline{\Omega}^0_{X,\pi} \coloneqq \underline{\cO}_{X,\pi}$, $Z\Omega^0_{X,\pi} \coloneqq \cO_X$, $G\underline{\Omega}^0_{X,\pi}\coloneqq F_{*}\underline{\cO}_{X,\pi}$, and $G\Omega^0_{X,\pi}\coloneqq F_{*}\cO_{X}$.
\end{definition}
\noindent {Note that we have a map $G\Omega^{i}_{X,\pi} \to \mathcal{H}^0(G\underline{\Omega}^i_{X,\pi}) =\pi_{*}G\Omega^i_Y(\log E)(-E)$, but this is not isomorphism in general.}

Consider the following maps \eqref{eq:Hara-sesEE} and \eqref{eq:dual-Cartier-Hare-EE}:
\begin{align*}
    &C_{-(1/p)E}\colon  Z\Omega^i_Y(\log E)(-E) \to \Omega^i_Y(\log E)(-E),\\
    &C^{-1}_{-(1/p)E}\colon \Omega^i_Y(\log E)(-E)\to G\Omega^i_Y(\log E)(-E).
\end{align*}

\begin{definition}\label{def:Cartier operator along a resolutuon}
    With the notation of Setting \ref{setting:iso-charp}, we define the \emph{Cartier operators along $\pi$} by applying $R\pi_*$ to the above maps so that for {$i>0$} we get\footnote{Specifically, $\underline{C}_{\pi}\coloneqq R\pi_{*}C_{-(1/p)E}$, $C_{\pi}\coloneqq \mathcal{H}^0(\underline{C}_{\pi})$, and $\underline{C}^{-1}_{\pi}\coloneqq R\pi_{*}C^{-1}_{-(1/p)E}$.}:
    \[\underline{C}_{\pi} \colon Z\underline{\Omega}^i_{X,\pi} \to \underline{\Omega}^i_{X,\pi},\quad  C_{\pi} \colon Z\Omega^i_{X,\pi} \to \Omega^i_{X,\pi},\quad\text{and}\quad \underline{C}^{-1}_{\pi} \colon \underline{\Omega}^i_{X,\pi} \to G\underline{\Omega}^i_{X,\pi}. \]
{For $i=0$, we set $\underline{C}_{\pi}=\mathrm{id}_{\underline{\sO}_{X,\pi}}$,  $C_{\pi}=\mathrm{id}_{\sO_X}$, and $\underline{C}^{-1}_{\pi}= \underline{\sO}_{X,\pi} \xrightarrow{F} F_{*}\underline{\sO}_{X,\pi}$ (cf.\  \eqref{eq:Frobenius of  O_{pi}}).}\\

Further, if $C_{\pi} \colon Z\Omega^i_{X,\pi} \to \Omega^i_{X,\pi}$ is surjective, then we define $C^{-1}_{\pi}$ as the following composition
\begin{equation} \label{eq:C-1pi}
C^{-1}_{\pi} \colon \Omega^i_{X,\pi} \underset{\cong}{\xleftarrow{C_{\pi}}} \frac{Z\Omega^i_{X,\pi}}{B\Omega^i_{X,\pi}} \lhook\joinrel\longrightarrow  \frac{F_*\Omega^i_{X,\pi}}{B\Omega^i_{X,\pi}} =G\Omega^i_{X,\pi}.
\end{equation}
\end{definition}
Note that by pushing forward the short exact sequences \eqref{eq:dlogEE} and \eqref{eq:Hara-sesEE} we get the following sequences:
\begin{align} 
\label{eq:pidlogEE} &0 \to Z\Omega^i_{X, \pi} \to F_*\Omega^{i}_{X,\pi} \xrightarrow{d} {B\Omega^{i+1}_{X,\pi}\, (\subseteq F_*\Omega^{i+1}_{X,\pi})}\\
\label{eq:Cpi-iso-sequence}
&0 \to B\Omega^i_{X,\pi} \to Z\Omega^i_{X,\pi} \xrightarrow{C_{\pi}} \Omega^i_{X,\pi}.
\end{align}
The second sequence was implicitly used in \eqref{eq:C-1pi}.

The maps $C^{-1}_{\pi}$ and $\mathcal{H}^0(\underline{C}^{-1}_\pi)$ agree with each other outside of $\{x\}$. Thus, when $C_{\pi} \colon Z\Omega^i_{X,\pi} \to \Omega^i_{X,\pi}$ is surjective, we obtain the following commutative diagram:
\begin{equation} \label{eq:diagram-main-iso}
\begin{tikzcd}
{} & \Omega^i_{X,\pi} \arrow[r]\arrow[ld,hookrightarrow,"C^{-1}_{\pi}"']\arrow[d,hookrightarrow,"\mathcal{H}^0(\underline{C}^{-1}_\pi)"] & \underline{\Omega}^i_{X,\pi} \arrow[d, "\underline{C}^{-1}_\pi"] \\
G\Omega^i_{X,\pi} \arrow[r,hookrightarrow] & \mathcal{H}^0(G\underline{\Omega}^i_{X,\pi})\arrow[r] & G\underline{\Omega}^i_{X,\pi}.
\end{tikzcd}
\end{equation}

In what follows we show that if $\underline{\Omega}^i_{X,\pi}$ are concentrated in degree $0$, then so are $B\underline{\Omega}^i_{X,\pi}$,  $Z\underline{\Omega}^i_{X,\pi}$, and $G\underline{\Omega}^i_{X,\pi}$.
\begin{proposition} \label{prop:iso-k-Du-Bois-C-surjective}
With the notation of Setting \ref{setting:iso-charp}, suppose that $X$ is pre-$k$-Du Bois along $\pi$ for some integer $k \geq 0$. Then
\begin{enumerate}
    \item $B\underline{\Omega}^i_{X,\pi} = B\Omega^{i}_{X,\pi}[0]$,\\[-0.9em]
    \item $Z\underline{\Omega}^i_{X,\pi} = Z\Omega^{i}_{X,\pi}[0]$,\\[-0.9em]
    \item $G\underline{\Omega}^i_{X,\pi}  = G\Omega^{i}_{X,\pi}[0]$,\\[-0.9em]
    \item $C_{\pi} \colon Z\Omega^i_{X,\pi} \to \Omega^i_{X,\pi}$ is surjective
\end{enumerate}
for every integer $0 \leq i \leq k$.
\end{proposition}
\begin{proof}
We start by proving (1), (2), and {(4)}. These statements are clear for $i=0$. Thus by induction, we may assume that:
\begin{enumerate}
    \item $B\underline{\Omega}^{i-1}_{X,\pi} = B\Omega^{i-1}_{X,\pi}[0]$,\\[-0.9em]
    \item $Z\underline{\Omega}^{i-1}_{X,\pi} = Z\Omega^{i-1}_{X,\pi}[0]$, and\\[-0.9em]
    \item $C_{\pi} \colon Z\Omega^{i-1}_{X,\pi} \to \Omega^{i-1}_{X,\pi}$ is surjective
\end{enumerate}
in order to show (1), (2), and {(4)} for $i>0$. By pushing forward the exact sequence \eqref{eq:dlogEE}
we get the exact triangle
\[
Z\Omega^{i-1}_{X,\pi} \to F_*\Omega^{i-1}_{X,\pi} \to B\underline{\Omega}^i_{X,\pi} \xrightarrow{+1}.
\]
In particular $B\underline{\Omega}^i_{X,\pi} = B\Omega^{i}_{X,\pi}[0]$.
Subsequently, by applying $R\pi_{*}$ to \eqref{eq:Hara-sesEE},
we get
\[
B\Omega^i_{X,\pi} \to Z\underline{\Omega}^i_{X,\pi} \xrightarrow{\underline{C}_{\pi}} \Omega^i_{X,\pi} \xrightarrow{+1}.
\]
Therefore, $Z\underline{\Omega}^{i}_{X,\pi} = Z\Omega^{i}_{X,\pi}[0]$, and $C_\pi \colon Z\Omega^{i}_{X,\pi} \to \Omega^{i}_{X,\pi}$ is surjective.

Finally, by (1) and the assumption $\underline{\Omega}^i_{X,\pi} = \Omega^i_{X,\pi}[0]$,
we obtain (3) as follows:
\[G\underline{\Omega}^i_{X,\pi}  = \mathrm{Cone}(B\underline{\Omega}^i_{X,\pi}\to F_{*}\underline{\Omega}^i_{X,\pi})=\frac{F_{*}\Omega^i_{X,\pi}}
{B\Omega^i_{X,\pi}}[0]=G\Omega^{i}_{X,\pi}[0]. \qedhere \]
\end{proof}

We often compare the above definitions with their reflexified versions.
\begin{proposition} \label{prop:pi-iso-reflexive-agree-basic} With the notation of Setting \ref{setting:iso-charp}, suppose that $\Omega^i_{X, \pi} = \Omega^{[i]}_X$ for every integer $0 \leq i \leq k$. Then the exact sequence $0\to B\Omega^i_{X,\pi} \to Z\Omega^i_{X,\pi} \xrightarrow{C_{\pi}} \Omega^i_{X,\pi}$ agrees with $0\to B\Omega^{[i]}_X \to Z\Omega^{[i]}_X \xrightarrow{C} \Omega^{[i]}_{X}$.
\end{proposition}
\begin{proof}
%\commentbox{I slightly modified below}
Recall that the exact sequence $0\to B\Omega^i_{X,\pi} \to Z\Omega^i_{X,\pi} \xrightarrow{C_{\pi}} \Omega^i_{X,\pi}$ coincides with 
$0\to B\Omega^{[i]}_X \to Z\Omega^{[i]}_X \xrightarrow{C} \Omega^{[i]}_{X}$ outside of $\{x\}$.
Moreover, $Z\Omega^{[i]}_X$ and $B\Omega^{[i]}_X$ are reflexive by definition.
Since two morphisms between torsion-free coherent sheaves on an integral scheme coincide if they agree on the generic point, it suffices to show that $Z\Omega^{i}_{X,\pi}$ and $B\Omega^{i}_{X,\pi}$ are reflexive.
Fix $0\leq i\leq k$.
Consider the following sequence \eqref{eq:pidlogEE}:
\[
0 \to Z\Omega^i_{X, \pi} \to F_*\Omega^{i}_{X,\pi} \xrightarrow{d} F_*\Omega^{i+1}_{X,\pi}.
\]
Since $F_*\Omega^{i}_{X,\pi}$ is reflexive and $F_*\Omega^{i+1}_{X,\pi}$ is torsion-free, we get that $Z\Omega^i_{X,\pi}$ is reflexive.
Next, consider the following sequence \eqref{eq:Cpi-iso-sequence}:
\[
0 \to B\Omega^i_{X,\pi} \to Z\Omega^i_{X,\pi} \xrightarrow{C_{\pi}} \Omega^i_{X,\pi}.
\]
Since $Z\Omega^i_{X,\pi}$ and $\Omega^i_{X,\pi}$ are reflexive, $B\Omega^i_{X, \pi}$ is reflexive too, concluding the proof.
\end{proof}

\subsubsection{$k$-$F$-injective singularities}
\label{ss:definition-pre-k-iso-and-iterated}
\begin{definition} \label{def:prekFinj-iso}
With the notation of Setting \ref{setting:iso-charp} pick an integer $k \geq 0$. We say that $X$ is \emph{pre-$k$-F-injective} if for every integer $0 \leq i \leq k$:
\begin{enumerate}
\item $C_{\pi} \colon Z\Omega^i_{X,\pi} \to \Omega^i_{X,\pi}$ is surjective, and\\[-0.9em]
\item $C_\pi^{-1} \colon H^j_\m(\Omega^{i}_{X,\pi}) \to H^j_\m(G\Omega^{i}_{X,\pi})$ is injective for all $0 \leq j \leq d - i$.
\end{enumerate}
\end{definition}

\begin{proposition}
    With the notation of Setting \ref{setting:iso-charp},
    the pre-$k$-$F$-injectivity of $X$ is independent of the choice of $\pi$. 
\end{proposition}
\begin{proof}
    By Lemma \ref{lem:independence of resolutions,i>0,iso}(1), $\Omega^i_{X,\pi}$ does not depend on $\pi$, and thus $d\colon F_{*}\Omega^i_{X,\pi}\to F_{*}\Omega^{i+1}_{X,\pi}$ does not depend on $\pi$ either. 
    Recall that two morphisms between torsion-free coherent sheaves on an integral scheme coincide if they agree on the generic point.
    Thus, \[ Z\Omega^i_{X,\pi}=\ker(F_{*}\Omega^i_{X,\pi}\xrightarrow{d} F_{*}\Omega^{i+1}_{X,\pi}) \] does not depend on the choice of $\pi$. 
    
    Similarly, $C\colon Z\Omega^i_{X,\pi}\to \Omega^i_{X,\pi}$, $B\Omega^i_{X,\pi}=\ker(Z\Omega^i_{X,\pi}\xrightarrow{C} \Omega^i_{X,\pi})$, $G\Omega^{i}_{X,\pi}$, and
    $C^{-1}_\pi \colon \Omega^{i}_{X,\pi}\to G\Omega^{i}_{X,\pi}$ do not depend on the choice of $\pi$ either.
\end{proof}

In many situations $\Omega^i_{X,\pi}$ agrees with $\Omega^{[i]}_X$; for example when we reduce a rational singularity modulo $p \gg 0$ \cite[Corollary 1.12]{KS21}. In this case, pre-$F$-injectivity can be defined through the reflexified Cartier operator.

\begin{proposition} \label{prop:pi-iso-reflexive-agree} With the notation of Setting \ref{setting:iso-charp}, suppose that $\Omega^i_{X, \pi} = \Omega^{[i]}_X$ for every integer $0 \leq i \leq k$. Then $X$ is pre-$k$-$F$-injective if and only if for every integer $0 \leq i \leq k$:
\begin{enumerate}
\item $C \colon Z\Omega^{[i]}_{X} \to \Omega^{[i]}_{X}$ is surjective, and\\[-0.9em]
\item $C^{-1} \colon H^j_\m(\Omega^{[i]}_{X}) \to H^j_\m(G\Omega^{[i]}_{X})$ is injective for all integers $0 \leq j \leq d - i$.
\end{enumerate}
\end{proposition}
\noindent We remind the reader that $G\Omega^{[i]}_X = \frac{F_*\Omega^{[i]}_{X}}{B\Omega^{[i]}_{X}}$ is not necessarily equal to the reflexivisation of $G\Omega^i_X$ (cf.\ Proposition \ref{prop:when-omega-reflexive-basic-properties}).
\begin{proof}
This is automatic by Proposition \ref{prop:pi-iso-reflexive-agree-basic}.
\end{proof}

\begin{corollary}\label{cor:k-F-inj of reflexvie case(isorated)}
With the notation of Setting \ref{setting:iso-charp}, suppose that
$\Omega^i_{X}$ is reflexive for every integer $0 \leq i \leq k$. Then $X$ is pre-$k$-$F$-injective if and only if it is $k$-$F$-injective in the sense of Definition \ref{def:intro-kFinj}. 
\end{corollary}
\begin{proof}
By Proposition \ref{prop:omega-pullback-reflexive}, we have the following natural factorisation for $i>0$:
\[
\Omega^i_X \to \Omega^i_{X, \pi} \hookrightarrow \Omega^{[i]}_X.
\]
Therefore, $\Omega^i_X \cong  \Omega^i_{X, \pi} \cong \Omega^{[i]}_X$ for $0 < i\leq k$ by assumption.
Moreover, the isomorphism $\sO_X\cong \sO_{X,\pi}$ is always valid.
Now, we can conclude immediately by Proposition \ref{prop:pi-iso-reflexive-agree} and Proposition \ref{prop:when-omega-reflexive-basic-properties}.
\end{proof}

\subsection{Higher Cartier operators and $k$-$F$-injectivity}

We generalise the above definitions to the case of higher Cartier operators as that will be needed in the proof of the main theorem.

Consider the complex $F^n_*\Omega^{\bullet}_Y(\log E)(-E)$ constructed by taking $\Delta=-\frac{1}{p^n}E$ in \eqref{eq:logDeltadeRhamHigher}. By \eqref{eq:Hara-ses(iterated)}, 
we have the short exact sequence
\begin{equation} 
\label{eq:Hara-sesEE-higher}
0 \to B_n\Omega^i_Y(\log E)(-E) \to Z_n\Omega^i_Y(\log E)(-E) \xrightarrow{\tiny C_{n,-(1/p^n)E}} \Omega^i_Y(\log E)(-E) \to 0.
\end{equation}
We also have the following map, see (\ref{eq:dual-Cartier-Hara-higher}):
\begin{equation} \label{eq:dual-Cartier-Hare-EE-higher}
C^{-1}_{-(1/p^n)E}\colon \Omega^i_Y(\log E)(-E) \to G_n\Omega^i_Y(\log E)(-E).
\end{equation}
\vspace{-0.8em}
\begin{definition} \label{def:DDB-iso-higher}
With the notation of Setting \ref{setting:iso-charp}, define for $i>0$:
\begin{alignat*}{5}
B_n\underline{\Omega}^i_{X,\pi} &\coloneqq R\pi_*B_n\Omega^i_Y(\log E)(-E) \quad &&\text{ with }\quad  &B_n{\Omega}^i_{X,\pi}&\coloneqq \pi_*B_n\Omega^i_Y(\log E)(-E),  \\
Z_n\underline{\Omega}^i_{X,\pi} &\coloneqq R\pi_*Z_n\Omega^i_Y(\log E)(-E) \quad &&\text{ with }\quad  &Z_n{\Omega}^i_{X,\pi}&\coloneqq \pi_*Z_n\Omega^i_Y(\log E)(-E), \\
G_n\underline{\Omega}^i_{X,\pi} &\coloneqq R\pi_*G_n\Omega^i_Y(\log E)(-E) \quad &&\text{ with }\quad  &G_n{\Omega}^i_{X,\pi}&\coloneqq {F^n_{*}{\Omega}^i_{X,\pi}}/{B_n{\Omega}^i_{X,\pi}}.  \qquad\qquad\quad\
\end{alignat*}
For $i=0$, we set
$B_n\underline{\Omega}^0_{X,\pi} \coloneqq 0$, $B_n\Omega^0_{X,\pi} \coloneqq 0$, $Z_n\underline{\Omega}^0_{X,\pi} \coloneqq \underline{\cO}_{X,\pi}$, $Z_n\Omega^0_{X,\pi} \coloneqq \cO_X$, {$G_n\underline{\Omega}^0_{X,\pi} \coloneqq F_{*}\underline{\cO}_{X,\pi}$, and $G_n\Omega^0_{X,\pi} \coloneqq F_{*}\cO_X$.}
\end{definition}

\noindent Consider the following maps \eqref{eq:Hara-sesEE-higher} and \eqref{eq:dual-Cartier-Hare-EE-higher}:
\begin{align*}
    &C_{-(1/p^n)E}\colon  Z_n\Omega^i_Y(\log E)(-E) \to \Omega^i_Y(\log E)(-E) \,\,\text{and}\,\,\\
    &C^{-1}_{-(1/p^n)E}\colon \Omega^i_Y(\log E)(-E)\to G_n\Omega^i_Y(\log E)(-E).
\end{align*}

\begin{definition}\label{def:Cartier operator along a resolutuon-higher}
    With the notation with Setting \ref{setting:iso-charp} we define the \emph{higher Cartier operators along $\pi$} by applying $R\pi_*$ to the above maps so that for $i>0$ we get\footnote{Specifically, $\underline{C}_{n,\pi}\coloneqq R\pi_{*}C_{-(1/p^n)E}$, $C_{n,\pi}\coloneqq \mathcal{H}^0(\underline{C}_{n,\pi})$, and $\underline{C}^{-1}_{n,\pi}\coloneqq R\pi_{*}C^{-1}_{-(1/p^n)E}$.}:
    \[\underline{C}_{n,\pi} \colon Z_n\underline{\Omega}^i_{X,\pi} \to \underline{\Omega}^i_{X,\pi}, \quad
C_{n,\pi} \colon Z_n\Omega^i_{X,\pi} \to \Omega^i_{X,\pi}, \quad \text{ and }\quad
\underline{C}^{-1}_{n,\pi} \colon \underline{\Omega}^i_{X,\pi} \to G_n\underline{\Omega}^i_{X,\pi}.\]
For $i=0$, set $\underline{C}_{n,\pi}=\mathrm{id}_{\underline{\sO}_{X,\pi}}$, $C_{n,\pi}=\mathrm{id}_{\sO_X}$,
and $\underline{C}^{-1}_{n,\pi}=F^n\colon \underline{\sO}_{X,\pi} \to F^n_{*}\underline{\sO}_{X,\pi}$ (\eqref{eq:Frobenius of  O_{pi}}).
\end{definition}

Note that by pushing forward the short exact sequence \eqref{eq:Hara-sesEE-higher} we get:
\begin{equation}
0 \to B_n\Omega^i_{X,\pi} \to Z_n\Omega^i_{X,\pi} \xrightarrow{C_{n,\pi}} \Omega^i_{X,\pi}.
\end{equation}

\begin{remark}
An analogous construction to the one above yields maps:
\[
C_{n,m,\pi} \colon Z_{n}\Omega^i_{X,\pi} \to Z_{m}\Omega^i_{X,\pi}
\]
for all integers $n \geq m > 0$.
\end{remark}

\subsubsection{$k$-$F$-injectivity for iterated Cartier operators}
As in the case of $F$-injectivity and the Frobenius morphism, we would like to iterate the formation of $C^{-1}$ in Definition \ref{def:prekFinj-iso}. To this end, we first study assumption (1) therein.

\begin{lemma}\label{lem:surjectivity of iterated Cartier operator}
    With the notation of Setting \ref{setting:iso-charp}, the following are equivalent:
    \begin{enumerate}
        \item $C_{\pi}\colon Z\Omega^i_{X,\pi} \to \Omega^i_{X,\pi}$ is surjective.
        \item $C_{n,n-1,\pi}\colon Z_n\Omega^i_{X,\pi} \to Z_{n-1}\Omega^i_{X,\pi}$ is surjective for every integer $n>0$.
        \item $C_{n,\pi}\colon Z_n\Omega^i_{X,\pi} \to \Omega^i_{X,\pi}$ is surjective for every integer $n>0$.
    \end{enumerate}
\end{lemma}
\begin{proof}
    We only show that (1) implies (2) as other implications are easy.
    By pushing forward the pullback diagram
    \[
    \begin{tikzcd}[column sep = huge]
Z_n\Omega^i_Y(\log E)(-E)\arrow[rd, phantom, "\usebox\pullback" , very near start, yshift=-0.3em, xshift=-0.6em, color=black] \arrow[d,"C_{n,n-1}"]\arrow[r,hookrightarrow] & F_* Z_{n-1}\Omega^i_Y(\log E)(-E)\arrow[d,"C_{n-1,n-2}"]\\
Z_{n-1}\Omega^i_Y(\log E)(-E)   \arrow[r,hookrightarrow] & F_{*}Z_{n-2}\Omega^i_Y(\log E)(-E),  
\end{tikzcd}
\]
we obtain the pullback diagram (here we are using that $\pi_*$ is left exact)
\[
    \begin{tikzcd}[column sep = huge]
Z_n\Omega^i_{X,\pi}\arrow[rd, phantom, "\usebox\pullback" , very near start, yshift=-0.3em, xshift=-0.6em, color=black] \arrow[d,"C_{n,n-1,\pi}"]\arrow[r,hookrightarrow] & F_* Z_{n-1}\Omega^i_{X,\pi}\arrow[d,"C_{n-1,n-2,\pi}"]\\
Z_{n-1}\Omega^i_{X,\pi}   \arrow[r,hookrightarrow] & F_{*}Z_{n-2}\Omega^i_{X,\pi}. 
\end{tikzcd}
\]
Then we obtain the surjectivity of $C_{n,n-1,\pi}\colon Z_{n}\Omega^i_{X,\pi}\to Z_{n-1}\Omega^i_{X,\pi}$ from that of
\[C_{n-1,n-2,\pi}\colon Z_{n-1}\Omega^i_{X,\pi}\to Z_{n-2}\Omega^i_{X,\pi} \]
which can be assumed by induction.
\end{proof}

\begin{remark} \label{rem:surjectivity of iterated Cartier operator}
Similarly, we have the following pullback diagram
\begin{equation}\label{diagram:BZpullback}
    \begin{tikzcd}[column sep = huge]
B_n\Omega^i_{X,\pi}\arrow[rd, phantom, "\usebox\pullback" , very near start, yshift=-0.3em, xshift=-0.6em, color=black] \arrow[d,"C_{n,n-1,\pi}"]\arrow[r,hookrightarrow] & Z_{n}\Omega^i_{X,\pi}\arrow[d,twoheadrightarrow,"C_{n,n-1,\pi}"]\\
B_{n-1}\Omega^i_{X,\pi}   \arrow[r,hookrightarrow] & Z_{n-1}\Omega^i_{X,\pi},  
\end{tikzcd}
\end{equation}
and thus $C_{n,n-1,\pi}\colon B_n\Omega^i_{X,\pi} \to B_{n-1}\Omega^i_{X,\pi}$ is surjective for every integer $n>0$ if the assumptions of Lemma \ref{lem:surjectivity of iterated Cartier operator} are satisfied.
\end{remark}

In what follows, we construct the $n$-composed inverse Cartier operator $C^{-1}_{n,\pi}\colon \Omega^i_{X,\pi}\to G_{n}\Omega^i_{X,\pi}$ assuming that $C_{\pi}\colon Z\Omega^i_{X,\pi} \to \Omega^i_{X,\pi}$ is surjective.
\begin{construction} \label{cons-c-1-n-iso}
Suppose that $C_{\pi}\colon Z\Omega^i_{X,\pi} \to \Omega^i_{X,\pi}$ is surjective.
Note that then $C_{n,\pi}\colon Z_n\Omega^i_{X,\pi} \to \Omega^i_{X,\pi}$
is surjective by Lemma \ref{lem:surjectivity of iterated Cartier operator}. 
We define $C^{-1}_{n,\pi}\colon \Omega^i_{X,\pi}\to G_n\Omega^i_{X,\pi}$ as the composition 
\[
C^{-1}_{n,\pi}\colon \Omega^i_{X,\pi}\underset{\cong}{\xleftarrow{C_{n,\pi}}} \frac{Z_n\Omega^i_{X,\pi}}{B_n\Omega^i_{X,\pi}} \lhook\joinrel\longrightarrow \frac{F_{*}^n\Omega^i_{X,\pi}}{B_n\Omega^i_{X,\pi}}=G_n\Omega^i_{X,\pi}.
\]
{As in \eqref{eq:diagram-main-iso}, we obtain the following commutative diagram
\begin{equation} \label{eq:diagram-main-iso(iterated)}
\begin{tikzcd}
{} & \Omega^i_{X,\pi} \arrow[r]\arrow[ld,hookrightarrow,"C^{-1}_{n,\pi}"']\arrow[d,hookrightarrow,"\mathcal{H}^0(\underline{C}^{-1}_{n,\pi})"] & \underline{\Omega}^i_{X,\pi} \arrow[d, "\underline{C}^{-1}_\pi"] \\
G_n\Omega^i_{X,\pi} \arrow[r,hookrightarrow] & \mathcal{H}^0(G_n\underline{\Omega}^i_{X,\pi})\arrow[r] & G_n\underline{\Omega}^i_{X,\pi}.
\end{tikzcd}
\end{equation}}

Recall that we have isomorphisms {(cf.~ \eqref{diagram:BZpullback})}: \[
\Omega^i_{X,\pi} \underset{\cong}{\xleftarrow{C_{\pi}}} \frac{Z\Omega^i_{X,\pi}}{B\Omega^i_{X,\pi}}\underset{\cong}{\xleftarrow{C_{2,1,\pi}}}\cdots
\underset{\cong}{\xleftarrow{C_{n,n-1,\pi}}} \frac{Z_n\Omega^i_{X,\pi}}{B_n\Omega^i_{X,\pi}}.\] Thus,
we have a factorisation
\begin{equation}\label{eq:factirisation of iterated inverse Cartier operator}
    C^{-1}_{n,\pi} \colon \Omega^{i}_{X,\pi} \xrightarrow{C^{-1}_{\pi}} G\Omega^{i}_{X,\pi} = G_1\Omega^{i}_{X,\pi} \xrightarrow{C^{-1}_{2,1,\pi}} G_2\Omega^{i}_{X,\pi} \longrightarrow \cdots \xrightarrow{C^{-1}_{n,n-1,\pi}} G_n\Omega^{i}_{X,\pi}.
\end{equation}
\end{construction}

\begin{lemma} \label{lem:iso-pre-k-cn-1n}
With the notation of Setting \ref{setting:iso-charp}, fix integers $k \geq 0$, $n > 0$, and suppose that $X$ is pre-$k$-$F$-injective. Then
\[
C^{-1}_{n-1,n,\pi} \colon H^j_\m(G_{n-1}\Omega^{i}_{X,\pi}) \to H^j_\m(G_{n}\Omega^{i}_{X,\pi})
\]
is injective for all integers $0 \leq i \leq k$ and all integers $0 \leq j \leq d - i$.
\end{lemma}

\begin{proof}
Consider the following diagram:
\begin{equation}\label{C between G_n}
\begin{tikzcd}
0 \arrow[r] & B_{n-1}\Omega^{i}_{X,\pi} \arrow[r] \arrow{d}{C^{-1}_{n,n-1,\pi}}[swap]{\cong} & F^{n-1}_*\Omega^{i}_{X,\pi} \arrow[r] \arrow[d, "C^{-1}_{n,n-1,\pi}"] & G_{n-1}\Omega^{i}_{X,\pi} \arrow[r] \arrow[d, "{C^{-1}_{n,n-1,\pi}}"] & 0 \\
0 \arrow[r] & \frac{B_n\Omega^{i}_{X,\pi}}{F^{n-1}_*B\Omega^{i}_{X,\pi}} \arrow[r] & F^{n-1}_*\left(\frac{F_*\Omega^{i}_{X,\pi}}{B\Omega^{i}_{X,\pi}}\right) \arrow[r] & G_{n}\Omega^{i}_{X,\pi} \arrow[r] & 0.
\end{tikzcd}
\end{equation}
By {the surjectivity of $C_{n,n-1,\pi}\colon B_{n}\Omega^{i}_{X,\pi} \to B_{n-1}\Omega^{i}_{X,\pi}$ (see Remark \ref{rem:surjectivity of iterated Cartier operator})}, the leftmost vertical arrow is an isomorphism. After applying local cohomology $H^j_\m$, the middle vertical arrow is an injection. Hence, by carefully tracing through this diagram (or applying the \emph{four lemma}), we also get that the rightmost vertical arrow is an injection on local cohomology.
\end{proof}

\begin{proposition} \label{prop:iso-iterated-def-pre-k-F-injective}
With the notation of Setting \ref{setting:iso-charp} fix an integer $k \geq 0$ and an integer $n>0$. Then $X$ is pre-$k$-F-injective if and only if for every integer $0 \leq i \leq k$:
\begin{itemize}
    \item $C_{n, \pi} \colon Z_n\Omega^i_{X,\pi} \to \Omega^i_{X,\pi}$ is surjective, and \\[-0.9em]
    \item $C^{-1}_{n,\pi} \colon H^j_\m(\Omega^{i}_{X,\pi}) \to H^j_\m(G_{n}\Omega^{i}_{X,\pi})$ is injective for all integers $0 \leq j \leq d - i$.
\end{itemize}
\end{proposition}
\begin{proof}
 The implication from right to left is clear by Lemma \ref{lem:surjectivity of iterated Cartier operator} and the factorisation \eqref{eq:factirisation of iterated inverse Cartier operator}.
 The implication from left to right follows from Lemma \ref{lem:surjectivity of iterated Cartier operator}, the same factorisation \eqref{eq:factirisation of iterated inverse Cartier operator}, and Lemma \ref{lem:iso-pre-k-cn-1n}.
\end{proof}

The proposition above allows for a nice reinterpretation of $k$-F-injectivity.
\begin{definition}
With the notation of Setting \ref{setting:iso-charp}, define
\begin{align*}
\Omega^{i,{\rm perf}}_{X,\pi} \coloneqq \varinjlim_n G_n\Omega^{i}_{X,\pi} \quad \text{ and } \quad
Z\Omega^{i,{\rm perf}}_{X,\pi} \coloneqq \varprojlim_n Z_n\Omega^{i}_{X,\pi}
\end{align*}
where the colimit is taken with respect to
\[
C^{-1}_{{n,n-1},\pi} \colon G_{n-1}\Omega^{i}_{X,\pi} \to G_n\Omega^{i}_{X,\pi}
\]
and the limit is taken with respect to 
\[
C_{n,n-1,\pi} \colon Z_{n}\Omega^{i}_{X,\pi} \to Z_{n-1}\Omega^{i}_{X,\pi}.
\]
We denote the induced maps from and to $\Omega^i_{X,\pi}$ by
\[
C^{-1}_{\perf,\pi} \colon H^j_\m(\Omega^{i}_{X,\pi}) \to H^j_\m(\Omega^{i, {\rm perf}}_{X,\pi}) \quad \text{ and } \quad  C_{\perf, \pi} \colon Z\Omega^{i,\perf}_{X,\pi} \to \Omega^i_{X,\pi}.
\]
\end{definition}
\begin{proposition} \label{prop:iso-iterated-def-pre-k-F-injective-perfection}
With the notation of Setting \ref{setting:iso-charp}, fix an integer $k \geq 0$ and an integer $n>0$. Then $X$ is pre-$k$-F-injective if and only if for every integer $0 \leq i \leq k$:
\begin{itemize}
    \item $C_{\perf, \pi} \colon Z\Omega^{i,\perf}_{X,\pi} \to \Omega^i_{X,\pi}$ is surjective, and \\[-0.9em]
    \item $C^{-1}_{\perf,\pi} \colon H^j_\m(\Omega^{i}_{X,\pi}) \to H^j_\m(\Omega^{i, {\rm perf}}_{X,\pi})$ is injective for all integers $0 \leq j \leq d - i$.
\end{itemize}
\end{proposition}
\begin{proof}
This follows immediately from Lemma \ref{lem:surjectivity of iterated Cartier operator} and Proposition \ref{prop:iso-iterated-def-pre-k-F-injective} as local cohomology commutes with colimits.
\end{proof}

\subsection{$k$-$F$-injective implies $k$-Du Bois} \label{ss:towards-proof-1-iso}
In this subsection, we prove that pre-$k$-$F$-injective singularites are pseudo-pre-$k$-Du Bois under some additional assumptions (Theorem \ref{thm:intro-main}).

With the notation of Setting \ref{setting:iso-charp}, 
we have the following exact triangle:
\begin{equation}\label{eq:G}
G_n\Omega^{i}_{X,\pi} \xrightarrow{\text{nat.}} G_n\underline{\Omega}^i_{X,\pi} \xrightarrow{r} \mathrm{Cone}\Big(G_n\Omega^{i}_{X,\pi} \to G_n\underline{\Omega}^i_{X,\pi} \Big) \xrightarrow{+1},
\end{equation}
where the first map is the composition $G_n\Omega^{i}_{X,\pi}\hookrightarrow \mathcal{H}^0(G_n\underline{\Omega}^i_{X,\pi}) \to G_n\underline{\Omega}^i_{X,\pi}$
of the natural maps.

\begin{lemma} \label{lemma:iso-key-factorisation-main-theorem}
With the notation of Setting \ref{setting:iso-charp}, fix an integer {$k > 0$}, and let $A$ be an ample anti-effective exceptional $\bQ$-divisor on $Y$ such that $\rdown{A}=-E$. Assume that 
\[
R^{>0}\pi_{*}\Omega^i_Y(\log E)(p^rA) = 0
\]
for all $0 \leq i \leq k-1$ and $r \geq 0$. Then for every $n \gg 0$ and $0<i\leq k$, the composition
    \begin{equation*}
    \underline{\Omega}^i_{X,\pi} \xrightarrow{\underline{C}^{-1}_{n,\pi}} G_n\underline{\Omega}^i_{X,\pi} \xrightarrow{r} \mathrm{Cone}\Big(G_n\Omega^{i}_{X,\pi} \xrightarrow{\text{nat.}} G_n\underline{\Omega}^i_{X,\pi}  \Big)
    \end{equation*}
    is zero, where the map $r$ is defined in \eqref{eq:G}.
\end{lemma}
\begin{proof} 
{Fix an integer $0<i\leq k$
Take $n\gg 0$ such that $A\leq -(1/p^n)E$.
By taking $\Delta'\coloneqq A$ and $\Delta\coloneqq -(1/p^n)E$ in \eqref{eq:Hara-ses(iterated) commuative},
we get the following factorisation
\begin{multline*}
    \underline{C}^{-1}_{n,\pi} \colon \underline{\Omega}^i_{X,\pi}\coloneqq R\pi_*\Omega^i_Y(\log E)(-E)=R\pi_*\Omega^i_Y(\log E)(A)\\
    \xrightarrow{R\pi_{*}C^{-1}_{n,A}} R\pi_*G_n\Omega^i_Y(\log E)(p^nA)\xrightarrow{R\pi_{*}\iota} R\pi_*G_n\Omega^i_Y(\log E)(-E)=G_n\underline{\Omega}^i_{X,\pi},
\end{multline*}
where \begin{multline*}\iota\colon G_n\Omega^i_Y(\log E)(p^nA)\coloneqq \frac{F^n_{*}\Omega^i_Y(\log E)(p^nA)}{B_n\Omega^i_Y(\log E)(p^nA)}\\ \longrightarrow G_n\Omega^i_Y(\log E)(-E)\coloneqq \frac{F^n_{*}\Omega^i_Y(\log E)(-E)}{B_n\Omega^i_Y(\log E)(-E)}\end{multline*}
is the map induced by inclusions.
In the same manner, we define $\iota_{\pi}\colon \frac{F^n_*\pi_{*}\Omega^i_Y(\log E)(p^nA)}{\pi_{*}B_n\Omega^i_Y(\log E)(p^nA)}\to G_n\Omega^i_{X,\pi}\coloneqq \frac{F^n_*\pi_{*}\Omega^i_Y(\log E)(-E)}{\pi_{*}B_n\Omega^i_Y(\log E)(-E)}$.

Consider the following commutative diagram induced by natural inclusions:
\[
\begin{tikzcd}
    R\pi_*G_n\Omega^i_Y(\log E)(p^nA)\arrow[r,"r'"]\arrow[d,"R\pi_{*}\iota"] &\mathrm{Cone}\Big(\frac{F^n_*\pi_{*}\Omega^i_Y(\log E)(p^nA)}{\pi_{*}B_n\Omega^i_Y(\log E)(p^nA)} \xrightarrow{\text{nat.}} R\pi_*G_n{{\Omega}}^i_Y(\log E)(p^nA) \Big)\arrow[d]\\
    G_{n}\underline{\Omega}^i_{X,\pi}\arrow[r,"r"] &\mathrm{Cone}\Big(G_n\Omega^{i}_{X,\pi} \xrightarrow{\text{nat.}} G_n\underline{\Omega}^i_{X,\pi} \Big).
\end{tikzcd}
\]
Then, the map 
\[
\underline{\Omega}^i_{X,\pi} \xrightarrow{\underline{C}^{-1}_{n,\pi}} G_n\underline{\Omega}^i_{X,\pi} \xrightarrow{r} \mathrm{Cone}\Big(G_n\Omega^{i}_{X,\pi} \xrightarrow{\text{nat.}} G_n\underline{\Omega}^i_{X,\pi} \Big)
\]
in the statement factorises through:
\begin{multline*}
    \underline{\Omega}^i_{X,\pi} \xrightarrow{R\pi_{*}C^{-1}_{n,A}} R\pi_*G_n\Omega^i_Y(\log E)(p^nA)\\
\xrightarrow{r'}\mathrm{Cone}\Big(\frac{F^n_*\pi_{*}\Omega^i_Y(\log E)(p^nA)}{\pi_{*}B_n\Omega^i_Y(\log E)(p^nA)} \xrightarrow{\text{nat.}} R\pi_*G_n\Omega^i_Y(\log E)(p^nA) \Big).
\end{multline*}
By applying Diagram \ref{eq:fgh-derived} to 
\[\cF \coloneqq R\pi_*B_n\Omega^i_Y(\log E)(p^nA),\,\,\cG \coloneqq R\pi_*F^n_*\Omega^i_Y(\log E)(p^nA),\,\,\text{and}\,\,\cH \coloneqq R\pi_*G_n\Omega^i_Y(\log E)(p^nA),
\]
we obtain a (non-canonical) isomorphism
\begin{multline*}
    \mathrm{Cone}\Big(\frac{F^n_*\pi_{*}\Omega^i_Y(\log E)(p^nA)}{\pi_{*}B_n\Omega^i_Y(\log E)(p^nA)} \xrightarrow{\text{nat.}} R\pi_*G_n\Omega^i_Y(\log E)(p^nA) \Big)\\
    \cong \mathrm{Cone}\Big(R^{>0}\pi_*B_n\Omega^i_Y(\log E)(p^nA)\to R^{>0}\pi_*F^n_*\Omega^i_Y(\log E)(p^nA) \Big).
\end{multline*}
By assumption and Lemma \ref{lemma:Kawakami-vanishing}, we have
\[
R^{>0}\pi_*B_n\Omega^i_Y(\log E)(p^nA)=0 \quad\text{and}\quad
R^{>0}\pi_*F^n_*\Omega^i_Y(\log E)(p^nA) = 0 
\]
for all $n \gg0$. 
Thus,
\[
\text{Cone}\Big(R^{>0}\pi_*B_n\Omega^i_Y(\log E)(p^nA)\\ \to R^{>0}\pi_*F^n_*\Omega^i_Y(\log E)(p^nA) \Big) =0
\]
for $n\gg0$, which concludes the proof.}
\end{proof}

\begin{theorem} \label{thm:main(isolated)}
With the notation of Setting \ref{setting:iso-charp}, suppose that  $C_{\pi}\colon Z\Omega^i_{X,\pi}\to \Omega^i_{X,\pi}$ is surjective.
Fix an integer $k \geq 0$ and let $A$ be an ample anti-effective exceptional $\bQ$-divisor on $Y$ such that $\rdown{A}=-E$. Assume that 
\begin{enumerate}
    \item $X$ is pre-$k$-$F$-injective, and 
    \item $R^{>0}\pi_{*}\Omega^i_Y(\log E)(p^rA) = 0$ for all $0 \leq i \leq k-1$ and $r \geq 0$.
\end{enumerate}
Then $X$ is pseudo-pre-$k$-Du Bois {along $\pi$}.
\end{theorem}
\noindent 
Note that (2) for $r=0$ is equivalent to $X$ being pre-$(k-1)$-Du Bois along $\pi$ 
{since $\tau^{>0}\underline{\sO}_{X,\pi}=R^{>0}\pi_{*}\sO_Y(-E)$ by \eqref{eq:du bois isolated} and 
$\tau^{>0}\underline{\Omega}^i_{X,\pi}=R^{>0}\pi_{*}\Omega^i_Y(\log E)(-E)$ for all $i>0$ by Definition \ref{def:Du Bois Complex(isolated)}.}
In general, (2) is satisfied when we reduce a pre-$(k-1)$-Du Bois singularity from characteristic zero modulo $p \gg 0$, and so this result will fit into an inductive argument proving one direction of Theorem \ref{thm:intro-main}.

\begin{proof}
{If $k=0$, then the assertion follows from Proposition \ref{prop:BST}.}
Fix $0 < i \leq  k$. By \eqref{eq:G}, we have the following exact triangle:
\begin{equation*}
G_n\Omega^{i}_{X,\pi} \to G_n\underline{\Omega}^i_{X,\pi} \xrightarrow{r} \mathrm{Cone}\Big(G_n\Omega^{i}_{X,\pi} \to G_n\underline{\Omega}^i_{X,\pi}  \Big) \xrightarrow{+1}.
\end{equation*}
Since the composition
    \begin{equation*}
    \underline{\Omega}^i_{X,\pi} \xrightarrow{\underline{C}^{-1}_{n,\pi}} G_n\underline{\Omega}^i_{X,\pi} \xrightarrow{r} \mathrm{Cone}\Big(G_n\Omega^{i}_{X,\pi} \to G_n\underline{\Omega}^i_{X,\pi}  \Big)
    \end{equation*}
    is zero by Lemma \ref{lemma:iso-key-factorisation-main-theorem}, we have a factorisation
\[
\underline{C}^{-1}_{n,\pi} \colon \underline{\Omega}^i_{X,\pi} \to G_n\Omega^{i}_{X,\pi} \to G_n\underline{\Omega}^i_{X,\pi}
\]
by Lemma \ref{lem:derived-cat-factor}.

By taking $\mathcal{H}^0$ and looking at \eqref{eq:diagram-main-iso(iterated)}, we get a factorisation:
\[
C_{n,\pi}^{-1}\colon \Omega^i_{X,\pi} \to\underline{\Omega}^i_{X,\pi}\to G_n\Omega^{i}_{X,\pi}\, (\subseteq \mathcal{H}^0(G_n\underline{\Omega}^{i}_{X,\pi})).
\]
This immediately implies the sought-after statement thanks to Proposition \ref{prop:iso-iterated-def-pre-k-F-injective}. 
\end{proof}

\begin{lemma}\label{lem:tersor and higher direct image}
    Let $X_A$ be an affine scheme of finite type over a Noetherian ring $A$.
    Let $\pi_A\colon Y_A\to X_A$ be a proper morphism over $A$, and let $\cF_A$ be a coherent $\sO_{Y_A}$-module which is free over $A$.
    Suppose that we can represent $R\pi_{A,*}\cF_A$ as a bounded complex of $\cO_{X_A}$-coherent sheaves which are free over $A$ and such that all images and kernels of the maps in the complex are free over  $A$. Further, assume that $R^j\pi_{A,*}\cF_A$ is free over $A$ for all $j\geq 0$.
    Then, for every point $s\in \Spec\,A$, we have a quasi-isomorphism
    \[
    R\pi_{A,*}\cF_A\otimes_A k(s) \cong R\pi_{s,*}\cF_s,
    \]
    where $Y_s\coloneqq Y\times_A \Spec\,k(s)$, $\pi_s\coloneqq \pi_A|_{Y_s}$, and $\cF_s\coloneqq \cF_A|_{Y_s}$.
\end{lemma}
\begin{proof}
    By \cite[\href{https://stacks.math.columbia.edu/tag/08HY}{Tag 08HY}]{stacks-project}, we have a natural map 
   \[
   R\pi_{A,*}\cF_A\otimes_A k(s) \to R\pi_{s,*}\cF_s. 
   \]
   Since we assumed that {$R\pi_{A,*}$ can be represented by a complex with all the finitely many} terms and all images and kernels being free over $A$, we have
   \[
   \cH^j(R\pi_{A,*}\cF_A\otimes^L_A k(s))\cong R^j\pi_{A,*}\cF_A\otimes_{A} k(s).
   \]
   Thus, it suffices to show that a natural map
   \[
   R^j\pi_{A,*}\cF_A\otimes_{A} k(s) \to R^j\pi_{s,*}\cF_s
   \]
   is an isomorphism.
   Since $X_A$ is affine, we have $R^j\pi_{A,*}\cF_A=H^j(Y_A,\cF_A)$ and $R^j\pi_{s,*}\cF_s=H^j(Y_s,\cF_s)$.
   Now, the assertion follows from \cite[Lemma 4.1]{Hara98}.
\end{proof}

In what follows, we explain how to use the above result and reduction modulo $p>0$ to show that singularities of pre-$k$-$F$-injective type are pre-$k$-Du-Bois.
\begin{proposition}\label{prop:reduction of pseudo-pre Du Bois}
    With the notation of Setting \ref{setting:iso}, suppose that $k$ is of characteristic zero and $X$ is pre-$(k-1)$-Du Bois.
    {Given a model $X_A$ (resp.~$\pi_A$, $x_A$) of $X$ (resp.~$\pi$, $x$)} over a finitely generated $\mathbb{Z}$-subalgebra $A$ of $k$,
    there exists a Zariski-open set $S\subseteq \Spec\,A$ of closed points such that
    \[
    H^j_{x_s}(\Omega^i_{X_s,\pi_s})\to \mathbb{H}^j_{x_s}(\underline{\Omega}^i_{X_s,\pi_s})
    \]
    is surjective for all closed points $s \in S$,  $0<i\leq k$, and $j\geq 0$, {where $X_s\coloneqq X_A\times_A k(s)$, $\pi_s\coloneqq \pi_A\times_A k(s)$, and $x_s\coloneqq x_A\times_A k(s)$.}
\end{proposition}

\begin{proof}
   Fix $0< i\leq k$ and $j \geq 0$.
   {By Lemma \ref{lem:inj of delta} and local duality},
   the map
   \begin{equation}\label{eq:ext injective}
       \Ext^{-j}(\underline{\Omega}^{i}_{X,\pi},\omega_{X}^{\bullet}) \to \Ext^{-j}(\Omega^{i}_{X,\pi},\omega_{X}^{\bullet}).
   \end{equation}
   is injective. 
   We need to show that
   \[
   \Ext^{-j}(\underline{\Omega}^{i}_{X_s,\pi},\omega_{X_s}^{\bullet}) \to \Ext^{-j}(\Omega^{i}_{X_s,\pi},\omega_{X_s}^{\bullet})
   \]
   is injective as well.
    By enlarging $A$, we can take flat models  
    \[
   \begin{tikzcd}
   E_A \arrow[r] \arrow[d] & Y_A \arrow[d,"\pi_A"] \\
\{x_A\} \arrow[r] & X_A
\end{tikzcd}
\]
over $A$ such that
    \begin{enumerate}
        \item $X_A=\Spec\,R_A$ is an affine scheme such that
        $X_A\setminus x_A$ is smooth over $\Spec\,A$,
        \item $\pi_A\colon Y_A\to X_A$ is a projective birational morphism with $\Exc(\pi_A)=E_A$,
        \item $Y_A$ is smooth over $\Spec\,A$,
        \item $E_A$ is simple normal crossing over $\Spec\,A$, and
        \item $\mathcal{H}^b(\underline{\Omega}^a_{X_A,\pi_A})$ is free over $A$ for all $a>0$ and $b\geq 0$ by generic freeness \cite[\href{https://stacks.math.columbia.edu/tag/051S}{Tag 051S}]{stacks-project}, {where $\underline{\Omega}^{a}_{X_A/A,\pi_A}\coloneqq R\pi_{A,*}\Omega^a_{Y_A/A}(\log\,E_A)(-E_A)$.}
    \end{enumerate}
    
    We show the following two claims:

    \begin{claim}\label{claim:ext vanishing}
     For any flat $\sO_{X_A}$-module $M$ {which is free over $A$}, we have 
   \[
   \Ext^l_{\sO_{X_A}}(\sO_{X_s}, M)=0
   \]
   for every $l\neq c \coloneqq \mathrm{codim}_{X_A}({X_s})$.
   \end{claim}
   \begin{proof}[Proof of Claim \ref{claim:ext vanishing}]
       This is contained in the proof of \cite[Proposition 7.4]{Schwede}. 
   \end{proof}

    \begin{claim}\label{claim:ext isomorphisms}
    We have isomorphisms
    \begin{align*}
        & \Ext^{-j}(\Omega^{i}_{X_s,\pi},\omega_{X_s}^{\bullet})\cong \Ext^c_{\sO_{X_A}}(\sO_{X_s}, \Ext^{-j-c}_{\sO_{X_A}}(\Omega^{i}_{X_A,\pi_A}, \omega^{\bullet}_{X_A}))\\
        & \Ext^{-j}(\underline{\Omega}^{i}_{X_s,\pi},\omega_{X_s}^{\bullet})\cong \Ext^c_{\sO_{X_A}}(\sO_{X_s}, \Ext^{-j-c}_{\sO_{X_A}}(\underline{\Omega}^{i}_{X_A,\pi_A}, \omega^{\bullet}_{X_A})),
    \end{align*}
    where $c \coloneqq \mathrm{codim}_{X_A}(X)$.
    \end{claim}
    \begin{proof}[Proof of Claim \ref{claim:ext isomorphisms}]
   We only prove the latter isomorphism since the proof of the former is analogous.
   By applying Lemma \ref{lem:tersor and higher direct image} to $\cF_A=\Omega^a_{Y/A}(\log E_A)(-E_A)$, we get a quasi-isomorphism
   \[
\underline{\Omega}^{a}_{X_A/A,\pi}\otimes^{L}_{A} k(s)\to R\pi_{s,*}\Omega^a_{Y_s/s}(\log\,E_s)(-E_s)\cong \underline{\Omega}^a_{X_s,{\pi_s}}.
   \]
Then we obtain
   \begin{align*}
       \Ext^{-j}_{\sO_{X_s}}(\underline{\Omega}_{X_s,{\pi_s}}^i, \omega_{X_s}^{\bullet})
       &=R^{-j}\Hom_{\sO_{X_A}}(\underline{\Omega}^i_{X_s,{\pi_s}}, \omega^{\bullet}_{X_A})\\
       &=R^{-j}\Hom_{\sO_{X_A}}(\underline{\Omega}^{i}_{X_A/A,\pi_A}\otimes^{L}_{A}k(s), \omega^{\bullet}_{X_A})\\
       &=R^{-j}\Hom_{\sO_{X_A}}(\underline{\Omega}^{i}_{X_A/A,\pi_A}\otimes^{L}_{\sO_{X_A}}\sO_{X_s}, \omega^{\bullet}_{X_A})\\
       &=R^{-j}\Hom_{\sO_{X_A}}(\sO_{X_s}, R\Hom_{\sO_{X_A}}(\underline{\Omega}^{i}_{X_A/A,\pi_A}, \omega^{\bullet}_{X_A}))
   \end{align*}
   where we used Grothendieck duality for the first equality,  \cite[\href{https://stacks.math.columbia.edu/tag/0661}{Tag 0661}]{stacks-project} for the third equality, and \cite[\href{https://stacks.math.columbia.edu/tag/0A65}{Tag 0A65}]{stacks-project} for the fourth equality.
   Consider the spectral sequence:
   \begin{multline*}
   R^a\Hom_{\sO_{X_A}}(\sO_{X_s}, R^b\Hom_{\sO_{X_A}}(\underline{\Omega}^{i}_{X_A/A,\pi_A}, \omega^{\bullet}_{X_A}))\Rightarrow\\
   R^{a+b}\Hom_{\sO_{X_A}}(\sO_{X_s}, R\Hom_{\sO_{X_A}}(\underline{\Omega}^{i}_{X_A/A,\pi_A}, \omega^{\bullet}_{X_A})).
   \end{multline*}
   By enlarging $A$ if necessary, we may assume that $R^b\Hom_{\sO_{X_A}}(\underline{\Omega}^{i}_{X_A/A,\pi_A}, \omega^{\bullet}_{X_A})$ is a free $A$-module for every $b\geq 0$.
   Then, by Claim \ref{claim:ext vanishing}, we obtain
   \begin{multline*}
       R^{-j}\Hom_{\sO_{X_A}}(\sO_{X_s}, R\Hom_{\sO_{X_A}}(\underline{\Omega}^{i}_{X_A,\pi_A}, \omega^{\bullet}_{X_A}))\cong\\
       R^c\Hom_{\sO_{X_A}}(\sO_{X_s}, R^{-j-c}\Hom_{\sO_{X_A}}(\underline{\Omega}^{i}_{X_A,\pi_A}, \omega^{\bullet}_{X_A})),
   \end{multline*}
   as desired.
   \end{proof}
   By the injectivity of \eqref{eq:ext injective} and by enlarging $A$ if necessary, we can assume that 
   \[
   \Ext^{-j-c}_{\sO_{X_A}}(\underline{\Omega}^{i}_{X_A,\pi_A},\omega_{X_A}^{\bullet}) \hookrightarrow \Ext^{-j-c}_{\sO_{X_A}}(\Omega^{i}_{X_A,\pi_A},\omega_{X_A}^{\bullet})
   \]
   is injective. 
   Let $C$ be the cokernel of the above map. {By enlarging $A$, we may assume that $C$ is free over $A$.}
   Then we have an exact sequence
   \begin{multline*}
       0\overset{\text{Claim\,\ref{claim:ext vanishing}}}{=}\Ext^{c-1}_{\sO_{X_A}}(\sO_{X_s}, C)\to\\ \Exc^c_{\sO_{X_A}}(\sO_{X_s}, \Ext^{-j-c}_{\sO_{X_A}}(\underline{\Omega}^{i}_{X_A,\pi_A},\omega_{X_A}^{\bullet}))\to 
   \Exc^c_{\sO_{X_A}}(\sO_{X_s}, \Ext^{-j-c}_{\sO_{X_A}}(\Omega^{i}_{X_A,\pi_A},\omega_{X_A}^{\bullet})).
   \end{multline*}
   Thus we obtain the desired injectivity by Claim \ref{claim:ext isomorphisms}.
\end{proof}

\begin{theorem}\label{thm:k-F-inj to k-Du Bois}
    With the notation of Setting \ref{setting:iso}, fix a positive integer $k\geq 0$.
    Suppose that $k$ is of characteristic zero.
    Furthermore, suppose that
    given a model of $X$ over a finitely generated $\mathbb{Z}$-subalgebra $A$ of $k$,
    there exists a Zariski-dense set of closed points $S\subseteq \Spec\,A$ such that 
    $X_s$ is pre-$k$-$F$-injective for every $s\in S$. 
    Then $X$ is pre-$k$-Du Bois.
\end{theorem}
\begin{proof}
{We take flat models over $A$ as in the proof of Proposition \ref{prop:reduction of pseudo-pre Du Bois}.
We prove the assertion by induction on $k$. The case where $k=0$ follows from \cite[Theorem 6.1]{Schwede}. 
Assume that $k>0$.
Since $X_s$ is pre-$(k-1)$-$F$-injective for every $s\in S$,
it follows that $X$ is pre-$(k-1)$-Du Bois by induction hypothesis. Then, by Lemma \ref{lem:inj of delta}, the map
    \[
    H^j_{x}(\Omega^k_{X,\pi})\to \mathbb{H}^j_{x}(\underline{\Omega}^k_{X,\pi})
    \]
    is surjective for all $j\geq 0$. As $k>0$, we can apply Proposition \ref{prop:reduction of pseudo-pre Du Bois}
     to deduce that
     \begin{equation}\label{eq:map of local coh}
        H^j_{x_s}(\Omega^k_{X_s,\pi_s})\to \mathbb{H}^j_{x_s}(\underline{\Omega}^k_{X_s,\pi_s})
    \end{equation}
    is surjective for all $j\geq 0$.}

    {Next, we check the conditions (1) and (2) in Theorem \ref{thm:main(isolated)} are satisfied.}
    The condition (1) is satisfied by assumption.
     To verify (2), we take an ample $\Q$-divisor $B$ on $Y$ such that $\lfloor B\rfloor =-E$.
    By enlarging $A$ if necessary, we can take a flat model $B_A$ of $B$.
    By relative Serre vanishing, there exists $M\gg0$ such that
    \[R^{>0}\pi_{A,*}\Omega^i_{Y_A/A}(\log E_A)(mB_A)=0\]
    for all $m \geq M$ and $0\leq i \leq k-1$.
    Then, by taking the characteristics $p(s)$ of the residue fields $k(s)$  to be sufficiently large, we may assume that 
    \[
    R^{>0}\pi_{s,*}\Omega^i_{Y_s}(\log E_s)(p(s)^rB_s)=0
    \]
    {for all $0\leq i \leq k-1$ and $r>0$.
    
    Moreover, $R^{>0}\pi_{*}\Omega^i_{Y}(\log E)(-E)=0$ since $X$ is pre-$(k-1)$-Du Bois. Indeed, if $i=0$, then this follows from the isomorphism $R^{>0}\pi_{*}\sO_Y(-E)\cong \tau^{>0}\underline{\sO}_X$ (see \eqref{eq:du bois isolated} and Lemma \ref{lem:independence of resolutions,i=0,iso} (2)), and
    if $i>0$, then we have $\underline{\Omega}^i_{X,x}=\underline{\Omega}^i_{X}$ and thus $R^{>0}\pi_{*}\Omega^i_{Y}(\log E)(-E)=\mathcal{H}^{>0}(\underline{\Omega}^i_{X})=0$ for all $0<i\leq k-1$.
    
    Now by enlarging $A$, we can assume that $R^{>0}\pi_{*,s}\Omega^i_{Y_s}(\log E_s)(-E_s)=0$ for $0\leq i<k-1$ {by Lemma \ref{lem:tersor and higher direct image}}.
    Thus, we can apply Theorem \ref{thm:main(isolated)} to conclude $X_s$ is pseudo-pre-$k$-Du Bois along $\pi_s$, and the map \eqref{eq:map of local coh} is injective for all $j\leq d-k$.
    Now, essentially the same argument as in Theorem \ref{thm:pseudo-pre-k-Du Bois=pre-k-Du Bois in char 0} shows that $\mathcal{H}^j(\underline{\Omega}^k_{X_s,\pi_s})=0$ for $j\leq d-k-1$.
    }
    Moreover, by \cite[Theorem 7.29 (b)]{Peter-Steenbrink(Book)} and \cite[Lemma 2.5]{Friedman-Laza2}, we have $\mathcal{H}^j(\underline{\Omega}^k_X)=0$ for all $j\geq d-k$, which implies $\mathcal{H}^j(\underline{\Omega}^k_{X_s,\pi_s})=0$ after enlarging $A$.
    
    Therefore $X_s$ is pre-$k$-Du Bois along $\pi_s$, and we can easily see that $X$ is also pre-$k$-Du Bois.
\end{proof}

\begin{corollary}\label{cor:k-F-inj to k-Du Bois}
    In Theorem \ref{thm:k-F-inj to k-Du Bois}, further assume that $\Omega^i_{X}$ and $\Omega^i_{X_s}$ are reflexive for all $i\leq k$, and that  
    $X_s$ is $k$-$F$-injective (see Definition \ref{def:intro-kFinj}). Then $X$ is $k$-Du Bois. 
\end{corollary}
\noindent Recall that local freeness and torsion-freeness are constructible properties in a family \cite[Appendix E]{Ulrich-Torsten}.
Together with this and \cite[Proposition 1.1]{Har80}, we know that reflexivity is also a constructible property. Thus, we can assume that $\Omega^i_{X_s}$ is reflexive in the above.
\begin{proof}
    Firstly, $X$ is pre-$k$-Du Bois by Theorem \ref{thm:k-F-inj to k-Du Bois}. Moreover, 
    $\Omega_X^{i}\to \Omega^{i}_{X,h}$ is an isomorphism for all $0 \leq i \leq k$, because $\Omega^i_X$ is reflexive. Hence $X$ is $k$-Du-Bois.  
\end{proof}

\subsection{$k$-Du Bois implies $k$-$F$-injective}

In order to show that $k$-Du Bois singularities are $k$-$F$-injective for reductions modulo infinitely many $p>0$, we prove the following result.
\begin{lemma}\label{lem:ordinarity conj}
    Let $Y$ be a $d$-dimensional smooth projective variety over a field of characteristic zero and let $E$ be a reduced divisor with simple normal crossing support.
     Suppose that the ordinarity conjecture \ref{conj:ordinarity} holds.
     Then given a model of $Y$ over a finitely generated $\mathbb{Z}$-subalgebra $A$ of $k$,
    there exists a Zariski-dense set of closed points $S\subseteq \Spec\,A$ such that $H^{j}(Y_s, B^i_{Y_s}(\log E_s))=0$ for all $i,j\geq 0$.
\end{lemma}
\begin{proof}
By the ordinarity conjecture \ref{conj:ordinarity} applied to the disjoint union of all the strata $V_1, \ldots, V_m$ of $E$ (including $Y$), we get that
    \[
H^j(V_{{l},s}, B\Omega^i_{V_{{l},s}})=0
    \]
    for all $i, j, l\geq 0$, where $V_{l,s}$ denotes the fibre over $s \in \Spec\, A$ of a model of $V_{l}$. Thus, we can conclude the proof by the residue exact sequence for $B^j_{Y_s}(\log\,E_s)$ (see Lemma \ref{lem:residue exact sequences} (2)).
\end{proof}

\begin{theorem}\label{thm:pre-k-Du Bois to pre-$F$-inj in p>0}
    {We take $x\in X$ as in} Setting \ref{setting:iso-charp} and suppose that $k$ is of characteristic $p>d-1$.
    Fix an open immersion $X\hookrightarrow \overline{X}$ into a normal projective variety $\overline{X}$.
    Suppose that 
    \begin{enumerate}
        \item there exists a log resolution $\overline{\pi}\colon \overline{Y}\to \overline{X}$ with $\overline{E}\coloneqq \Exc(\overline{\pi})$,
        \item there exists a sufficiently ample Cartier divisor $H_{\overline{X}}$ on $\overline{X}$ such that $\overline{H}+\overline{E}$ is simple normal crossing where $\overline{H}\coloneqq\pi^{*}H_{\overline{X}}$,
        \item $(\overline{H}, \overline{E}|_{\overline{H}})$ lifts to $W_2(k)$, and
        \item {$H^j(\overline{Y}, B\Omega^i_{\overline{Y}}(\log \overline{E}+\overline{H}))=0$ for all $i,j\geq 0$.}
    \end{enumerate}
    Set $Y\coloneqq \overline{\pi}^{-1}(X)$ and $\pi\coloneqq \overline{\pi}|_Y\colon Y\to X$.
    If $X$ is pre-$k$-Du Bois along $\pi$ for some integer $k>0$, then
    $C_{\pi}\colon H_{\m}^j(\Omega^i_{X,\pi})\to H_{\m}^j(G\Omega^i_{X,\pi})$ is injective when $0<i\leq k$ and $j\leq d-i$.
\end{theorem}

\begin{proof}
We divide the proof into several claims:
\begin{claim}\label{cl1}
        The map $H^j(\overline{Y}, Z^i_{\overline{Y}}(\log \overline{E}+\overline{H}))
        \xrightarrow{C} 
        H^j(\overline{Y}, \Omega^i_{\overline{Y}}(\log \overline{E}+\overline{H}))$
        is surjective for all $i,j\geq 0$.
    \end{claim}
    \begin{proof}[Proof of Claim \ref{cl1}]
        The assertion follows from {(4)} and
        the short exact sequence
        \[
        0\to B\Omega^i_{\overline{Y}}(\log \overline{E}+\overline{H}) \to Z\Omega^i_{\overline{Y}}(\log \overline{E}+\overline{H})\xrightarrow{C} \Omega^i_{\overline{Y}}(\log \overline{E}+\overline{H}) \to 0. \qedhere
        \]
    \end{proof}
\begin{claim}\label{cl2}
        The map $C \colon H^j(\overline{Y}, Z\Omega^i_{\overline{Y}}(\log \overline{E})\otimes \sO_{\overline{Y}}(\overline{H}))
        \to 
        H^j(\overline{Y}, \Omega^i_{\overline{Y}}(\log \overline{E})\otimes \sO_{\overline{Y}}(\overline{H}))$
        is surjective for $i\geq 0$ and $i+j\geq d$.
    \end{claim}
    \begin{proof}[Proof of Claim \ref{cl2}]
       We have the following commutative diagram:
       \[
       \begin{tikzcd}
           Z\Omega^i_{\overline{Y}}(\log \overline{E}+\overline{H}) \arrow[r,"C"]\arrow[d] &\Omega_{\overline{Y}}^{i}(\log \overline{E}+\overline{H})\arrow[d] \\
           Z\Omega^i_{\overline{Y}}(\log \overline{E})\otimes \sO_{\overline{Y}}(\overline{H})\arrow[r,"C"] & \Omega_{\overline{Y}}^{i}(\log \overline{E})\otimes \sO_{\overline{Y}}(\overline{H})
       \end{tikzcd}
       \]
with vertical maps being natural inclusions.
Consider the restriction exact sequence:
\[
0\to \Omega_{\overline{Y}}^{i}(\log \overline{E}+\overline{H})\to \Omega_{\overline{Y}}^{i}(\log \overline{E})\otimes \sO_{\overline{Y}}(\overline{H})\to \Omega_{\overline{H}}^{i}(\log \overline{E}|_{\overline{H}})\otimes \sO_{\overline{H}}(\overline{H}) \to 0.
\]
Since $i+j\geq d>\dim\,\overline{H}$, we have $H^j(\overline{H}, \Omega_{\overline{H}}^{i}(\log \overline{E}|_{\overline{H}})\otimes \sO_{\overline{H}}(\overline{H}))=0$ by Akizuki-Nakano vanishing of $W_2(k)$-liftable pair \cite[Corollary 3.8]{Hara98}.
Combining with Claim \ref{cl1}, the assertion holds.
    \end{proof}

    \begin{claim}\label{cl3}
        The map
        $R^{j}\overline{\pi}_{*}Z_{\overline{Y}}^i(\log \overline{E})\xrightarrow{C}
    R^{j}\overline{\pi}_{*}\Omega^i_{\overline{Y}}(\log \overline{E})$
    is surjective for $i\geq 0$ and $i+j\geq d$.
    \end{claim}
    \begin{proof}[Proof of Claim \ref{cl3}]
    {Firstly, by Fujita vanishing, we can assume that \[R^{>0}\overline{\pi}_{*}\Omega^i_{\overline{Y}}(\log \overline{E})\otimes \sO_{\overline{Y}}(m\overline{H})=0\] for all $i\geq 0$ and $m>0$ since $H_{\overline X}$ was chosen to be sufficiently ample. Here, we note that the choice of $\overline{H}$ does not depend on $p$.
    By the exact triangles,
    \begin{align*}
        &R\overline{\pi}_{*}B\Omega^i_{\overline{Y}}(\log \overline{E})\to R\overline{\pi}_{*}Z\Omega^i_{\overline{Y}}(\log \overline{E})\to 
        R\overline{\pi}_{*}\Omega^i_{\overline{Y}}(\log \overline{E})\xrightarrow{+1}\\
        &R\overline{\pi}_{*}Z\Omega^{i-1}_{\overline{Y}}(\log \overline{E})\to F_{*}R\overline{\pi}_{*}\Omega^{i-1}_{\overline{Y}}(\log \overline{E})\to 
        R\overline{\pi}_{*}B\Omega^{i}_{\overline{Y}}(\log \overline{E})\xrightarrow{+1}
    \end{align*}
    we can see that $R^{>0}\overline{\pi}_{*}Z\Omega^i_{\overline{Y}}(\log \overline{E})\otimes \sO_{\overline{Y}}(\overline{H})=0$ for all $i\geq 0$.}
    
    Now, consider the spectral sequence
        \[
        E_2^{m,n}=H^m\Big(\overline{Y}, R^n\overline{\pi}_{*}(\Omega^{i}_{\overline{Y}}(\log \overline{E})\otimes \sO_{\overline{Y}}(\overline{H}))\Big)
        \Rightarrow H^{m+n}\Big(\overline{Y}, \Omega^{i}_{\overline{Y}}(\log \overline{E})\otimes \sO_{\overline{Y}}(\overline{H}))\Big).
        \]
        Since we have
        \[
        R^n\overline{\pi}_{*}(\Omega^{i}_{\overline{Y}}(\log \overline{E})\otimes \sO_{\overline{Y}}(\overline{H}))\cong R^n\overline{\pi}_{*}\Omega^{i}_{\overline{Y}}(\log \overline{E})\otimes \sO_{\overline{X}}(H_{\overline{X}}),
        \]
        we may assume $E^{m,n}=0$ when $m>0$. 
        Thus, 
        \[
        H^j(\overline{Y}, \Omega^{i}_{\overline{Y}}(\log \overline{E})\otimes \sO_{\overline{Y}}(\overline{H})))\cong H^0(\overline{Y}, R^j\overline{\pi}_{*}(\Omega^{i}_{\overline{Y}}(\log \overline{E})\otimes \sO_{\overline{Y}}(\overline{H}))).
        \]
        Similarly, we obtain 
        \[
        H^j(\overline{Y}, Z\Omega^{i}_{\overline{Y}}(\log \overline{E})\otimes \sO_{\overline{Y}}(\overline{H})))\cong H^0(\overline{Y}, R^j\overline{\pi}_{*}(Z\Omega^{i}_{\overline{Y}}(\log \overline{E})\otimes \sO_{\overline{Y}}(\overline{H}))).
        \]
        Thus, by Claim \ref{cl2}, the map
        \[
           H^0(\overline{Y}, R^j\overline{\pi}_{*}(Z\Omega^{i}_{\overline{Y}}(\log \overline{E})\otimes \sO_{\overline{Y}}(\overline{H})))\xrightarrow{C}
        H^0(\overline{Y}, R^j\overline{\pi}_{*}(Z\Omega^{i}_{\overline{Y}}(\log \overline{E})\otimes \sO_{\overline{Y}}(\overline{H})))
        \]
        is surjective.
        We have the following diagram:
        \[
        \begin{tikzcd}
            H^0(\overline{Y}, R^j\overline{\pi}_{*}(Z\Omega^{i}_{\overline{Y}}(\log \overline{E})\otimes \sO_{\overline{Y}}(\overline{H})))\arrow[r,"C"]\arrow[d]& H^0(\overline{Y}, R^j\overline{\pi}_{*}(\Omega^{i}_{\overline{Y}}(\log \overline{E})\otimes \sO_{\overline{Y}}(\overline{H})))\arrow[d] \\
 R^j\overline{\pi}_{*}(Z\Omega^{i}_{\overline{Y}}(\log \overline{E}))\otimes \sO_{\overline{Y}}(\overline{H})\arrow[r,"C"] & R^j\overline{\pi}_{*}(\Omega^{i}_{\overline{Y}}(\log \overline{E}))\otimes \sO_{\overline{Y}}(\overline{H}),
        \end{tikzcd}
        \]
where we can assume that the right vertical map is surjective as $\overline{H}$ is taken to be sufficiently ample.
Thus, the lower horizontal map is surjective, which shows the assertion.
\end{proof}

\begin{claim}\label{cl4}
    $X$ is pre-$k$-$F$-injective.
\end{claim}
 \begin{proof}[Proof of Claim \ref{cl4}]
     By the definition of pre-$k$-$F$-injectivity, we aim to show that for every integer $0\leq i\leq k$,
     \begin{enumerate}
         \item $C_{\pi}\colon Z\Omega^i_{X,\pi}\to \Omega^i_{X,\pi}$ is surjective, and
         \item $C_{\pi}\colon H_{\m}^j(\Omega^i_{X,\pi})\to H_{\m}^j(G\Omega^i_{X,\pi})$ is injective for $i+j\leq d$.
     \end{enumerate}
     By Proposition \ref{prop:iso-k-Du-Bois-C-surjective}, we obtain (1).
     We show (2). By Proposition \ref{prop:iso-k-Du-Bois-C-surjective} again, (2) is equivalent to the injectivity
     \[C\colon \mathbb{H}_{\m}^j(R\pi_{*}\Omega^i_{Y}(\log E)(-E))\to \mathbb{H}_{\m}^j(R\pi_{*}G\Omega^i_{Y}(\log E)(-E))
     \]
     for all $0\leq i\leq k$ and $i+j\leq d$, where $Y\coloneqq \pi^{-1}(X)$.
     By local duality, this is equivalent to the surjectivity of
     \[
     R^{-j}\text{Hom}(R\pi_*G\Omega^i_Y(\log E)(-E), \omega^\bullet_X) \to R^{-j}\text{Hom}(R\pi_*\Omega^{i}_Y(\log E)(-E), \omega^\bullet_X).
     \]
     By Grothendieck duality (Lemma \ref{lem:preliminaries-duality}), 
     this map becomes:
\[
R^{d-j}\pi_*Z\Omega^{d-i}_Y(\log E) \to R^{d-j}\pi_*\Omega^{d-i}_Y(\log E).
\] 
Since $d-i+d-j\geq d$, this is surjective by Claim \ref{cl3}.
 \end{proof}
\end{proof}

\begin{theorem}\label{thm:pre-k-F-inj type to pre-$k$-Du Bois}
    With the notation of Setting \ref{setting:iso}, suppose that $k$ is of characteristic zero and $X$ is pre-$k$-Du Bois for some integer $k>0$.
    Then,
    given a model of $X$ over a finitely generated $\mathbb{Z}$-subalgebra $A$ of $k$,
    there exists a Zariski-dense set of closed points $S\subseteq \Spec\,A$ such that 
    $X_s$ is pre-$k$-$F$-injective for all $s\in S$.
\end{theorem}
\begin{proof}
    Take an open immersion
    $X\hookrightarrow \overline{X}$ into a normal projective variety $\overline{X}$.
    Let $Z\subset X$ be the closed subscheme such that the blow-up along $Z$ coincides with $\pi$.
    Let $\overline{\pi}\colon \overline{Y}\to {\overline{X}}$ be a composition of the blow-up along the closure $\overline{Z}\subset \overline{X}$ and a resolution that is an isomorphism over the smooth locus.
    Take a sufficiently ample divisor $H_{\overline{X}}$ on $\overline{X}$.
    Set $\overline{H}=\overline{\pi}^{*}H_{\overline{X}}$.
    By Bertini theorem, we may assume that $\overline{E}+\overline{H}$ is simple normal crossing, where $\overline{E}\coloneqq \Exc(\overline{\pi})$.
    Now, by enlarging $A$, we can find a Zariski-dense set $S$ of closed points such that the assumptions of Theorem \ref{thm:pre-k-Du Bois to pre-$F$-inj in p>0} are satisfied. In fact, in the case of (3), the pair $(H_s,E_s|_{H_s})$ lifts to $W_2(k)$ by \cite[Proposition 2.5]{ABL} as we can assume that $A$ is smooth over $\Z$.
    {As for (4), we use Lemma \ref{lem:ordinarity conj}.}
    Therefore, $C_{\pi}\colon H_{\m}^j(\Omega^i_{X_s,\pi})\to H_{\m}^j(G\Omega^i_{X_s,\pi})$ is injective for all $0<i\leq k$ and $j\leq d-i$.
    Recall that $\Omega^0_{X_s,\pi}=\sO_{X_s}$ and $G\Omega^0_{X_s,\pi}=F_{*}\sO_
{X_s}$.
    Thus, by \cite[Theorem B]{Bhatt-Schwede-Takagi}, $C_{\pi}\colon H_{\m}^j(\Omega^0_{X_s,\pi})\to H_{\m}^j(G\Omega^0_{X_s,\pi})$ is injective for all $j\leq d$.
    Finally, $C_{\pi} \colon Z\Omega^i_{X_s,\pi} \to \Omega^i_{X_s,\pi}$ is surjective for all $0\leq i\leq k$ by Proposition \ref{prop:iso-k-Du-Bois-C-surjective}.
    Thus, we conclude that $X_s$ is pre-$k$-$F$-injective.
\end{proof}

\begin{corollary}\label{cor:k-Du Bois to k-F-inj}
    With the notation of Setting \ref{setting:iso}, suppose that $k$ is of characteristic zero.
    Further assume that $\Omega^i_X$ is reflexive for all $0\leq i\leq k$ and $X$ is $k$-Du Bois for some integer $k>0$.
    Then, given a model of $X$ over a finitely generated $\mathbb{Z}$-subalgebra $A$ of $k$,
    there exists a Zariski-dense set of closed points $S\subseteq \Spec\,A$ such that 
    $X_s$ is $k$-$F$-injective for all $s\in S$.
\end{corollary}
\begin{proof}
    Recall that reflexivity is a constructible property (see Corollary \ref{cor:k-F-inj to k-Du Bois}).
    Now, by Theorem \ref{thm:pre-k-F-inj type to pre-$k$-Du Bois} and Corollary \ref{cor:k-F-inj of reflexvie case(isorated)} we get that $X_s$ is $k$-$F$-injective.
\end{proof}

\section{Applications to Frobenius liftable singularities}
In what follows we give a proof of Theorem \ref{thm:Frobenius-application-intro} from the introduction.

\begin{lemma}\label{lem:F-lift to pre-k-F-inj}
    Let $X=\Spec\,R$ be a normal affine variety with isolated singularities over a perfect field of positive characteristic. We fix an integer $k\geq 0$. {Assume that $\Omega^i_{X,\pi}$ is reflexive for all integers $0\leq i \leq  k$.}
    If $X$ is Frobenius liftable, then it is pre-$k$-$F$-injective for all $k\geq 0$.
\end{lemma}
\begin{proof} 
    Since $X$ is Frobenius liftable, the map
    $C\colon Z\Omega^{[i]}_X\to \Omega^{[i]}_X$
    is a split surjection for all $i\geq 0$ (cf.~\cite[Lemma 3.8]{Kaw4}).
    Taking $\mathcal{H}om_{\sO_X}(-,\omega_X)$, we obtain a split injection
    $\Omega^{[i]}_X\to (G\Omega^{[i]}_X)^{**}$ (cf.\ Lemma \ref{lem:preliminaries-duality}).
    Since we have a factorisation $\Omega^{[i]}_X\to G\Omega^{[i]}_X \to (G\Omega^{[i]}_X)^{**}$, we get that
     \[
     C^{-1}\colon H^j_{\m}(\Omega^{[i]}_X) \to H^j_{\m}(G\Omega^{[i]}_X)
     \]
     is injective for all $i,j\geq 0$ and, in particular, $X$ is pre-$k$-$F$-injective by Proposition \ref{prop:pi-iso-reflexive-agree}.
\end{proof}

\begin{theorem} \label{thm:Frobenius-application}
Let $X = \Spec R$ be a normal affine variety over an algebraically closed field $k$ of characteristic zero. Assume that $R = k[x_1,\ldots, x_n]/f$ for $f \in k[x_1,\ldots, x_n]$ and $X$ is smooth outside $\m = (x_1,\ldots, x_n)$. Further, assume that $X$ is rational at $\m$.
Moreover, suppose that given a model of $X$ over a finitely generated $\bZ$-subalgebra $A$ of $k$, there exists a Zariski-dense set of closed points $S \subseteq \Spec A$ such that $X_s$ is Frobenius liftable over $W_2(k(s))$ where $k(s)$ is the residue field of $s \in S$. 

Then the following statements hold.
\begin{enumerate}
\item If {$\dim\,X\geq 4$}, then $X$ is smooth.
\item If {$\dim\,X= 3$}, then $X$ is either smooth or $\widehat{R}_\m \cong k\llbracket x_1,x_2,x_3,x_4 \rrbracket/(x_1x_2-x_3x_4)$.
\end{enumerate}
\end{theorem}
\begin{proof}
Let $x\in X$ be the point corresponding to $\m$. 
We take a projective birational morphism $\pi\colon Y\to X$ and $E$ as in Setting \ref{setting:iso}, and take flat models of them by enlarging $A$.

We start with the proof of (1). By \cite[Theorem 1.11]{Graf15}, we have that $\Omega^1_X$ and $\Omega^2_X$ are reflexive. 
By \cite[Corollary 1.11]{KS21}, $\Omega^i_{X,\pi}$ is reflexive for every $i\geq 0$, and thus so is $\Omega^i_{X_s,\pi_s}$ (see Corollary \ref{cor:k-F-inj to k-Du Bois}).
Since $X_s$ is Frobenius liftable, it is pre-$k$-$F$-injective for every $k$ by Lemma \ref{lem:F-lift to pre-k-F-inj}, and so Theorem \ref{thm:k-F-inj to k-Du Bois} implies that $X$ is pre-$k$-Du Bois for every $k$. In particular, $X$ is $2$-Du Bois.

Hence, by \cite[Theorem 1]{JKSY22} (cf.\ \cite{MOPW23}), the minimal exponent $\alpha(f) \geq 3$. Since $\alpha(f)>2$, we can apply \cite[Theorem 1.5]{MOPW23} which states that 
\[
\cH^{1}(\underline{\Omega}^{n-2}_{X})_x \cong \cO_{X,x}/(J_f + (f))
\]
where $J_f$ is the Jacobian of $f$. But
$\cH^{1}(\underline{\Omega}^{n-2}_{X})_x=0$
as $X$ is $(n-2)$-pre-Du-Bois, which implies that $X$ is smooth.

As for (2), we see by the above argument that $\alpha(f) \geq 2$. If $\alpha(f)>2$, then we are done as above. Hence we may assume that $\alpha(f)=2$. In this case, since $n=4$, \cite[Corollary 6.3]{BM23} implies that up to analytic change of coordinates we can write $f$ as $x_1x_2-x_3x_4$. This concludes the proof.
\end{proof}

\begin{remark} \label{remark:Frobenius-application-complete}
The same proof works if $X = \Spec R$ is a normal affine variety such that for every maximal ideal $\m \in X$ we have $\widehat{R}_\m \cong k\llbracket x_1, \ldots, x_n \rrbracket/f$ for some power series $f \in k\llbracket x_1, \ldots, x_n \rrbracket$.

Indeed, the Deligne--Du Bois complex is stable under completion, and so
$\cH^i(\underline{\Omega}^k_{\widehat{R}_\m}) = 0$
for $i>0$ and all $k\geq 0$. By \cite[Corollary 2.24]{GLS} and since the singularity is isolated, we have that $\widehat{R}_\m \cong k\llbracket x_1, \ldots, x_n \rrbracket/g$ for any polynomial $g \in k[x_1, \ldots, x_n]$ whose terms agree with that of $f$ to high degree. Again by stability under completion of the Deligne--Du Bois complex, we have that
$\Spec k[x_1, \ldots, x_n]/g$
is pre-$k$-Du Bois for every $k$. Hence, we can now apply the same argument as in the proof of Theorem \ref{thm:Frobenius-application}.
\end{remark}

\begin{corollary} \label{cor:Frobenius-applications}
Let $X$ be a three-dimensional Gorenstein terminal variety over an algebraically closed field $k$ of characteristic zero. Suppose that given a model of $X$ over a finitely generated $\bZ$-subalgebra $A$ of $k$, there exists a Zariski-dense set of closed points $S \subseteq \Spec A$ such that $X_s$ is Frobenius liftable over $W_2(k(s))$ where $k(s)$ is the residue field of $s \in S$.

Then for each $x \in X$, we have that either $x$ is a smooth point or 
\[\widehat{\cO_{X,x}} \cong k\llbracket x_1,x_2,x_3,x_4 \rrbracket/(x_1x_2-x_3x_4).
\]
\end{corollary}
\begin{proof}
    This follows by Theorem \ref{thm:Frobenius-application}, Remark \ref{remark:Frobenius-application-complete}, and the classification of Gorenstein terminal singularities. Specifically, write $X = \Spec R$ and note that by \cite[Main Theorem I]{Reid83} (see also \cite[Corollary 5.38]{KM98}), for every maximal ideal $\m$, we have that $\widehat{R}_\m \cong k\llbracket x_1,x_2,x_3,x_3\rrbracket/f$ for $f \in k\llbracket x_1,x_2,x_3,x_4 \rrbracket$. Hence, by Remark \ref{remark:Frobenius-application-complete}, we can conclude that $\m$ is smooth or $f$ is equal, up to an analytic change of coordinates, to $x_1x_2-x_3x_4$. 
\end{proof}

\section{Appendix: pseudo-$k$-Du Bois singularities in characteristic zero}

In this section, we verify that the \emph{a priori} weaker notion of pseudo-pre-$k$-Du Bois singularities agrees with that of pre-$k$-Du-Bois singularities in characteristic zero.

\begin{definition}\label{def:pseudo-pre-$k$-Du Bois}
Let $X$ be a reduced separated scheme of essentially finite type over a field of characteristic zero. Assume that $X$ is equidimensional of dimension $d$.

We say that $X$ is \emph{pseudo-pre-$k$-Du Bois} if the natural map
\[
H^j_x(\Omega^i_{X}) \to \mathbb{H}^j_x(\underline{\Omega}^i_{X})
\]
is injective for every closed point $x \in X$, every $0 \leq i \leq k$, and every $j \leq d-i$.
\end{definition}

The aim of this section is to prove the following theorem:

\begin{theorem}\label{thm:pseudo-pre-k-Du Bois=pre-k-Du Bois in char 0}
    With notation as in Definition \ref{def:pseudo-pre-$k$-Du Bois}, assume that $X$ has only an isolated singularity $\{x\}$.
    Let $k\geq 0$ be an integer.
    Then $X$ is pre-$k$-Du Bois if and only if $X$ is pseudo-pre-$k$-Du Bois.
\end{theorem}

\begin{definition}
    Let $X$ be a reduced separated scheme of finite type over a field of characteristic zero.
    Let $(\underline{\Omega}^{\bullet}_X,F)$ be the filtered Du Bois complex.
    We set $\underline{\Omega}^{\leq i}_X\coloneqq \underline{\Omega}^{\bullet}_X/F^{i+1}\underline{\Omega}^{\bullet}_X$ and $\Omega^{\leq i}_{X,h}\coloneqq  [\Omega^0_{X,h}\xrightarrow{d}\Omega^1_{X,h}\xrightarrow{d}\cdots\xrightarrow{d}\Omega^i_{X,h}]$.
    {We have the natural map $\Omega^{\leq i}_{X,h}\to \underline{\Omega}^{\leq i}_X$ (see the proof of \cite[Proposition 2.3]{SVV}).}
\end{definition}

\begin{lemma}\label{lem:E_1-deg}
Let $X=\Spec\,R$ be a reduced affine scheme over a field of characteristic zero and let $\overline{X}$ be a compactification of $X$.
Then the natural map
\[
\mathbb{H}^j(\overline{X}, \Omega^{\leq i}_{\overline{X},h})\to \mathbb{H}^j(\overline{X}, \underline{\Omega}^{\leq i}_{\overline{X}})
\]
is surjective for all $i\geq 0$ and $j\geq 0$.
\end{lemma}
\begin{proof}
We may assume that $k=\mathbb{C}$.
Then this is \cite[Proposition 2.3]{SVV}.
\end{proof}

\begin{lemma}\label{lem:inj of delta'}
    Let $X=\Spec\,R$ be a reduced affine scheme of finite type over a field of characteristic zero, let $x\in X$ be a closed point, and  let $\m\subset R$ the maximal ideal corresponding to $\{x\}$.
    Suppose that $X$ is smooth outside $\{x\}$.
    Then the natural map
    \[
    \mathbb{H}^j_{\m}(X, \Omega^{\leq i}_{X,h})\to \mathbb{H}^j_{\m}(X, \underline{\Omega}^{\leq i}_{X})
    \]
    is surjective for all $i,j\geq 0$.
\end{lemma}
\begin{proof}
We take an open immersion $X\hookrightarrow \overline{X}$ to a projective variety $\overline{X}$.
Set $D\coloneqq \overline{X}\setminus X$ and $Z\coloneqq D\cup \{x\}$ and $U \coloneqq \overline{X} \backslash D = X \backslash \{x\}$.
We then have the following commutative diagram:
\[
\begin{tikzcd}
\mathbb{H}^{j-1}(U, \Omega^{\leq i}_{\overline{X},h})\arrow[r] \arrow[d, equal] & \mathbb{H}^j_Z(\overline{X}, \Omega^{\leq i}_{\overline{X},h}) \arrow[r] \arrow[d] & \mathbb{H}^j(\overline{X}, \Omega^{\leq i}_{\overline{X},h}) \arrow[r] \arrow[d,twoheadrightarrow,"\text{Lem\,}\ref{lem:E_1-deg}"] & \mathbb{H}^j(U, \Omega^{\leq i}_{\overline{X},h})\arrow[d, equal] \\
\mathbb{H}^{j-1}(U, \underline{\Omega}^{\leq i}_{\overline{X}}) \arrow[r] & \mathbb{H}^j_Z(\overline{X}, \underline{\Omega}^{\leq i}_{\overline{X}}) \arrow[r]  & \mathbb{H}^j(\overline{X}, \underline{\Omega}^{\leq i}_{\overline{X}}) \arrow[r] & \mathbb{H}^j(U, \underline{\Omega}^{\leq i}_{\overline{X}}).
\end{tikzcd}
\]
By diagram chase, the second vertical map
$\mathbb{H}^j_Z(\overline{X}, \Omega^{\leq i}_{\overline{X},h})\to \mathbb{H}^j_Z(\overline{X}, \underline{\Omega}^{\leq i}_{\overline{X}})$ is surjective.
Then, by the following diagram
\[
\begin{tikzcd}
\mathbb{H}^{j}_Z(\overline{X}, \Omega^{\leq i}_{\overline{X},h})\arrow[r,equal] \arrow[d, twoheadrightarrow] & \mathbb{H}^{j}_D(\overline{X}, \Omega^{\leq i}_{\overline{X},h})\oplus \mathbb{H}^{j}_{\m}(\overline{X}, \Omega^{\leq i}_{\overline{X},h}) \arrow[d, twoheadrightarrow]\\
\mathbb{H}^{j}_Z(\overline{X}, \underline{\Omega}^{\leq i}_{\overline{X}})\arrow[r,equal] & \mathbb{H}^{j}_D(\overline{X}, \underline{\Omega}^{\leq i}_{\overline{X}})\oplus \mathbb{H}^{j}_{\m}(\overline{X}, \underline{\Omega}^{\leq i}_{\overline{X}})
\end{tikzcd}
\]
it follows that $\mathbb{H}^{j}_{\m}(\overline{X}, \Omega^{\leq i}_{\overline{X},h})\to \mathbb{H}^{j}_{\m}(\overline{X}, \underline{\Omega}^{\leq i}_{\overline{X}})$
is surjective.
By excision \cite[III, Section 3, Exercise 2.5]{Har}, we get thus get the desired surjectivity
$\mathbb{H}^{j}_{\m}(X, \Omega^{\leq i}_{X,h})\to \mathbb{H}^{j}_{\m}(X, \underline{\Omega}^{\leq i}_{X})$.
\end{proof}

\begin{lemma}\label{lem:inj of delta}
Let $X=\Spec\,R$ be a reduced affine scheme of finite type over a field of characteristic zero, let $x\in X$ be a closed point, and let $\m\subset R$ be the maximum ideal corresponding to $\{x\}$.
Suppose that $X$ is pre-$(i-1)$-Du Bois and that $X$ is smooth outside $\{x\}$.
Then the natural map
\[
H^{j}_{\m}(\Omega^i_{X,h})\to \mathbb{H}^j_{\m}(\underline{\Omega}^i_{X})
\]
is surjective for all $j\geq 0$.
\end{lemma}
\begin{proof}
We first prove that 
\begin{equation} \label{eq:injofdelta1}
\mathbb{H}^j_{\m}(\Omega_{X,h}^{\leq i-1})\to \mathbb{H}^j_{\m}(\underline{\Omega}_{X}^{\leq i-1})
\end{equation}
is an isomorphism for all $j\in \Z$.
Since $X$ is pre-$(i-1)$-Du Bois, we have the following diagram:
\[ 
\begin{tikzcd}
    \Omega_{X,h}^{i-1}[-(i-1)]\arrow[r]\arrow[d,equal] &\Omega_{X,h}^{\leq i-1} \arrow[r]\arrow[d]& \Omega_{X,h}^{\leq i-2}\arrow[r,"+1"]\arrow[d] &  {} \\
\underline{\Omega}_{X}^{i-1}[-(i-1)]\arrow[r]&\underline{\Omega}_{X}^{\leq i-1} \arrow[r]& \underline{\Omega}_{X}^{\leq i-2}\arrow[r,"+1"] &{},
\end{tikzcd}
\]
where the left vertical map is an isomorphism since $X$ is pre-$(i-1)$-Du Bois by assumption.
Then by the five lemma, the assertion can be reduced to proving that
\[
\mathbb{H}^j_{\m}(\Omega_{X,h}^{\leq i-2})\to \mathbb{H}^j_{\m}(\underline{\Omega}_{X}^{\leq i-2})
\]
is an isomorphism for all $j\geq 0$.
By repeating this procedure, we conclude the proof of (\ref{eq:injofdelta1}).

Now, we finish the proof of the lemma.
We have the following diagram
\[
\begin{tikzcd}
    \Omega^i_{X,h}[-i]\arrow[r]\arrow[d] &\Omega_{X,h}^{\leq i} \arrow[r]\arrow[d]& \Omega_{X,h}^{\leq i-1}\arrow[r,"+1"]\arrow[d] & {}  \\
{\underline{\Omega}^i_{X}[-i]} \arrow[r]\arrow[d] &\underline{\Omega}_{X}^{\leq i}\arrow[r]\arrow[d]& \underline{\Omega}_{X}^{\leq i-1}\arrow[r,"+1"]\arrow[d] & {}\\
{\tau^{>0}\underline{\Omega}^i_X}[-i]\arrow[r]\arrow[d,"+1"] & {C_1}\arrow[r]\arrow[d,"+1"] & {C_2}\arrow[r,"+1"]\arrow[d,"+1"]& {}\\
{} & {} & {} & {}
\end{tikzcd}
\]
and thus 
\[
\begin{tikzcd}
  \mathbb{H}^{i+j-1}_{\m}(\Omega_{X,h}^{\leq i-1})\arrow[r]\arrow{d}{\cong} \arrow[r]& H^{j}_{\m}(\Omega^i_{X,h})\arrow[r]\arrow[d] &\mathbb{H}^{i+j}_{\m}(\Omega_{X,h}^{\leq i}) \arrow[d]  \\
\mathbb{H}^{i+j-1}_{\m}(\underline{\Omega}_{X}^{\leq i-1})\arrow[r]\arrow[d]\arrow[r]& \mathbb{H}^{j}_{\m}(\underline{\Omega}^i_{X})\arrow[r]\arrow[d] &\mathbb{H}^{i+j}_{\m}(\underline{\Omega}_{X}^{\leq i}) \arrow[d]\\
\mathllap{0=\ }\mathbb{H}^{i+j-1}_{\m}({ C_2})\arrow[r]\arrow[d]\arrow[r]& \mathbb{H}^{j}_{\m}(\tau^{>0}\underline{\Omega}^i_X)\arrow[r,hookrightarrow]\arrow[d,"\delta"] & \mathbb{H}^{i+j}_{\m}({C_1})\arrow[d,"\text{Lem}\,\,\ref{lem:inj of delta'}",hookrightarrow]\\
\mathbb{H}^{i+j}_{\m}(\Omega_{X,h}^{\leq i-1})\arrow[r]\arrow{d}{\cong} \arrow[r]& H^{j+1}_{\m}(\Omega^i_{X,h})\arrow[r]\arrow[d] &\mathbb{H}^{i+j+1}_{\m}(\Omega_{X,h}^{\leq i}) \arrow[d]  \\
\mathbb{H}^{j+1}_{\m}(\underline{\Omega}_{X}^{\leq i-1})\arrow[r]\arrow[r]& \mathbb{H}_\m^{j+1}(\underline{\Omega}^i_{X})\arrow[r] &\mathbb{H}^{i+j+1}_{\m}(\underline{\Omega}_{X}^{\leq i})
\end{tikzcd}
\]
for all $j\geq 0$.
Thus, the map $\delta$ in the above diagram is injective, which shows the desired surjectivity.
\end{proof}

\begin{proof}[Proof of Theorem \ref{thm:pseudo-pre-k-Du Bois=pre-k-Du Bois in char 0}]
    We only prove the `only if' part since the converse is clear.
    We may assume that $X=\Spec R$ is affine.
    We argue by induction on $0 \leq i \leq k$. Assume that $X$ is 
    \begin{enumerate}
        \item pseudo-pre-$i$-Du Bois and, 
        \item pre-$(i-1)$-Du Bois\footnote{with this assumption being empty if $i=0$}. 
    \end{enumerate}
    Our goal is to show that $X$ is pre-$i$-Du Bois.

    Suppose by contradiction that $X$ is not pre-$i$-Du Bois.
    Fix the maximum ideal $\m$ corresponding to the closed point $x\in X$.
    Consider the exact triangle
    \[
    \Omega^i_{X,h}\to \underline{\Omega}^i_X\to \tau^{>0}\underline{\Omega}^i_X\xrightarrow{+1}.
    \]
    Since $X$ is pseudo-pre-$k$-Du Bois,
    \[
    H^j_{\m}(\Omega^i_{X,h})\to \mathbb{H}^j_{\m}(\underline{\Omega}^i_X)
    \]
    is injective for all $j\leq d-i$ by definition.
    Together with Lemma \ref{lem:inj of delta}, this implies that
    \[
    H^j_{\m}(\Omega^i_{X,h})\to \mathbb{H}^j_{\m}(\underline{\Omega}^i_X)
    \]
    is an isomorphism for all $j\leq d-i$.
    Thus $\mathbb{H}^j_{\m}(\tau^{>0}\underline{\Omega}^i_X)=0$ for all $j\leq d-i-1$.
    Since $\tau^{>0}\underline{\Omega}^i_{X}$ is supported at $\{x\}$,   
we have 
\[
\mathcal{H}^j(\tau^{>0}\underline{\Omega}^i_{X})=\mathbb{H}^j(\tau^{>0}\underline{\Omega}^i_{X})=\mathbb{H}^j_{\m}(\tau^{>0}\underline{\Omega}^i_{X})=0
    \]
    for all $j\leq d-i-1$.
    By the exact triangle
    \[
    \Omega^i_{X,h}\to \underline{\Omega}^i_{X}\to \tau^{>0}\underline{\Omega}^i_{X}\xrightarrow{+1}
    \]
    again, we conclude that $\mathcal{H}^j(\underline{\Omega}^i_{X})=0$ for all $0< j\leq d-i-1$.
    Finally, we have $\mathcal{H}^j(\underline{\Omega}^i_X)=0$
    for all $i+j\geq d$ by \cite[Theorem 7.29 (b)]{Peter-Steenbrink(Book)} and \cite[Lemma 2.5]{Friedman-Laza2}.
    Therefore, we obtain $\mathcal{H}^j(\underline{\Omega}^i_{X})=0$ for all $j>0$.
\end{proof}

\section{Appendix: differential forms on simple normal crossing schemes}
Let $k$ be a field.
Set $A\coloneqq k[x_1,\ldots,x_n]$, 
$E=\{x_1x_2\cdots x_n=0\}\subset \mathbb{A}^n=\Spec\,k[x_1,\ldots,x_n]$, $E_1=\{x_1=0\}$,  and $E_1^c=\{x_2\cdots x_n=0\}$.
Moreover, we introduce the following notation.
\begin{enumerate}
    \item For $f\in A=k[x_1,\ldots,x_n]$, we denote $f$ modulo $x_1$ by $\overline{f}$. 
    \item For $1\leq i\leq n$, we set $\Sigma_i \coloneqq \{\bolda=(a_1, \dots, a_i) \mid 1 \leq a_1 < \cdots < a_i \leq n\}$.
    \item For $1\leq i\leq n$, we set $\Pi_i \coloneqq \{\bolda=(a_1, \dots, a_i) \mid 2 \leq a_1 < \cdots < a_i \leq n\}$.
    \item For $\bolda\in \Pi_i$ for some $i$, we denote $\{2,\ldots,n\}\setminus \bolda$ by $\bolda^c$.  
\end{enumerate}
Then the following statements hold:
\begin{enumerate}
    \item $\Omega^i_{A}=\bigoplus_{\bolda\in \Sigma_i} Ad\boldx_{\bolda}$
    \item $\begin{aligned}[t]
        \Omega^i_E/{\rm tors}&=\Omega^i_{A}\Big/\Big(\bigoplus_{\bolda \in \Sigma_i} {A\boldx_{\{1,2,\ldots,n\}\setminus\bolda}}d\boldx_{\bolda}\Big)\\
        &=\Omega^i_{A}\Big/\Big((\bigoplus_{\bolda \in \Pi_{i-1}} {A\boldx_{\bolda^c}}dx_1\wedge d\boldx_{\bolda})\oplus(\bigoplus_{\boldb \in \Pi_i} {Ax_1\boldx_{\boldb^c}}d\boldx_{\boldb})\Big)
    \end{aligned}$
    \item $\Omega^i_{E_1^c}/{\rm tors}=\Omega^i_{A}\Big/\Big((\bigoplus_{\bolda \in \Pi_{i-1}} { A\boldx_{\bolda^c}}dx_1\wedge d\boldx_{\bolda})\oplus(\bigoplus_{\boldb \in \Pi_i} { A\boldx_{\boldb^c}}d\boldx_{\boldb})\Big)$
    \item $\Omega^i_{E_1}=\oplus_{\bolda\in \Pi_{i}} \overline{A}d\overline{\boldx}_{\bolda}$
    \item $\Omega^i_{E_1^c}|_{E_1}/{\rm tors}=\Omega^i_{\overline{A}}/(\bigoplus_{\bolda \in \Pi_i} {\overline{A}\overline{\boldx}_{\bolda^c}}d\overline{\boldx}_{\bolda})$
\end{enumerate}

\begin{theorem}\label{thm:torsion exact sequence}
    There exists a short exact sequence
    \[
    0\to \Omega^i_E/{\rm tors} \xrightarrow{\phi} \Omega^i_{E_1^c}/{\rm tors} \oplus \Omega^i_{E_1} \xrightarrow{\psi}\Omega^i_{E_1^c}|_{E_1}/{\rm tors} \to 0,
    \]
    where the map $\phi$ is defined by \[\sum_{\bolda\in\Sigma_i} f_{\bolda}d\boldx_{\bolda} \mapsto (\sum_{\bolda\in\Sigma_i} f_{\bolda}d\boldx_{\bolda}, \sum_{\bolda\in\Pi_i} \overline{f}_{\bolda}d\overline{\boldx}_{\bolda})\]
 and the map $\psi$ is defined by \[(\sum_{\bolda\in\Sigma_i} f_{\bolda}d\boldx_{\bolda}, \sum_{\bolda\in\Pi_i} \overline{g}_{\bolda}d\overline{\boldx}_{\bolda})\mapsto \sum_{\bolda\in\Pi_i} (\overline{f}_{\bolda}-\overline{g}_{\bolda})d\overline{\boldx}_{\bolda}.\]
\end{theorem}
\begin{proof}
We divide the proof into the following claims:
\begin{claim}\label{cl:phi is inj}
The map $\phi$ is injective.
\end{claim}
\begin{proof}[Proof of Claim \ref{cl:phi is inj}]
    Take an element in $\ker(\phi)$.
    Note that it can be uniquely written as
    \[
    \sum_{\bolda\in\Pi_{i-1}}f_{\bolda}dx_1\wedge d\boldx_{\bolda}
    +\sum_{\boldb\in\Pi_{i}}f_{\boldb}d\boldx_{\boldb}
    \]
    for some $f_{\bolda}\in A$ and $f_{\boldb}\in A$.
 Since this element maps to zero in $\Omega^i_{E_1^c}/{\rm tors}$, 
    we have that 
    \[f_{\bolda}=x_{\bolda^c}g_{\bolda}\] for $\bolda\in\Pi_{i-1}$ (here $\bolda^c=\{2,\ldots,n\}\setminus\bolda$), and 
    \[f_{\boldb}=x_{\boldb^c}h_{\boldb}\] for $\boldb\in\Pi_{i}$ (here $\boldb^c=\{2,\ldots,n\}\setminus\boldb$).
    Recall that
    \[
    \Omega^i_E/{\rm tors}=\Omega^i_{A}/\big((\bigoplus_{\bolda \in \Pi_{i-1}}A\boldx_{\bolda^c} dx_1\wedge d\boldx_{\bolda})\oplus(\bigoplus_{\boldb \in \Pi_i} Ax_1\boldx_{\boldb^c}d\boldx_{\boldb})\big).
    \]
    Thus,
    \begin{align*}
        \sum_{\bolda\in\Pi_{i-1}}f_{\bolda}dx_1\wedge d\boldx_{\bolda}
    +\sum_{\boldb\in\Pi_{i}}f_{\boldb}d\boldx_{\boldb}&=
    \sum_{\bolda\in\Pi_{i-1}}x_{\bolda^c}g_{\bolda}dx_1\wedge d\boldx_{\bolda}
+\sum_{\boldb\in\Pi_{i}}x_{\bolda^c}h_{\boldb}d\boldx_{\boldb}\\
&= \sum_{\boldb\in\Pi_{i}}x_{\boldb^c}h_{\boldb}d\boldx_{\boldb} \in \Omega^i_{E}/{\rm tors}
    \end{align*}
    in view of $\sum_{\bolda\in\Pi_{i-1}}x_{\bolda^c}g_{\bolda}dx_1\wedge d\boldx_{\bolda}\in \bigoplus_{\bolda \in \Pi_{i-1}} A\boldx_{\bolda^c}dx_1\wedge d\boldx_{\bolda}$.
    
    Since $\sum_{\boldb\in\Pi_{i}}x_{\bolda^c}h_{\boldb}d\boldx_{\boldb}\in \Omega^i_{E}/{\rm tors}$ maps to zero in $\Omega^i_{E_1}$,
    there exists $e_{\boldb}\in k[x_1,\ldots,x_n]$ such that
    $h_{\boldb}=x_1e_{\boldb}$ for all $\boldb\in\Pi_{i}$.
    Thus,
    \[\sum_{\boldb\in\Pi_{i}}x_{\boldb^c}h_{\boldb}d\boldx_{\boldb}=\sum_{\boldb\in\Pi_{i}}x_{\boldb^c}x_1e_{\boldb}d\boldx_{\boldb}=0 \in \Omega^i_{E}/{\rm tors},
    \]
    and we conclude.
    \end{proof}

\begin{claim}\label{cl:psi is surj}
    The map $\psi$ is surjective.
\end{claim}
\begin{proof}[Proof of Claim \ref{cl:psi is surj}]
    Since $\Omega^i_{E_1}\to \Omega^i_{E_1^c|_{E_1}}/{\rm tors}$ is surjective, the assertion holds.
\end{proof}

\begin{claim}\label{cl:Im is contained in Ker}
    $\psi\circ\phi=0$.
\end{claim}
\begin{proof}[Proof of Claim \ref{cl:Im is contained in Ker}]
    This immediately follows from by the definitions of $\phi$ and $\psi$.
\end{proof}

\begin{claim}\label{cl:Ker is contained in Im}
    $\mathrm{ker}(\psi)\subset \mathrm{im}(\phi)$.
\end{claim}
\begin{proof}[Proof of Claim \ref{cl:Ker is contained in Im}]
    Recall that any element $\sum_{\bolda\in\Sigma_i} f_{\bolda}d\boldx_{\bolda}\in \Omega^i_{\mathbb{A}^n}$ can be uniquely written as
    \[
    \sum_{\bolda\in\Pi_{i-1}}f_{\bolda}dx_1\wedge d\boldx_{\bolda}+\sum_{\boldb\in\Pi_{i}}f_{\boldb}d\boldx_{\boldb}
    \]
    for some $f_{\bolda}\in A$ and $f_{\boldb}\in A$.
     Take any element \[(\sum_{\bolda\in\Pi_{i-1}}f_{\bolda}dx_1\wedge d\boldx_{\bolda}+\sum_{\boldb\in\Pi_{i}}f_{\boldb}d\boldx_{\boldb}, \sum_{\boldb\in\Pi_i} \overline{g}_{\boldb}d\overline{\boldx}_{\boldb})\in \ker(\psi).\] 
    Then
    $\sum_{\boldb\in\Pi_i} (\overline{f}_{\boldb}-\overline{g}_{\boldb})d\overline{\boldx}_{\boldb}=0$.

    Recall that $\Omega^i_{E_1^c}|_{E_1}/{\rm tors}=\Omega^i_{\mathbb{A}^n}/(\bigoplus_{\boldb \in \Pi_i} \overline{A}\overline{\boldx}_{\boldb^c}d\overline{\boldx}_{\boldb})$.
    Thus, for any $\boldb\in\Pi_i$, there exists $h\in k[x_1,\ldots,x_n]$ such that
    \[\overline{f}_{\boldb}=\overline{g}_{\boldb}+\overline{h}\overline{\boldx}_{\boldb^c}, 
    \]
    where we recall that $\boldb^c=\{2,\ldots,n\}\setminus \boldb$.
    In particular, for any $\boldb\in\Pi_i$, there exists $e_{\boldb}\in k[x_1,\ldots,x_n]$ such that
    \[
    f_{\boldb}=g_{\boldb}+h\boldx_{\boldb^c}+e_{\boldb}x_1. 
    \]
    Then, for any $\boldb\in\Pi_i$, 
    \[
    f_{\boldb}d\boldx_{\boldb}=(g_{\boldb}+h\boldx_{\boldb^c}+e_{\boldb}x_1)d\boldx_{\boldb}=(g_{\boldb}+e_{\boldb}x_1)d\boldx_{\boldb} \in \Omega^i_{E_1^c}/{\rm tors}, 
    \]
    because $h\boldx_{\boldb^c}d\boldx_{\boldb}\in \bigoplus_{\boldb \in \Pi_i} A\boldx_{\boldb^c}d\boldx_{\boldb}$.
    Take
    \[
    \sum_{\bolda\in\Pi_{i-1}}f_{\bolda}dx_1\wedge d\boldx_{\bolda}+\sum_{\boldb\in\Pi_i}(g_{\boldb}+e_{\boldb}x_1)d\boldx_{\boldb}
\in  \Omega^i_E/{\rm tors}.   \]
Then\begin{align*}
    &\phi \Big(\sum_{\bolda\in\Pi_{i-1}}f_{\bolda}dx_1\wedge d\boldx_{\bolda}+\sum_{\boldb\in\Pi_i}(g_{\boldb}+e_{\boldb}x_1)d\boldx_{\boldb}
\Big)\\&=\Big(\sum_{\bolda\in\Pi_{i-1}}f_{\bolda}dx_1\wedge d\boldx_{\bolda}+\sum_{\boldb\in\Pi_i}(g_{\boldb}+e_{\boldb}x_1)d\boldx_{\boldb}\in \Omega^i_{E_1^c}/{\rm tors},\sum_{\boldb\in\Pi_i}\overline{g}_{\boldb}d\overline{\boldx}_{\boldb}\in \Omega^i_{E_1}\Big)\\
&=\Big(\sum_{\bolda\in\Pi_{i-1}}f_{\bolda}dx_1\wedge d\boldx_{\bolda}+\sum_{\boldb\in\Pi_i}f_{\boldb}d\boldx_{\boldb}\in \Omega^i_{E_1^c}/{\rm tors} ,\sum_{\bolda\in\Pi_i}\overline{g}_{\boldb}d\overline{\boldx}_{\boldb}\in \Omega^i_{E_1}\Big), 
\end{align*}
which concludes the proof.
\end{proof}
\end{proof}

\input{main.bbl}

% \bibliographystyle{skalpha}
% \bibliography{bibliography.bib}

\end{document}

%% file: main.bbl

%% file: main.bbl
\begin{bibdiv}
\begin{biblist}

\bib{ABL}{article}{
      author={Arvidsson, Emelie},
      author={Bernasconi, Fabio},
      author={Lacini, Justin},
       title={On the {K}awamata-{V}iehweg vanishing theorem for log del {P}ezzo surfaces in positive characteristic},
        date={2022},
        ISSN={0010-437X},
     journal={Compos. Math.},
      volume={158},
      number={4},
       pages={750\ndash 763},
         url={https://doi.org/10.1112/S0010437X22007394},
      review={\MR{4438290}},
}

\bib{AWZ}{article}{
      author={Achinger, Piotr},
      author={Witaszek, Jakub},
      author={Zdanowicz, Maciej},
       title={Global {F}robenius liftability {I}},
        date={2021},
        ISSN={1435-9855,1435-9863},
     journal={J. Eur. Math. Soc. (JEMS)},
      volume={23},
      number={8},
       pages={2601\ndash 2648},
         url={https://doi.org/10.4171/jems/1063},
      review={\MR{4269423}},
}

\bib{AWZ2}{article}{
      author={Achinger, Piotr},
      author={Witaszek, Jakub},
      author={Zdanowicz, Maciej},
       title={Global {F}robenius liftability {II}: surfaces and {F}ano threefolds},
        date={2023},
        ISSN={0391-173X},
     journal={Ann. Sc. Norm. Super. Pisa Cl. Sci. (5)},
      volume={24},
      number={1},
       pages={329\ndash 366},
         url={https://doi-org.kyoto-u.idm.oclc.org/10.2422/2036-2145.202005_003},
      review={\MR{4587749}},
}

\bib{Bhatt_2012}{article}{
      author={Bhatt, Bhargav},
       title={Derived splinters in positive characteristic},
        date={2012},
        ISSN={0010-437X,1570-5846},
     journal={Compos. Math.},
      volume={148},
      number={6},
       pages={1757\ndash 1786},
         url={https://doi.org/10.1112/S0010437X12000309},
      review={\MR{2999303}},
}

\bib{fbook}{book}{
      author={Brion, Michel},
      author={Kumar, Shrawan},
       title={Frobenius splitting methods in geometry and representation theory},
      series={Progress in Mathematics},
   publisher={Birkh\"{a}user Boston, Inc., Boston, MA},
        date={2005},
      volume={231},
        ISBN={0-8176-4191-2},
      review={\MR{2107324}},
}

\bib{borger2009lambdaringsfieldelement}{article}{
      author={Borger, James},
       title={Lambda-rings and the field with one element},
        date={2009},
     journal={arXiv:0906.3146},
         url={https://arxiv.org/abs/0906.3146},
}

\bib{Bhatt-Schwede-Takagi}{incollection}{
      author={Bhatt, Bhargav},
      author={Schwede, Karl},
      author={Takagi, Shunsuke},
       title={The weak ordinarity conjecture and {$F$}-singularities},
        date={2017},
   booktitle={Higher dimensional algebraic geometry---in honour of {P}rofessor {Y}ujiro {K}awamata's sixtieth birthday},
      series={Adv. Stud. Pure Math.},
      volume={74},
   publisher={Math. Soc. Japan, Tokyo},
       pages={11\ndash 39},
         url={https://doi.org/10.2969/aspm/07410011},
      review={\MR{3791207}},
}

\bib{DuBois}{article}{
      author={Du~Bois, Philippe},
       title={Complexe de de {R}ham filtr\'e{} d'une vari\'et\'e{} singuli\`ere},
        date={1981},
        ISSN={0037-9484},
     journal={Bull. Soc. Math. France},
      volume={109},
      number={1},
       pages={41\ndash 81},
         url={http://www.numdam.org/item?id=BSMF_1981__109__41_0},
      review={\MR{613848}},
}

\bib{BM23}{article}{
      author={Dirks, Bradley},
      author={Musta\c~t\u a, Mircea},
       title={Minimal exponents of hyperplane sections: a conjecture of {T}eissier},
        date={2023},
     journal={J. Eur. Math. Soc. (JEMS)},
      volume={25},
      number={12},
       pages={4813\ndash 4840},
}

\bib{EV92}{book}{
      author={Esnault, H\'{e}l\`ene},
      author={Viehweg, Eckart},
       title={Lectures on vanishing theorems},
      series={DMV Seminar},
   publisher={Birkh\"{a}user Verlag, Basel},
        date={1992},
      volume={20},
        ISBN={3-7643-2822-3},
         url={https://doi.org/10.1007/978-3-0348-8600-0},
      review={\MR{1193913}},
}

\bib{Friedman-Laza2}{article}{
      author={Friedman, Robert},
      author={Laza, Radu},
       title={The higher {D}u {B}ois and higher rational properties for isolated singularities},
        date={2024},
        ISSN={1056-3911,1534-7486},
     journal={J. Algebraic Geom.},
      volume={33},
      number={3},
       pages={493\ndash 520},
      review={\MR{4739666}},
}

\bib{Friedman-Laza1}{article}{
      author={Friedman, Robert},
      author={Laza, Radu},
       title={Higher {D}u {B}ois and higher rational singularities},
        date={2024},
        ISSN={0012-7094,1547-7398},
     journal={Duke Math. J.},
      volume={173},
      number={10},
       pages={1839\ndash 1881},
         url={https://doi.org/10.1215/00127094-2023-0051},
        note={Appendix by Morihiko Saito},
      review={\MR{4776417}},
}

\bib{Geisser}{article}{
      author={Geisser, Thomas},
       title={Arithmetic cohomology over finite fields and special values of {$\zeta$}-functions},
        date={2006},
        ISSN={0012-7094,1547-7398},
     journal={Duke Math. J.},
      volume={133},
      number={1},
       pages={27\ndash 57},
         url={https://doi.org/10.1215/S0012-7094-06-13312-4},
      review={\MR{2219269}},
}

\bib{GLS}{book}{
      author={Greuel, G.-M.},
      author={Lossen, C.},
      author={Shustin, E.},
       title={Introduction to singularities and deformations},
      series={Springer Monographs in Mathematics},
   publisher={Springer, Berlin},
        date={2007},
        ISBN={978-3-540-28380-5; 3-540-28380-3},
      review={\MR{2290112}},
}

\bib{GNPP}{book}{
      author={Guill\'{e}n, F.},
      author={Navarro~Aznar, V.},
      author={Pascual~Gainza, P.},
      author={Puerta, F.},
       title={Hyperr\'{e}solutions cubiques et descente cohomologique},
      series={Lecture Notes in Mathematics},
   publisher={Springer-Verlag, Berlin},
        date={1988},
      volume={1335},
        ISBN={3-540-50023-5},
         url={https://doi-org.kyoto-u.idm.oclc.org/10.1007/BFb0085054},
        note={Papers from the Seminar on Hodge-Deligne Theory held in Barcelona, 1982},
      review={\MR{972983}},
}

\bib{Graf15}{article}{
      author={Graf, Patrick},
       title={The generalized {L}ipman-{Z}ariski problem},
        date={2015},
        ISSN={0025-5831,1432-1807},
     journal={Math. Ann.},
      volume={362},
      number={1-2},
       pages={241\ndash 264},
         url={https://doi.org/10.1007/s00208-014-1112-9},
      review={\MR{3343876}},
}

\bib{Ulrich-Torsten}{book}{
      author={G\"ortz, Ulrich},
      author={Wedhorn, Torsten},
       title={Algebraic geometry {I}. {S}chemes---with examples and exercises},
     edition={Second},
      series={Springer Studium Mathematik---Master},
   publisher={Springer Spektrum, Wiesbaden},
        date={2020},
        ISBN={978-3-658-30732-5; 978-3-658-30733-2},
         url={https://doi.org/10.1007/978-3-658-30733-2},
      review={\MR{4225278}},
}

\bib{Har}{book}{
      author={Hartshorne, Robin},
       title={Algebraic geometry},
   publisher={Springer-Verlag, New York-Heidelberg},
        date={1977},
        ISBN={0-387-90244-9},
        note={Graduate Texts in Mathematics, No. 52},
      review={\MR{0463157}},
}

\bib{Har80}{article}{
      author={Hartshorne, Robin},
       title={Stable reflexive sheaves},
        date={1980},
        ISSN={0025-5831},
     journal={Math. Ann.},
      volume={254},
      number={2},
       pages={121\ndash 176},
         url={https://doi.org/10.1007/BF01467074},
      review={\MR{597077}},
}

\bib{Hara98}{article}{
      author={Hara, Nobuo},
       title={A characterization of rational singularities in terms of injectivity of {F}robenius maps},
        date={1998},
        ISSN={0002-9327},
     journal={Amer. J. Math.},
      volume={120},
      number={5},
       pages={981\ndash 996},
         url={http://muse.jhu.edu/journals/american_journal_of_mathematics/v120/120.5hara.pdf},
      review={\MR{1646049}},
}

\bib{Huber-Jorder}{article}{
      author={Huber, Annette},
      author={J\"{o}rder, Clemens},
       title={Differential forms in the h-topology},
        date={2014},
        ISSN={2313-1691,2214-2584},
     journal={Algebr. Geom.},
      volume={1},
      number={4},
       pages={449\ndash 478},
         url={https://doi.org/10.14231/AG-2014-020},
      review={\MR{3272910}},
}

\bib{Huber-Kelly}{article}{
      author={Huber, Annette},
      author={Kelly, Shane},
       title={Differential forms in positive characteristic, {II}: cdh-descent via functorial {R}iemann-{Z}ariski spaces},
        date={2018},
        ISSN={1937-0652,1944-7833},
     journal={Algebra Number Theory},
      volume={12},
      number={3},
       pages={649\ndash 692},
         url={https://doi.org/10.2140/ant.2018.12.649},
      review={\MR{3815309}},
}

\bib{Huber-Kebekus-Kelly}{article}{
      author={Huber, Annette},
      author={Kebekus, Stefan},
      author={Kelly, Shane},
       title={Differential forms in positive characteristic avoiding resolution of singularities},
        date={2017},
        ISSN={0037-9484,2102-622X},
     journal={Bull. Soc. Math. France},
      volume={145},
      number={2},
       pages={305\ndash 343},
         url={https://doi.org/10.24033/bsmf.2739},
      review={\MR{3749788}},
}

\bib{JKSY22}{article}{
      author={Jung, Seung-Jo},
      author={Kim, In-Kyun},
      author={Saito, Morihiko},
      author={Yoon, Youngho},
       title={Higher {D}u {B}ois singularities of hypersurfaces},
        date={2022},
        ISSN={0024-6115,1460-244X},
     journal={Proc. Lond. Math. Soc. (3)},
      volume={125},
      number={3},
       pages={543\ndash 567},
         url={https://doi.org/10.1112/plms.12464},
      review={\MR{4480883}},
}

\bib{JS23}{article}{
      author={Jeffries, Jack},
      author={Singh, Anurag~K.},
       title={Differential operators on classical invariant rings do not lift modulo {$p$}},
        date={2023},
        ISSN={0001-8708,1090-2082},
     journal={Adv. Math.},
      volume={432},
       pages={Paper No. 109276, 53},
         url={https://doi.org/10.1016/j.aim.2023.109276},
      review={\MR{4634904}},
}

\bib{Kat70}{article}{
      author={Katz, Nicholas~M.},
       title={Nilpotent connections and the monodromy theorem: {A}pplications of a result of {T}urrittin},
        date={1970},
        ISSN={0073-8301},
     journal={Inst. Hautes \'{E}tudes Sci. Publ. Math.},
      number={39},
       pages={175\ndash 232},
         url={http://www.numdam.org/item?id=PMIHES_1970__39__175_0},
      review={\MR{291177}},
}

\bib{Kaw4}{article}{
      author={Kawakami, Tatsuro},
       title={Extendability of differential forms via {C}artier operators},
        date={2022},
     journal={arXiv preprint arXiv:2207.13967, to appear in J.~Eur.~Math.~Soc.~(JEMS)},
}

\bib{KM98}{book}{
      author={Koll\'{a}r, J\'{a}nos},
      author={Mori, Shigefumi},
       title={Birational geometry of algebraic varieties},
      series={Cambridge Tracts in Mathematics},
   publisher={Cambridge University Press, Cambridge},
        date={1998},
      volume={134},
        ISBN={0-521-63277-3},
         url={https://doi.org/10.1017/CBO9780511662560},
        note={With the collaboration of C. H. Clemens and A. Corti, Translated from the 1998 Japanese original},
      review={\MR{1658959}},
}

\bib{KS21}{article}{
      author={Kebekus, Stefan},
      author={Schnell, Christian},
       title={Extending holomorphic forms from the regular locus of a complex space to a resolution of singularities},
        date={2021},
        ISSN={0894-0347},
     journal={J. Amer. Math. Soc.},
      volume={34},
      number={2},
       pages={315\ndash 368},
         url={https://doi.org/10.1090/jams/962},
      review={\MR{4280862}},
}

\bib{Kovacs-Schwede-Smith}{article}{
      author={Kov\'acs, S\'andor~J.},
      author={Schwede, Karl},
      author={Smith, Karen~E.},
       title={The canonical sheaf of {D}u {B}ois singularities},
        date={2010},
        ISSN={0001-8708,1090-2082},
     journal={Adv. Math.},
      volume={224},
      number={4},
       pages={1618\ndash 1640},
         url={https://doi.org/10.1016/j.aim.2010.01.020},
      review={\MR{2646306}},
}

\bib{Kawakami-Takamatsu}{article}{
      author={Kawakami, Tatsuro},
      author={Takamatsu, Teppei},
       title={On {F}robenius liftability of surface singularities},
        date={2024},
     journal={arXiv:2402.08152},
         url={https://arxiv.org/abs/2402.08152},
}

\bib{KTTWYY1}{article}{
      author={Kawakami, Tatsuro},
      author={Takamatsu, Teppei},
      author={Tanaka, Hiromu},
      author={Witaszek, Jakub},
      author={Yobuko, Fuetaro},
      author={Yoshikawa, Shou},
       title={Quasi-${F}$-splittings in birational geometry},
        date={2022},
     journal={arXiv preprint arXiv:2208.08016, to appear in Ann. Sci.~\'Ec.~Norm.~Sup\'er.},
}

\bib{Mustata-Popa22}{article}{
      author={Musta\c{t}\u{a}, Mircea},
      author={Popa, Mihnea},
       title={Hodge filtration on local cohomology, {D}u {B}ois complex and local cohomological dimension},
        date={2022},
     journal={Forum Math. Pi},
      volume={10},
       pages={Paper No. e22, 58},
         url={https://doi-org.kyoto-u.idm.oclc.org/10.1017/fmp.2022.15},
      review={\MR{4491455}},
}

\bib{MOPW23}{article}{
      author={Musta\c~t\u a, Mircea},
      author={Olano, Sebasti\'an},
      author={Popa, Mihnea},
      author={Witaszek, Jakub},
       title={The {D}u {B}ois complex of a hypersurface and the minimal exponent},
        date={2023},
     journal={Duke Math. J.},
      volume={172},
      number={7},
       pages={1411\ndash 1436},
}

\bib{MP19}{article}{
      author={Musta\c~t\u a, Mircea},
      author={Popa, Mihnea},
       title={Hodge ideals},
        date={2019},
        ISSN={0065-9266,1947-6221},
     journal={Mem. Amer. Math. Soc.},
      volume={262},
      number={1268},
       pages={v+80},
}

\bib{Peter-Steenbrink(Book)}{book}{
      author={Peters, Chris A.~M.},
      author={Steenbrink, Joseph H.~M.},
       title={Mixed {H}odge structures},
      series={Ergebnisse der Mathematik und ihrer Grenzgebiete. 3. Folge. A Series of Modern Surveys in Mathematics [Results in Mathematics and Related Areas. 3rd Series. A Series of Modern Surveys in Mathematics]},
   publisher={Springer-Verlag, Berlin},
        date={2008},
      volume={52},
        ISBN={978-3-540-77015-2},
      review={\MR{2393625}},
}

\bib{Reid83}{incollection}{
      author={Reid, Miles},
       title={Minimal models of canonical {$3$}-folds},
        date={1983},
   booktitle={Algebraic varieties and analytic varieties ({T}okyo, 1981)},
      series={Adv. Stud. Pure Math.},
      volume={1},
   publisher={North-Holland, Amsterdam},
       pages={131\ndash 180},
}

\bib{Schwede}{article}{
      author={Schwede, Karl},
       title={{$F$}-injective singularities are {D}u {B}ois},
        date={2009},
        ISSN={0002-9327},
     journal={Amer. J. Math.},
      volume={131},
      number={2},
       pages={445\ndash 473},
         url={https://doi.org/10.1353/ajm.0.0049},
      review={\MR{2503989}},
}

\bib{stacks-project}{misc}{
      author={{Stacks Project Authors}, The},
       title={\itshape {S}tacks {P}roject},
         how={\url{http://stacks.math.columbia.edu}},
        date={2024},
}

\bib{SVV}{article}{
      author={Shen, Wanchun},
      author={Venkatesh, Sridhar},
      author={Vo, Anh~Duc},
       title={On $ k $-du bois and $ k $-rational singularities},
        date={2023},
     journal={arXiv preprint arXiv:2306.03977},
}

\bib{Zda18}{article}{
      author={Zdanowicz, Maciej},
       title={Liftability of singularities and their {F}robenius morphism modulo {$p^2$}},
        date={2018},
        ISSN={1073-7928,1687-0247},
     journal={Int. Math. Res. Not. IMRN},
      number={14},
       pages={4513\ndash 4577},
         url={https://doi.org/10.1093/imrn/rnw297},
      review={\MR{3830576}},
}

\end{biblist}
\end{bibdiv}
